\documentclass[a4paper,12pt]{article}
\usepackage{amsmath}
\usepackage{amssymb}
\usepackage{setspace}
\usepackage{fullpage}
\usepackage{bbm}
\usepackage{amsthm}
\usepackage[latin1]{inputenc}
\usepackage{xcolor}
\usepackage{hyperref}
\usepackage{enumerate}
\usepackage[normalem]{ulem}
\usepackage{ifthen}

\newtheorem{theorem}{Theorem}[section]
\newtheorem{lemma}[theorem]{Lemma}
\newtheorem*{rtheorem}{Main Theorem}

\newtheorem{result}[theorem]{Result}
\theoremstyle{definition}

\newtheorem{definition}[theorem]{Definition}

\newtheorem{remark}[theorem]{Remark}

\usepackage{pdfsync}
\def\PG{\mathrm{PG}} \def\AG{\mathrm{AG}}

\def\PGammaL{\mathrm{P}\Gamma\mathrm{L}}

\def\PGL{\mathrm{PGL}}

  \def\C{\mathcal{C}}
\def\D{\mathcal{D}}

 \def\S{\mathcal{S}}

\def\F{\mathbb{F}}

\def\la{\lambda}
\def\a{\alpha}
\DeclareMathOperator{\rk}{rk}

\newcommand{\comments}[1]{}

\newcommand{\versie}{arxiv} 

\title{The weight distributions of linear sets in $\PG(1,q^5)$}

\author{Maarten De Boeck \thanks{University of Rijeka, Faculty of Mathematics, Radmile Matej\v{c}i\'c 2, 51000 Rijeka, Croatia\newline Ghent University, Department of Mathematics: Algebra and Geometry, Gent, Flanders, Belgium\newline Email: maarten.deboeck@telenet.be} \and Geertrui Van de Voorde \thanks{Postdoctoral fellow of the Research Foundation Flanders (FWO -- Vlaanderen) and supported by the Marsden Fund Council administered by the Royal Society of New Zealand\newline University of Canterbury, School of Mathematics and Statistics, Private Bag 4800, 8140 Christchurch, New Zealand \newline Email: geertrui.vandevoorde@canterbury.ac.nz}}

\begin{document}
\maketitle
\begin{abstract}
In this paper, we study the weight distributions of $\F_q$-linear sets in $\PG(1,q^5)$. Our Main Theorem proves that a linear set $S$ of rank $5$, which is not scattered has the following weight distribution for its points with weight larger than 1: (i) one point of weight 4 or 5, (ii) one point of weight 3 and 0, $q$, or $q^2$ points of weight 2, (iii) $s$ points of weight 2 where $s\in [q-2\sqrt{q}+1,q+2\sqrt{q}+1]\cup\{2q,2q+1,2q+2,3q,3q+1,q^2+1\}$. In particular, there are no $2$-clubs in $\PG(1,q^5)$.
 \end{abstract}
 


\section{Introduction}

\subsection{Linear sets and their weight distribution}

Linear sets are particular subsets of a projective space and form a natural generalisation of subgeometries. Apart from being an interesting combinatorial subject in their own right, linear sets have been used in the construction and characterisation of several combinatorial objects, e.g. blocking sets, translation ovoids and semifields (see e.g. \cite{micheloverview,olga} for an overview of these applications). More recently, the study of linear sets regained traction through its connection with rank-metric codes (see \cite{polzul,zulzin} and Remark \ref{remark:RD}). We will review some of those links in Subsection \ref{connections}.

Despite these intensive investigations in the last decades, many questions about linear sets remain open. For example, their possible weight distributions, which also determines their possible sizes, remains an open problem in general. 

Linear sets can be defined as follows. Let $\F_q$ denote the finite field of order $q$, where $q$ is a prime power, and let $\PG(n,q)=\PG(V)$ denote the projective space of dimension $n$ over $\F_q$, where $V$ is an $(n+1)$-dimensional vector space over $\F_q$. 

A set $S$ is said to be an {\em $\F_q$-linear set} of {\em rank $k$} in $\PG(r-1,q^t)=\PG(W)$ if $S=L_U$, with
\[
	L_U=\left\{ \langle v\rangle_{q^t} \mid v\in U\setminus\{0\}\right\},
\]
where $U$ is a $k$-dimensional $\F_q$-vector subspace of $W=\F_{q^t}^{r}$ and $\langle v\rangle_{q^t}$ denotes the projective point determined by the vector $v$. 		

In this paper, we will study the possible {\em weight distributions} of $\F_q$-linear sets of $\PG(1,q^5)$. Let $P=\langle v\rangle_{q^t}$ be a point of a linear set $L_U$. The $t$-dimensional $\F_q$-vector space defining $P$ intersects $U$ in an $i$-dimensional $\F_q$-vector space for some $i>0$. The integer $i$ is called the {\em weight} of the point $P$ (see \cite{olga}).

It is clear that the number of points in a linear set is entirely determined if we know its weight distribution, but linear sets of the same size can have different weight distributions. An $\F_q$-linear set of rank $k$ which has one point of weight $k_0$ and whose other points have weight one is called a {\em $k_0$-club} and an $\F_q$-linear set of rank $k$ with precisely $\frac{q^{k+1}-1}{q-1}$ points (all of which are necessarily of weight one) is called {\em scattered}.

In this paper, we will prove the following theorem. 
\begin{rtheorem}
	Let $S$ be an $\F_q$-linear set of rank $5$ in $\PG(1,q^5)$ with $|S|>1$. If $S$ is not scattered, then either:
	\begin{enumerate}[(a)]
		\item $S$ contains exactly one point of weight $4$ (and hence, is a $4$-club), or
		\item $S$ contains exactly one point of weight $3$, and exactly $0$, $q$ or $q^2$ points of weight $2$, or
		\item $S$ contains exactly $s$ points of weight $2$, where
		\[
			s\in [q-2\sqrt{q}+1,q+2\sqrt{q}+1] \cup \{2q,2q+1,2q+2,3q,3q+1,q^2+1\},
		\]
		and no points of weight higher than $2$.
	\end{enumerate}
\end{rtheorem}
\begin{proof}
	This statement follows from the results of Subsection \ref{trace}, Theorem \ref{weight3}, Theorem \ref{eriseensecant} and Theorem \ref{arcs}.
\end{proof}
The cases $q=2,3,4$ are treated in more detail in Theorems \ref{qistwo},\ref{qisthree}, \ref{qisfour}. Note that linear sets with the same weight distribution are not necessarily equivalent. We do not address the equivalence problem in this paper.

%
%

\subsection{Explicit constructions}\label{constructions}
The Main Theorem does not say that all possibilities necessarily occur. In this subsection, we will describe the situation in more detail. 
\begin{itemize}
	\item	It is well-known that scattered linear sets of rank $n$ exist for all $n$, the standard example being $$\{\langle(x,x^q)\rangle_{q^n}\mid x\in \F_{q^n}^*\}.$$
	\item An $(n-1)$-club of rank $n$, which has one point of weight $n-1$ (and then necessarily all others of weight $1$), can always be constructed by taking $$\{\langle(x,\mathrm{Tr}_{q^n\mapsto q}(x))\rangle_{q^n}\mid x\in \F_{q^n}^*\}.$$
	\item	The construction of an $\F_q$-linear set of rank $5$ with one point of weight $3$ and exactly $q^2$ points of weight $2$ essentially follows from Theorem \ref{weight3}; we can take the set of points in $\PG(1,q^5)$ with coordinates given by
	$$\{\langle(\mu_1+\mu_2\alpha+\mu_3\alpha^2,\mu_4+\mu_5\alpha)\rangle_{q^5}\mid (\mu_1,\ldots,\mu_5)\in \F_q^5\setminus (0,0,0,0,0)\},$$ where $\alpha$ is a primitive element of $\F_{q^5}$. Here $\langle(1,0)\rangle$ has weight $3$, and the $q^{2}$ points $\langle(\lambda_1+\lambda_2\alpha,1)\rangle$, with $\lambda_1,\lambda_2\in\F_{q}$ have weight $2$. This example was generalised in \cite[Theorem 2.12]{preprint} to give a large class of examples of linear sets of rank $k$ of size $q^{k-1}+1$ which are not clubs (see Remark \ref{remjena}). 
		\item The construction of an $\F_q$-linear set of rank $5$ with one point of weight $3$ and exactly $q$ points of weight $2$ also follows from Theorem \ref{weight3}. For this, consider a primitive element $\alpha$ in $\F_{q^5}$ and an element $\beta\in \F_{q^5}$ which is not of the form $\frac{a\alpha+b}{c\alpha+d}$ with $a,b,c,d\in \F_q$ and take the set of points in $\PG(1,q^5)$ with coordinates given by
	\[
	\{\langle(\mu_1+\mu_2\alpha+\mu_3\alpha^2,\mu_4+\mu_5\beta)\rangle_{q^5}\mid (\mu_1,\ldots,\mu_5)\in \F_q^5\setminus (0,0,0,0,0)\}\;.
	\]
	\item 	A $3$-club of rank $5$ in $\PG(1,q^5)$, that is, a linear set with one point of weight $3$ and all others of weight $1$, was already constructed in \cite[Lemma 2.12]{linearsetview} as follows:
	\[
	\{\langle(\mu_1\alpha+\mu_2\alpha^2+\mu_3\alpha^3+\mu_4\alpha^4,\mu_4+\mu_5\alpha)\rangle_{q^5}\mid (\mu_1,\ldots,\mu_5)\in \F_q^5\setminus (0,0,0,0,0)\}\;.
	\]
	\item Our computer results (see Section \ref{kleineq}) show that for linear sets without points of weight $3$ or $4$, almost all of the possibilities for $s$ as described in the theorem will occur. The only exception is the case that $s=3q$ or $s=3q+1$ of which we conjecture that they never occur provided that there is no other value in the set $\{2q+2,q^2+1\}$ which equals $3q$ or $3q+1$, respectively. For example, for $q=2$, the possibility $s=6=3q$ does occur because $s=2q+2$, and for $q=3$, $s=10=3q+1$ occurs since it equals $q^2+1$.
	While it is hard to give an explicit construction for each of the possibilities for $s$, we can work our way backwards through the arguments of Theorem \ref{eriseensecant} to provide a construction of linear sets with exactly $2q, 2q+1,2q+2$ points of weight $2$. All of the linear sets constructed in Theorem \ref{eriseensecant} will be of the form 	$$L_{\gamma,\delta_1,\delta_2}=\{\langle(\mu_1+\mu_2\gamma+\mu_3\gamma\delta_1,\mu_4+\mu_5\gamma+\mu_3\gamma\delta_2)\rangle_{q^5}\mid (\mu_1,\ldots,\mu_5)\in \F_q^5\setminus (0,0,0,0,0)\},$$
	for some fixed $\gamma,\delta_1,\delta_2\in \F_{q^5}$. The exact number of points of weight $2$ is $2q,2q+1$ or $2q+2$, depending on the relation between $\gamma,\delta_1,\delta_2$.  For example, looking at subcase B.2.2 in the proof of Theorem \ref{eriseensecant} shows that if $\gamma\delta_2\in\langle 1,\gamma,\gamma^2,\gamma\delta_1\rangle$, $\gamma\delta_1\notin\langle 1,\gamma,\gamma^2,\gamma\delta_2\rangle$, and $\dim(\langle 1,\gamma,\gamma^2,\gamma\delta_1,\gamma^2\delta_1\rangle)\neq 5$, then the number of points of weight $2$ on $L_{\gamma,\delta_1,\delta_2}$ is $2q$.
	\par Similarly, all linear sets constructed in Theorem \ref{arcs} are of the form 
	\begin{multline*}
	L_{\gamma_0,\gamma_1,\gamma_2,\gamma_1',\gamma_2'}=\\\{\langle(\mu_1+\mu_2\gamma_0+\mu_3\gamma_0\gamma_1+\mu_4\gamma_2,\mu_5+\mu_3\gamma_1'+\mu_4\gamma_2')\rangle_{q^5}\mid (\mu_1,\ldots,\mu_5)\in \F_q^5\setminus (0,0,0,0,0)\}\;,
	\end{multline*}
	where $\gamma_0,\gamma_1,\gamma_2,\gamma_1',\gamma_2'$ are in $\F_{q^5}$ and the number of points on $L_{\gamma_0,\gamma_1,\gamma_2,\gamma_1',\gamma_2'}$ depends on the relation between the parameters, as considered in the several subcases. For example, looking at Case A.6 shows that if $\dim\langle 1,\gamma_0,\gamma_1',\gamma_0\gamma_1'\rangle=4$ and $\gamma_2',\gamma_0\gamma_2',\gamma_0\gamma_1,\gamma_2,\delta\gamma_2'+\gamma_1'\gamma_2\in \langle 1,\gamma_0,\gamma_1',\gamma_0\gamma_1'\rangle$, then there are precisely $q^2+1$ points of weight $2$ in $L_{\gamma_0,\gamma_1,\gamma_2,\gamma_1',\gamma_2'}$.
	However, in other subcases, the only conclusion we can draw is that the number of points of weight $2$ lies in between $[q-2\sqrt{q}+1,q+2\sqrt{q}+1]$; in order to deduce the precise value, we would need to have an explicit expression for the number of points on a cubic curve in function of its coefficients, which is not possible (see also Remark \ref{refinement} for a slight refinement  and the connection with cubic curves).
	\item In \cite{zanella}, an explicit construction of a wide class of $\F_q$-linear sets of rank $5$ in $\PG(1,q^5)$ is given. We discuss this in more detail in Remark \ref{remarkMontanucci}.
\end{itemize}

\begin{remark}\label{remjena} In \cite{preprint}, the authors construct a large family of $\F_q$-linear sets of rank $k$ in $\PG(1,q^{k})$ of size $q^{k-1}+1$ (which is the smallest possible size under the hypothesis that there is a point of weight one in the set). An easy counting argument shows that a linear set of rank 5 in $\PG(1,q^{5})$ of size $q^4+1$ has either:
\begin{enumerate}[(a)]
\item one point of weight $4$ and the other points of weight $1$ (a $4$-club),
\item one point of weight $3$, $q^2$ points of weight $2$ and $q^4-q^2$ points of weight one,
\item $q^2+q+1$ points of weight $2$ and $q^4-q^2-q$ points of weight one.
\end{enumerate}
They show that all the examples of type (b) can be obtained from their construction. Moreover, in $\PG(1,q^5)$ all $4$-clubs are equivalent (see \cite[Theorem 2.3]{linearsetview} and \cite[Theorem 3.7]{classes}), and hence, also arise from their construction. Our Main Theorem shows that possibility (c) does not occur. In other words, we see that all linear sets of size $q^4+1$ in $\PG(1,q^5)$ arise from the construction of \cite{preprint}.
\end{remark}

Very recently, in \cite[Theorem 4.4]{napolitano}, the authors study linear sets with complementary weights. In particular, they show that if a linear set of rank $n$ in $\PG(1,q^n)$ has exactly two points of weight greater than one then both must have weight at most $n/2$. In particular, this theorem says that there are no $\F_q$-linear sets of rank $5$ in $\PG(1,q^5)$ with exactly one point of weight $3$ and exactly one point of weight $2$.


\subsection{Connections with other research problems}\label{connections}
\subsubsection{Desarguesian spreads and field reduction}\label{opmerkingspread}

Every point of a linear set $S$ of rank $k+1$ in $\PG(1,q^5)$ can be represented as an element of a Desarguesian $4$-spread $\D$ in $\PG(9,q)$ which meets a certain $k$-dimensional subspace $\pi$ of $\PG(9,q)$. The {\em weight} of a point is then one more than the (projective) dimension of the intersection of the corresponding spread element with $\pi$. Hence, the possible weight distributions of linear sets are determined by the possible ways in which a $k$-space can intersect a Desarguesian $4$-spread $\D$. In general, the weight of a point in a linear set $S$ is only defined if we specify the vector space $U$ defining $S=L_U$. However, if a linear set of rank $5$ in $\PG(1,q^5)$ is defined by two different subspaces, say $\pi_1$ and $\pi_2$, then each spread element of $\D$ will meet $\pi_1$ and $\pi_2$ in a subspace of the same dimension (see \cite[Theorem 5.5]{Carlitz}). 

\subsubsection{Linearised polynomials}\label{remark:polynomials}

It is well-known that every $\F_q$-linear set $L$ of rank $5$ in $\PG(1,q^n)$, disjoint from the point $\langle(0,1)\rangle_{q^5}$ can be written as
$$L=\{\langle (x,f(x))\rangle_{q^5}\mid x \in \F_{q^5}\},$$
for some {\em $q$-polynomial} $f$ defined over $\F_{q^5}$. A $q$-polynomial, or {\em linearised} polynomial, is a polynomial of the form $f(x)=\sum_{i=0}^{4}a_ix^{q^i}$ for some $a_i\in \F_{q^5}$. Every $q$-polynomial defines an $\F_q$-linear map on $\F_{q^5}$ and conversely, every $\F_q$-linear map is defined by a $q$-polynomial. The weight of a point $\langle (1,\gamma)\rangle$ in $L$ is precisely the dimension of the $\F_q$-vector space of solutions $x$ to the equation $\frac{f(x)}{x}=\gamma$, which means that the weight distribution of the linear set $L$ is given by the multiset
$$\{\dim(\ker(f(x)-\gamma x))\mid \gamma\in \F_{q^5}\}\;.$$
Hence, our Main Theorem will describe all possibilities for $\{\dim(\ker(f(x)-\gamma x))\mid \gamma\in \F_{q^5}\},$ where $f$ is an arbitrary $q$-polynomial over $\F_{q^{5}}$.

\subsubsection{Rank distance codes}\label{remark:RD} 
Linear sets of rank $n$ in $\PG(1,q^n)$ and their weight distributions give rise to $\F_q$-linear rank distance codes of $n\times n$-matrices. A {\em rank distance code} (or RD-code) $\C$ is a subset  of the set of $m\times n$-matrices over $\F_q$, endowed with the metric
$$d(A,B)=rk(A-B)$$ for $A,B$ in $\F_q^{m\times n}$. If $\C$ forms a vector subspace, we see that the minimum rank of a non-zero element in $\C$ determines the minimum distance of the code, and more generally, the rank distribution determines the possible distances between code words in $\C$.
If $m=n$, then every matrix in $\F_q^{n\times n}$ defines an $\F_q$-linear map from $\F_{q^n}$ to $\F_{q^n}$, and alternatively, we can describe each such $\F_q$-linear map by a $q$-polynomial of degree at most $q^{n-1}$ as in the previous subsection. So, if we consider a set $U_f=\{(x,f(x))\mid x\in \F_{q^n}\}$, where $f$ is a $q$-polynomial, then $\C_f=\{ax+bf(x)\mid a,b \in \F_{q^n}\}$ determines a set of $q$-polynomials which forms a $\F_{q^n}$-subspace, so the code $\C_f$ is $\F_{q^n}$-linear and the dimension over $\F_{q^n}$ is two.

It is precisely this correspondence between the linear set $L_{U_f}$ in $\PG(1,q^n)$ and the RD-code $\C_f$ that has been exploited in recent years to construct new {\em maximum} rank distance (MRD) codes from scattered linear sets \cite{polzul,zulzin}. 
The weight distribution of the linear set $L_{U_f}$ in $\PG(1,q^n)$, determined by the vector subspace $U_f$ and the rank distribution of the code $C_f$ are related as follows (see \cite[Proposition 5.5]{withjohn}, adapted here for $\PG(1,q^n)$). Let $P$ be a point with coordinates $\langle(x_0,y_0)\rangle$ satisfying $ax_0+by_0=0$, then
\[
wt_{L_{f}}(P)=n-rk(ax+bf(x)).
\]
For example, if the linear set $L_f=\{\langle(x,f(x))\rangle_{q^5}\mid x\in \F_{q^5}^*\}$ has $1$ point of weight $3$ and exactly $q$ points of weight $2$, we find that the corresponding RD code $\C_f$ has $(q^5-1)$ code words of rank $2$, $q(q^5-1)$ code words of rank $3$ and $(q^4+q^3-q^2-q)(q^5-1)$ code words of rank $4$. The remaining $(q^5-q^4-q^3+q^2-1)(q^5-1)$ code words have rank $5$. Our Main Theorem will describe all possibilities for the rank distributions of code words in $\F_{q^5}$-linear rank distance codes of $\F_{q^5}$-dimension two.

\subsubsection{KM-arcs}\label{clubs}

The work in this paper was partially motivated by the existence problem of {\em $2$-clubs} (which have one point of weight $2$ and all others of weight one). $\F_2$-linear $t$-clubs have been shown to be equivalent to translation KM-arcs of type $2^t$ (see \cite {linearsetview}). A computer search from \cite{limb} (phrased in a coding theoretical setting) already showed that there are no $2$-clubs in $\PG(1,2^5)$. The Main Theorem of this paper, together with Theorems \ref{qistwo}, \ref{qisthree}, and  \ref{qisfour} show that this result holds for all $q$, i.e. there are no $\F_q$-linear $2$-clubs in $\PG(1,q^5)$.

\subsection{The weight distribution of \texorpdfstring{$\F_q$}{Fq}-linear sets of rank 5 in \texorpdfstring{$\PG(1,q^5)$}{PG(1,q5)}, \texorpdfstring{$q=2,3,4$}{q=2,3,4}.}\label{kleineq}
While our proof works for $q=2,3,4$, it is worth investigating which possibilities actually occur in these cases. In particular, for $q\in \{2,3,4\}$, the lower bound $q-2\sqrt{q}+1$ in the third bullet point in the Main Theorem still allows the possibility that there is a linear set with exactly one point of weight $2$, a $2$-club. However, such linear sets do not exist (see Remark \ref{clubs}). We also see that the possibilities $3q$ and $3q+1$ in the third bullet point do not occur for $q=2,3,4$ (recall that for $q=2$, the possibilities $2q+2$ and $3q$ coincide and for $q=3$, the possibilities $3q+1$ and $q^2+1$ coincide), which confirms our conjecture made in Subsection \ref{constructions}  for these small values of $q$.

It is not too hard to determine all possibilities for the weight distribution in $\PG(1,2^5)$ and $\PG(1,3^5)$ by computer. Using the GAP package FinInG \cite{fining}, we found the following.

\begin{theorem}\label{qistwo}Let $S$ be an $\F_2$-linear set of rank $5$ in $\PG(1,2^5)$ with $|S|>1$. Then either:
\begin{enumerate}[(a)]
\item $S$ contains one point of weight $4$ and $16$ of weight $1$ (and hence, is a $4$-club), or
\item $S$ contains one point of weight $3$, and $0$, $2$ or $4$ of weight $2$, and all others of weight $1$, or
\item $S$ contains $s$ points of weight $2$, where $s\in \{2,3,4,5,6\}$, and all others of weight $1$, or
\item $S$ contains $31$ points of weight $1$ (and hence, is scattered).
\end{enumerate}
We find that the possible sizes for an $\F_2$-linear set of rank $5$ in $\PG(1,2^5)$ are $17$, $19$, $21$, $23$, $25$, $27$ and $31$ and all these possibilities occur.
\end{theorem}

\begin{theorem} \label{qisthree} Let $S$ be an $\F_3$-linear set of rank $5$ in $\PG(1,3^5)$ with $|S|>1$. Then either:
\begin{enumerate}[(a)]
\item $S$ contains one point of weight $4$ and $81$ of weight $1$ (and hence, is a $4$-club), or
\item $S$ contains one point of weight $3$, and $0$, $3$ or $9$ of weight $2$, and all others of weight $1$, or
\item $S$ contains $s$ points of weight $2$, where $s\in \{2,3,4,5,6,7,8,10\}$, and all others of weight $1$, or
\item $S$ contains $121$ points of weight $1$ (and hence, is scattered).
\end{enumerate}
We find that the possible sizes for an $\F_3$-linear set of rank $5$ in $\PG(1,3^5)$ are $82$, $91$, $97$, $100$, $103$, $106$, $109$, $112$, $115$ and $121$ and all these possibilities occur.
\end{theorem}

\begin{theorem} \label{qisfour}
	Let $S$ be an $\F_4$-linear set of rank $5$ in $\PG(1,4^5)$ with $|S|>1$. Then either:
	\begin{enumerate}[(a)]
		\item $S$ contains one point of weight $4$ and $256$ of weight $1$ (and hence, is a $4$-club), or
		\item $S$ contains one point of weight $3$, and $0$, $4$ or $16$ of weight $2$, and all others of weight $1$, or
		\item $S$ contains $s$ points of weight $2$, where $s\in \{2,3,4,5,6,7,8,9,10,17\}$, and all others of weight $1$, or
		\item $S$ contains $341$ points of weight $1$ (and hence, is scattered).
	\end{enumerate}
	We find that the possible sizes for an $\F_3$-linear set of rank $5$ in $\PG(1,4^5)$ are $257$, $273$, $301$, $305$,  $309$, $313$, $317$, $321$, $325$, $329$, $333$ and $341$ and all these possibilities occur. 
\end{theorem}

\begin{remark}\label{remarkMontanucci}
	In \cite{zanella}, the authors study linear sets of rank $5$ in $\PG(1,q^5)$ of the form \begin{align} L_{\alpha,\beta}=\{\langle (x-\alpha x^{q^2},x^q-\beta x^{q^2})\rangle_{q^5}\vert x\in \F_{q^5}\},\label{zan}\end{align} where $\alpha^{q}\neq \beta^{q+1}$, aiming to find conditions on $\alpha,\beta\in \F_{q^h}$ to ensure that the resulting linear set is scattered. Let $\Sigma\cong \PG(4,q)$ be a canonical subgeometry of $\PG(4,q^5)$ and $\sigma$ the collineation of $\PG(4,q^5)$ whose fixed points are precisely those of $\Sigma$. Then the linear sets of the form \eqref{zan} are precisely those arising from the projection of $\Sigma$ from a plane $\Pi$ with $\Pi\cap \Pi^\sigma=\{P\}$ where $P$ is a point with $\dim\langle P,P^\sigma,P^{\sigma^2},P^{\sigma^{q^3}},P^{\sigma^{q^4}}\rangle=4$. Linear sets as projections of subgeometries are explained in more detail in Section \ref{Pi}. 	\par The weight distribution of $L_{\alpha,\beta}$ is unknown in general but it can be argued that no clubs (which have size $q^4+1$) will have this form. Using GAP \cite{fining}, we checked the sizes of all linear sets of the form $L_{\alpha,\beta}$, where $\alpha^{q}\neq \beta^{q+1}$ for $q=2,3,4$. 	
	For $q=2$, we found that these possible sizes are $19,21,23,25$, which also means that not all linear sets in $\PG(1,2^5)$, different from a club, are of the form $L_{\alpha,\beta}$: in particular, no scattered linear set or linear set with exactly two points of weight $2$ is of the form $L_{\alpha,\beta}$. For $q=3,4$, we found that for all sizes mentioned in Theorems \ref{qisthree} and \ref{qisfour} respectively, except for $q^{4}+1$, there is a linear set of the form $L_{\alpha,\beta}$. 
\end{remark}

\begin{remark} In order to make the computation of the different weight distributions of linear sets of rank $5$ in $\PG(1,4^5)$ feasible, we used the following strategy. Note that scattered linear sets always exist and that the possible weight distributions when there is a point of weight at least $3$ follow from our Main Theorem (which is valid for all $q\geq 2$).
It follows from \cite[Theorem 2.3]{bencecorrado}, which is valid for $q>2$, that an $\F_q$-linear set which is not scattered of pseudoregulus type and does not contain a point of weight $3$ or $4$ necessarily arises from the projection of $\Sigma$ from a plane $\Pi$ with $\Pi\cap \Pi^{\sigma}=\{P\}$ where $P$ is a point. If the point $P$ has $\dim\langle P,P^\sigma,P^{\sigma^2},P^{\sigma^{3}},P^{\sigma^{4}}\rangle=4$ then the corresponding linear set is described by \eqref{zan}, and if the point $P$ has $\dim\langle P,P^\sigma,P^{\sigma^2},P^{\sigma^{3}},P^{\sigma^{4}}\rangle<3$ then the corresponding linear set would have a point of weight $3$ or $4$. Hence, the only case we still need to consider is when $\dim\langle P,P^\sigma,P^{\sigma^2},P^{\sigma^{3}},P^{\sigma^{4}}\rangle=3$. A reasoning, completely analogous to the one in \cite{zanella}, shows that in this case, we may assume that the plane $\Pi$ is spanned by the points $\langle (1,\alpha,\alpha^2,\alpha^3,0)\rangle,\langle(1,\alpha^q,\alpha^{2q},\alpha^{3q},0)\rangle$, where $\alpha$ is a generator of $\F_{q^5}$, and a point of the form $\langle (\lambda_1,\lambda_2,\lambda_3,\lambda_4,\gamma)\rangle$, where $\lambda_1,\lambda_2,\lambda_3,\lambda_4$ are arbitrary elements of $\F_q$, not all zero, and $\gamma$ is an arbitrary element of $ \F_{q^5}\setminus \F_{q}$. We could then use GAP to run through all possibilities for $\lambda_1,\lambda_2,\lambda_3,\lambda_4,\gamma$ and calculate the size of the corresponding linear sets, leading to Theorem \ref{qisfour}. 
\end{remark}

\subsection{Linear sets of rank at most 4 in \texorpdfstring{$\PG(1,q^5)$}{PG(1,q5)}}

The possible weight distributions for linear sets of rank $2,3,4$ are easy to determine when taking Subsection \ref{opmerkingspread} into account. We provide the list here for completeness. 
Note that a linear set of rank at least $t+1$ in $\PG(1,q^t)$ is necessarily the full line. We will not deal with that case in the paper.

\begin{result}
	Let $S$ be a linear set of rank $k$ in $\PG(1,q^t)$, $k\leq t$.
	\begin{enumerate} [(a)]
		\item If $k=1$, then $S$ contains a unique point of weight $1$, so $|S|=1$;
		\item if $k=2$, then $S$ contains:
		\begin{enumerate}[(i)]
			\item a unique point of weight $2$, $|S|=1$, or
			\item $q+1$ points of weight $1$, so $|S|=q+1$;
		\end{enumerate}
		\item if $k=3$, then $S$ contains:
		\begin{enumerate}[(i)]
			\item a unique point of weight $3$, so $|S|=1$, or
			\item a unique point of weight $2$ and $q^2$ points of weight $1$, so $|S|=q^2+1$, or 
			\item $q^2+q+1$ points of weight $1$, so $|S|=q^2+q+1$;
		\end{enumerate}
		\item (See \cite[Lemma 10]{michelik}) if $k=4$, then $S$ contains:
		\begin{enumerate}[(i)]
			\item a unique point of weight $4$, so $|S|=1$, or
			\item a unique point of weight $3$ and $q^3$ points of weight $1$, so $|S|=q^3+1$ or
			\item one point of weight $2$ and $q^3+q^2$ points of weight $1$, so $|S|=q^3+q^2+1$, or
			\item two points of weight $2$ and $q^3+q^2-q-1$ points of weight $1$, so $|S|=q^3+q^2-q+1$, or
			\item $q+1$ points of weight $2$ and $q^3-q$ points of weight $1$, so $|S|=q^3+1$, or
			\item $q^2+1$ points of weight $2$ and no other points. In this case, $S\cong \PG(1,q^2)$ and $t$ is even.
		\end{enumerate}
	\end{enumerate}
	Moreover, all of the above cases always occur, except for $(d)(vi)$ which occurs if and only if $t$ is even.
\end{result}

\begin{remark}
	This result describes the possibilities for the weight distributions, but does not address the equivalence. The equivalence problem for linear sets is in general a difficult problem, and has mostly been studied for linear sets of rank $n$ in $\PG(1,q^n)$ (see e.g. \cite{classes}). In $\PG(1,q^3)$ all linear sets of size $q^2+1$ are $\PGammaL$-equivalent, and all linear sets of size $q^2+q+1$ are $\PGammaL$-equivalent (see e.g. \cite{michelik}). However, linear sets of the same size are not necessarily equivalent (consider for example a subline $\PG(1,q^2)$ in $\PG(1,q^4)$ and a $2$-club of rank $3$ which does not form a subline, see also \cite[Subsection 2.2.2]{preprint}). Recently, in \cite[Corollary 5.4]{rfat} (which only appeared on the arXiv after submission of this paper), the authors provided a full list of the equivalence classes of $\F_q$-linear sets of rank $4$ in $\PG(1,q^4)$(see also \cite{bonoli}).
\end{remark}

\subsection{Strategy for the proof of the Main Theorem}

\subsubsection{Linear sets as projections of subgeometries and the set \texorpdfstring{$\Omega_2$}{Omega2}} \label{Pi}
A well-known result of \cite{lunardon} states that every linear set can be obtained as the projection of a suitable subgeometry. Applied to the main case of interest for this paper, the result says the following:
\begin{result} \label{linearsetprojection} Let $S$ be an $\F_q$-linear set of rank $5$ on the line $L\cong \PG(1,q^5)$ and suppose that $S$ spans $L$. Embed $L$ in $\PG(4,q^5)$. Then there exists a subgeometry $\Sigma\cong \PG(4,q)$ of $\PG(4,q^5)$ and a plane $\Pi$ of $\PG(4,q^5)$, disjoint from $\Sigma$ and disjoint from $L$, such that $S$ is obtained as the projection of $\Sigma$ from $\Pi$ onto $L$. 

\end{result}
In what follows, by a {\em line $\ell$ of $\Sigma$}, we mean a set of $q+1$ points of $\Sigma$ that lie on a line, say $L$ of $\PG(4,q^5)$; we say that the line $\ell$ {\em extends} to the line $L$ and that $L$ is the extension of the line $\ell$. Similarly, by a {\em plane of $\Sigma$}, we mean a set of $q^2+q+1$ points of $\Sigma$ that lie on a plane of $\PG(4,q^5)$. 
In order to help the reader, subspaces of the subgeometry $\Sigma$ will be denoted by small letters and subspaces of the space $\PG(4,q^5)$ will be denoted by capital letters.

Furthermore, the {\em weight $w$} of a point $P$ of $S$ can also be defined as $w=d+1$ where $d$ is the projective dimension of the intersection of the hyperplane $\langle P, \Pi\rangle$ with the subgeometry $\Sigma$. In other words, the preimage of a point of weight $w$ under the projection map defined above is a $(w-1)$-dimensional subspace of $\Sigma$.

A proof of the equivalence between this definition and the classical definition of the weight of a point (as in \cite{olga}) can be found in \cite[Proposition 2.7]{withjohn}.

\begin{definition} A point of $\PG(4,q^5)$ is said to have {\em rank 2} if it lies on the extension of a line of $\Sigma$, but does not lie on $\Sigma$ itself. We denote the set of points of rank $2$ by $\Omega_2$.
\end{definition}

It is clear that for a point of rank $2$, there is a unique such line of $\Sigma$ since if two concurrent lines meet $\Sigma$ in $q+1$ points, their common point is in $\Sigma$. 

Using the notation from Result \ref{linearsetprojection}, a point $P$ of weight $2$ in a linear set arises from the projection of a line $m$ of $\Sigma$, i.e. the $3$-space $\langle P,\Pi\rangle$ meets $\Sigma$ in a line $m$. This implies that the extension of $m$ intersects $\Pi$ in a point, say $R$ of rank $2$. Hence, for every point $P$ of weight $2$, we find a unique point $R$ of rank $2$. It is not too hard to see that if there are only points of weight $1$ and $2$ in the linear set, this correspondence is one-to-one:

\begin{result} \cite[Corollary 6.7]{withjohn} Let $S$ be a linear set, obtained as the projection from a subgeometry $\Sigma$ from a subspace $\Pi$ onto a subspace $L$. Suppose that $S$ only contains points of weight $1$ and $2$, then the number of points of weight $2$ equals $|\Omega_2\cap \Pi|$, i.e. the number of points of rank $2$ in $\Pi$.
\end{result}
It also follows that if we find a point of rank $2$ in $\Pi$, then either it is in one-to-one correspondence with a point of weight $2$ in the linear set $S$, or it gives rise to a point of weight at least $3$. This observation will allow us to determine the possible weight distributions for linear sets in $\PG(1,q^5)$.

\subsubsection{Overview of this paper}

In Section \ref{intersections}, we will develop the framework to investigate the intersection of $\Pi$ with $\Omega_2$, where $\Pi$ and $\Omega_2$ are as in Subsection \ref{Pi}. We will first determine the possible sizes of the intersection of a line with $\Omega_2$ (Theorem \ref{intersectionwithline}). We introduce a way of representing rank $2$ points (Lemma \ref{gammatype}) which allows us to efficiently describe when the line through two points of $\Omega_2$ is a $(q+1)$-secant (see Theorem \ref{qplusonesecant}). We then use the same tools to study the $(q^2+q+1)$-secants to $\Omega_2$ in Theorem \ref{allesiser}. 

These ideas are essentially enough to describe the possible weight distributions of linear sets of rank $5$ containing a point of weight $3$ in Theorem \ref{weight3}. The case where the linear set contains a point of weight $4$ is easy; it is included here for completeness (see Subsection \ref{trace}).

The rest of this paper is devoted to the case when the linear set only contains points of weight one and two. We will first show that in this case, it is impossible for $\Pi$ to contain two $(q+1)$ secants to $\Omega_2$. When there is exactly one $(q+1)$-secant to $\Omega_2$ in $\Pi$, we will show that there are $q-1,q$ or $q+1$ additional points of rank $2$ in $\Pi$ (Theorem \ref{eriseensecant}).

Finally, we turn to the most difficult case, where there is no $(q+1)$-secant to the set $\Omega_2$ in $\Pi$. In that case (see Theorem \ref{arcs}), we show that the set of points in $\Pi\cap\Omega_2$ is either empty or forms an arc of size $s$ where $s\in [q-2\sqrt{q}+1,q+2\sqrt{q}+1]$ or $s\in \{2q,2q+1,2q+2,3q,3q+1,q^2+1\}$. To this end, we show that the number of points in $\Omega_2$ is given by the number of points on a certain cubic curve. 

\begin{remark}\label{refinement} We can somewhat refine the possible values in the interval $[q-2\sqrt{q}+1,q+2\sqrt{q}+1]$ that can occur: a result of Waterhouse \cite[Theorem 4.1]{waterhouse} gives necessary and sufficient conditions on $q$ and $t$ such that an elliptic curve (i.e. a non-singular non-empty cubic curve) with $q+1-t$ $\F_q$-rational points exist. In particular, such a curve exists for all values with $\gcd(t,q)=1$. But for example, it also follows from this theorem that, if $q=p^3$, $p=5$, there is no elliptic curve with exactly $q+1-p=q-4$ distinct $\F_q$-rational points, and hence, there will also not be an $\F_q$-linear set in $\PG(1,q^5)$, $q=5^3$ with exactly $q-4$ points of weight $2$.
\end{remark}

The proofs of Theorems \ref{eriseensecant} and \ref{arcs} contain many different subcases, depending on the parameters defining the points spanning the subspace $\Pi$. 
For both proofs, the subcases are all treated in a similar way but require care in the actual computation. In particular, in the proof of Theorem \ref{arcs}, we cannot simply stop when we reduced the problem to finding the points on a cubic curve: in each case we need to exclude the possibility that this cubic curve consists of $1$ single point (as that would lead to a $2$-club, which we show not to exist). 

For both Theorems, we have included a few cases to demonstrate the methods occurring in the proof in the paper; the details of the remaining cases have been included in an appendix\ifthenelse{\equal{\versie}{arxiv}}{}{ which is only available in the arXiv version of this paper}.

\section{The intersection of a subspace with the set \texorpdfstring{$\Omega_2$}{Omega2}}\label{intersections}

\subsection{The intersection of a line with \texorpdfstring{$\Omega_2$}{Omega2}}

\begin{theorem}\label{intersectionwithline}
	If $L$ is a line of $\PG(4,q^5)$, disjoint from the subgeometry $\Sigma\cong \PG(4,q)$, then $L$ meets $\Omega_2$ in $0,1,2,q+1$ or $q^2+q+1$ points. 
	\par Furthermore, if $L$ is a line containing exactly $q+1$ points of $\Omega_2$, then the span of the lines of $\Sigma$ whose extensions meet $L$ is a hyperplane of $\Sigma$
	. If $L$ is a line containing exactly $q^2+q+1$ points of $\Omega_2$, then the span of the lines of $\Sigma$ whose extensions meet $L$ is a plane of $\Sigma$.
\end{theorem}
\begin{proof}
Let $L$ be a line, disjoint from $\Sigma$ and meeting $\Omega_2$ in at least $3$ points, say $P_1,P_2,P_3$. Denote the line of $\Sigma$ whose extension contains $P_i$ by $\ell_i$, $i=1,2,3$. 

\textbf{Case 1: Suppose  that the three lines $\ell_1,\ell_2,\ell_3$ are not mutually disjoint.} Without loss of generality we suppose that $\ell_1$ and $\ell_2$ have a point $r$ in common. All the extended lines of the lines in the plane $\langle \ell_1,\ell_2\rangle$ of $\Sigma$ then meet $L$ in a point of $\Omega_2$, and all these $q^2+q+1$ points are distinct. 

If $\ell_3$ is disjoint from the plane $\langle \ell_1,\ell_2\rangle$, then $L$ lies in $\Xi_1$, the extension of the plane $\langle \ell_1,\ell_2\rangle$ and in $\Xi_2$, the extension of the $3$-space $\langle \ell_1,\ell_3\rangle$. Since $\ell_3$ is disjoint from $\langle \ell_1,\ell_2\rangle$, $\Xi_1$ is not contained in $\Xi_2$.  But this forces $L$ to be the intersection line of $\Xi_1$ and $\Xi_2$, which in turn means that $L$ meets $\Sigma$ in $\ell_1$, a contradiction. This implies that $\ell_3$ and $\langle \ell_1,\ell_2\rangle$ have at least one point in common, say $s$. It follows that the plane $\langle r,L\rangle$ is the same plane as the plane $\langle s,L\rangle$, and hence, that $\ell_3$ lies in the plane $\langle \ell_1,\ell_2\rangle$. This shows that there are no points of $\Omega_2$ on $L$ that do not lie on an extended line of the plane $\langle \ell_1,\ell_2\rangle$.

\textbf{Case 2: Suppose  that the three lines $\ell_1,\ell_2,\ell_3$ are mutually disjoint.} Suppose first that $\langle \ell_1,\ell_2\rangle \neq \langle \ell_1,\ell_3\rangle$. Then, the $3$-spaces $\langle \ell_1,\ell_2\rangle$ and $ \langle \ell_1,\ell_3\rangle$ meet in a plane, say $\pi$. It follows that $L$ is contained in the extension of $\pi$, and consequently, that each of the $q^2+q+1$ lines of $\pi$ extends to a line meeting $L$ in a rank $2$ point. This is a contradiction as we have seen in Case 1 that it is impossible that $L$ lies in the extension of a plane and contains points of rank $2$ not arising from extended lines of that plane.

This implies that $\langle \ell_1,\ell_2\rangle = \langle \ell_1,\ell_3\rangle$. Let $L_i$ be the extension of $\ell_i$. Consider the $q+1$ lines of $\Sigma$ intersecting each of the lines $\ell_1,\ell_2$ and $\ell_3$ in a point (this is the opposite regulus defined by these three lines of $\Sigma$). These $q+1$ lines extend to $q+1$ lines of $\PG(4,q^5)$. Let $M_1,M_2,M_3$ be three of these extended lines and consider the regulus in $\PG(4,q^5)$ defined by $M_1,M_2,M_3$, say $\mathcal{R}$. The lines $L_1,L_2,L_3$ are then contained in the opposite regulus $\mathcal{R}^{opp}$. Since $L$ meets $L_1,L_2,L_3$ each in a point, we know that $L$ belongs to $\mathcal{R}$. Let $\ell_4$ be a line of the regulus defined by $\ell_1,\ell_2,\ell_3$ in $\Sigma$ and let $L_4$ be its extension line. Since $L_4$ meets $M_1,M_2,M_3$, we find that $L_4$ belongs to $\mathcal{R}^{opp}$. Since $L$ is a line of  $\mathcal{R}$, we find that $L$ meets with $L_4$. This intersection point is a rank $2$ point, so we find $q+1$ rank $2$ points on $L$ arising from extended lines of the regulus in $\Sigma$ through $\ell_1,\ell_2,\ell_3$.

Suppose now to the contrary that there is an additional point, say $P_{q+2}$, of rank $2$ on $L$, lying on the extension of a line $m$. Repeating the argument above, where the points $P_1,P_2,P_3$ are replaced by the points $P_1,P_2,P_{q+2}$, shows that the line $m$ needs to be contained in the space $\langle \ell_1,\ell_2\rangle$. It also follows that $m$ does not meet any of the $q+1$ lines of the regulus determined in $\Sigma$ by $\ell_1,\ell_2,\ell_3$: if $\langle m,\ell_j\rangle$ is a plane, then $L$ would be contained in the extension of this plane, and we have seen in Case 1 that it is impossible for $L$ to contain points of rank $2$ that do not arise from extended lines of $\langle m,\ell_j\rangle$; but any point $P_i$ with $i\notin\{j,q+2\}$ is a rank $2$ point not arising from an extended line in $\langle m,\ell_j\rangle$.

Repeating the reasoning above by replacing $\ell_1,\ell_2,\ell_3$ with the lines $\ell_1,\ell_2,m$ we find $q-2$ extra points of rank $2$ on the extension of $q-2$ lines in $\langle \ell_1,\ell_2\rangle$. It is clear that we can repeat this process for $m$ together with any two lines of the regulus through $\ell_1,\ell_2,\ell_3$ until we have found a spread of $q^2+1$ lines each of which extends to a line meeting $L$ in a point of rank $2$. The spread is closed under taking reguli, and hence, is a regular spread. This implies that the set of $q^2+1$ points on $L$ form an $\F_{q^2}$-subline (see e.g. \cite[Theorem 1.5]{sara}), which is impossible since $\F_{q^5}$ does not have $\F_{q^2}$ as a subfield.
\end{proof}

\subsection{The type of a point of rank 2}
\subsubsection{The notation \texorpdfstring{$P=Q_1+\gamma Q_2$}{P=Q1+gQ2}}


Let $P$ be a point of $\Omega_2$, where, as before, $\Omega_2$ is the set of points lying on an extended line of a subgeometry $\Sigma\cong \PG(4,q)$ in $\PG(4,q^5)$. Then $P$ lies on a unique extended line $L$ of $\Sigma$. Let $Q_1$ and $Q_2$ be two points of $L\cap \Sigma$. Let $\Sigma$ be the canonical subgeometry defined by the points whose homogeneous coordinates belong to $\F_q^{5}$ up to an $\F_{q^{5}}$-multiple. Hence, coordinates of a point of $\Sigma$ are of the form $\alpha(x_0,\ldots,x_4)$ where $x_i\in \F_q$ and $\alpha\in \F_{q^5}^*$. 

\textbf{From now on, when we take coordinates for a point of $\Sigma$, we choose $\alpha\in \F_q$, i.e. take a vector with entries in $\F_q$}.

The coordinates of $P$ can be written as a linear combination of the coordinates of the points $Q_1$ and $Q_2$, where we can take the coefficient of this linear combination in $Q_1$ to be $1$. We abuse notation to write this as $P=Q_1+\gamma Q_2$ for some $\gamma\in \F_{q^5}$. 

It should be clear that, given points $Q_1$ and $Q_2$, the value of $\gamma$ is only  determined up to $\F_q$-multiple. Moreover, given $P$ on an extended line $L$, the points $Q_1$ and $Q_2$ are not uniquely determined.
But we will show in the next lemma that the set of $\gamma'$s for which $P=Q'_1+\gamma' Q'_2$, where $Q'_1$ and $Q'_2$ are coordinates of points of $\Sigma$ from $\F_q$,  can easily be determined.

\begin{lemma}\label{gammatype}
	Let $P$ be a point of rank $2$ such that $P$ has coordinates $Q_1+\gamma Q_2$ and $Q'_1+\gamma' Q'_2$, where $Q_1,Q_2,Q'_1$ and $Q'_2$ are coordinates of points of $\Sigma$, taken in $\F_q^5$. 
	Then $\gamma,\gamma'\in \F_{q^5}\setminus \F_q$ and there exist $a,b,c,d\in \F_q$ with $ad-bc\neq0$ such that 
	$$ \gamma'=\frac{a\gamma +b}{c\gamma +d}\;.$$

	Vice versa, if $P=Q'_1+\gamma' Q'_2$ with $ \gamma'=\frac{a\gamma +b}{c\gamma +d}$ and $ad-bc\neq 0$, then $P=Q_1+\gamma Q_2$ for some coordinates $Q_1,Q_2$ of points in $\Sigma$.
\end{lemma}
\begin{proof} Since $P$ has rank $2$, it does not lie in $\Sigma$ by definition. Hence, $\gamma$ and $\gamma'$ do not belong to $\F_q$. Furthermore, if $P$ would lie on two different extended lines of $\Sigma$, then $P$ would be the intersection point of those two lines, which lies in $\Sigma$. Hence, $Q'_1$ and $Q'_2$ are a linear combination of $Q_1$ and $Q_2$. Since $Q_1,Q_2,Q_1',Q_2'$ are vectors over $\F_q$, this linear combination has coefficients in $\F_q$. So we can write $Q_1=dQ_1'+bQ_2'$ and $Q_2=cQ_1'+aQ_2'$ for some $a,b,c,d$ in $\F_q$ with  $ad-bc\neq 0$.
From $Q_1+\gamma Q_2=Q_1'+\gamma' Q_2'$ then follows that $dQ_1'+bQ_2'+\gamma(cQ_1'+aQ_2')=Q_1'+\gamma' Q_2'$. This implies that $(d+c\gamma)Q_1'+(b+a\gamma)Q_2'$ are coordinates for $P$. Dividing by $(d+c\gamma)$ then yields that $\gamma'=\frac{b+a\gamma}{d+c\gamma}$ as required.
Vice versa, if $P=Q'_1+\gamma' Q'_2$ with $ \gamma'=\frac{a\gamma +b}{c\gamma +d}$, then $P=Q_1+\gamma Q_2$ with $Q_1=dQ_1'+bQ_2'$ and $Q_2=cQ_1'+aQ_2'$.
\end{proof}

Consider a matrix $A=\begin{pmatrix} d&c\\b&a\end{pmatrix}$ with $\begin{vmatrix} d&c\\b&a\end{vmatrix}\neq 0$. Then $A$ induces an element of $\PGL(2,q)$ acting by left multiplication on the points of $\PG(1,q^5)$ whose homogeneous coordinates are taken as column vectors. More specifically, $A$ induces a mapping from the point $\langle(1,\gamma)\rangle$ to $\langle(1,\gamma')\rangle$, where $\gamma'=\frac{a\gamma +b}{c\gamma +d}$. We see that the different values for $\gamma'$ found in the previous lemma correspond to the orbit of the action of $\PGL(2,q)$ on the point $\langle(1,\gamma)\rangle$. Furthermore, there are $q^5-q$ elements $\gamma\in \F_{q^5}\setminus \F_q$ and each such element determines an orbit of length $|\PGL(2,q)|=q(q^2-1)$. This means that there are precisely $q^2+1$ different orbits on the elements of $\F_{q^5}\setminus \F_q$. We call these orbits on field elements {\em $G$-orbits}.

If there is a $\gamma$ such that $P=Q_1+\gamma Q_2$, then we say that $P$ has {\em type $G(\gamma)$}, where $G(\gamma)$ is the $G$-orbit containing $\gamma$. By the lemma above, the type of $P$ is well-defined.
From Theorem \ref{qplusonesecant}, together with Theorem \ref{intersectionwithline}, it will follow that a line,
disjoint from $\Sigma$ through two rank $2$ points of the same type is either a $(q+1)$-secant or a $(q^2+q+1)$-secant. 

We start by giving an alternative interpretation to the elements of $\F_{q^5}\setminus \F_q$ lying in the same $G$-orbit.

\begin{lemma} \label{gammas}
	Let $\gamma,\gamma'$ be two elements in $\F_{q^5}\setminus\F_q$. Then $G(\gamma)=G(\gamma')$ if and only if $$\dim\langle 1,\gamma,\gamma',\gamma\gamma'\rangle_q\leq 3\;.$$
	Here, $\dim\langle 1,\gamma,\gamma',\gamma\gamma'\rangle_q$ denotes the (vector) dimension of the $\F_q$-vector space spanned by the elements $1,\gamma,\gamma',\gamma\gamma'$ in $\F_{q^5}$ seen as a $5$-dimensional vector space over $\F_q$.
\end{lemma}
\begin{proof} Suppose that $G(\gamma)=G(\gamma')$, then $\gamma'=\frac{a\gamma+b}{c\gamma+d}$ for some $a,b,c,d\in \F_q$ with $ad\neq bc$. It follows that $$(c\gamma+d)\gamma'=a\gamma+b,$$ and hence that 
$$b+a\gamma-d\gamma'-c\gamma\gamma'=0.$$
We see that $\{1,\gamma,\gamma',\gamma\gamma'\}$ is a dependent set over $\F_q$, and hence, $\dim\langle 1,\gamma,\gamma',\gamma\gamma'\rangle_q\leq 3$.

Vice versa, if $\dim\langle 1,\gamma,\gamma',\gamma\gamma'\rangle_q\leq 3$, there is a non-trivial linear combination say $b+a\gamma-d\gamma'-c\gamma\gamma'=0$ for some $a,b,c,d\in \F_q$, not all zero.

If follows that $\gamma'=\frac{a\gamma+b}{c\gamma+d}$. Now $ad\neq bc$ since $\gamma'\notin \F_q$. It follows that $G(\gamma)=G(\gamma')$.
\end{proof}

\begin{theorem}\label{qplusonesecant} Let $L$ be a line of $\PG(4,q^5)$ that contains two points of rank $2$, say $P_1=Q_1+\gamma Q_2$ and $P_2=Q_3+\gamma' Q_4$ and suppose that $\dim\langle Q_1,Q_2,Q_3,Q_4\rangle=3$. Then
\begin{enumerate}
\item  $L$ is disjoint from $\Sigma$,
\item $L$ is a $(q+1)$-secant to $\Omega_2$ if and only if $G(\gamma)=G(\gamma')$,
\item if $L$ is a $(q+1)$-secant to $\Omega_2$, then the $q+1$ points of $\Omega\cap L$ form an $\F_q$-subline.
\end{enumerate}
\end{theorem}
\begin{proof} 
Let $P_1=Q_1+\gamma Q_2$ and $P_2=Q_3+\gamma' Q_4$, where $Q_i$ are points of $\Sigma$, normalised as before to have coordinates in $\F_q$, and where $\gamma,\gamma'\in\F_{q^{5}}\setminus\F_{q}$.  If a point $\xi_1(Q_1+\gamma Q_2)+\xi_2(Q_3+\gamma'Q_4)$ with $\xi_1,\xi_2$ in $\F_{q^5}^*$ would be in $\Sigma$, then both $\xi_1$ and $\xi_1\gamma$ are in $\F_{q}$, a contradiction. So $L$ is disjoint from $\Sigma$.

From Theorem \ref{intersectionwithline}, we know that $L$ is either a $(q+1)$-secant or a $2$-secant. So suppose that $L$ contains a point $R$ of $\Omega_2$, different from $P_1$ and $P_2$. Then $R$ can be written as $\xi_1(Q_1+\gamma Q_2)+\xi_2(Q_3+\gamma'Q_4)$ for some $\xi_1,\xi_2$ in $\F_{q^5}^*$ since it lies on the line through $P_1$ and $P_2$. Since $R$ is in $\Omega_2$, by Theorem \ref{intersectionwithline}, the extended lines through points of rank $2$ on a $(q+1)$-secant all lie in a $3$-space. So we can write $R$ as 
\[
	(\la_1Q_1+\la_2Q_2+\la_3Q_3+\la_4Q_4)+\xi_3(\mu_1Q_1+\mu_2Q_2+\mu_3Q_3+\mu_4Q_4)
\]
for some $\xi_3\in \F_{q^5}^*$, $(\la_1,\la_2,\la_3,\la_4),(\mu_1,\mu_2,\mu_3,\mu_4) \in (\F_q^4)^*$, and with $(\la_1,\la_2,\la_3,\la_4)\neq s(\mu_1,\mu_2,\mu_3,\mu_4)$ for every $s\in\F_q$. It follows that the following system of equations in $\xi_{1}$, $\xi_{2}$ and $\xi_{3}$ must have a solution:
\begin{align}\label{eq:notatie}
	\begin{cases}
		\xi_1=\la_1+\mu_1\xi_3\\
		\gamma\xi_1=\la_2+\mu_2\xi_3\\
		\xi_2=\la_3+\mu_3\xi_3\\
		\gamma'\xi_2=\la_4+\mu_4\xi_3
	\end{cases}.
\end{align}
Eliminating $\xi_1$ from the first two equations, $\xi_2$ from the final two and $\xi_3$ from the two remaining equations, we see that, if there is a solution, we have that $$\gamma'=\frac{(\la_1\mu_4-\la_4\mu_1)\gamma+(\mu_2\la_4-\la_2\mu_4)}{(\la_1\mu_3-\la_3\mu_1)\gamma+(\la_3\mu_2-\la_2\mu_3)}\;.$$
We will now check that, if we have an admissible a solution, then
\[
	D=\begin{vmatrix} \la_1\mu_4-\la_4\mu_1&\mu_2\la_4-\la_2\mu_4\\\la_1\mu_3-\la_3\mu_1&\la_3\mu_2-\la_2\mu_3\end{vmatrix}\neq 0\;,
\]
which then shows that $\gamma$ and $\gamma'$ are of the same type. We see that $D=(\la_1\mu_2-\la_2\mu_1)(\la_3\mu_4-\la_4\mu_3)$. If the first factor would be zero, then the first two equations in \eqref{eq:notatie} would force $\gamma$ to be in $\F_q$. Similarly, the second factor is non-zero since $\gamma'\notin \F_q$.

Conversely, suppose that $G(\gamma)=G(\gamma')$, i.e. $\gamma'=\frac{a\gamma+b}{c\gamma+d}$ for some $a,b,c,d\in\F_q$ with $ad-bc\neq 0$. Consider    $Q_3'=dQ_3+bQ_4$ and $Q_4'=cQ_3+aQ_4$ which are points of $\Sigma$ on the line $Q_3Q_4$. Since $P_2=Q_3+\frac{a\gamma+b}{c\gamma+d}Q_4$, $P_2$ also is the point $(c\gamma+d)Q_3+(a\gamma+b)Q_4=Q_3'+\gamma Q_4'$, and hence, $P_2=Q_3'+\gamma Q_4'$.

Now consider, for $\mu\in \F_q$, the point $R_\mu$ given by $Q_1+\gamma Q_2+\mu (Q_3'+ \gamma Q_4')$. The point $R_\mu$ is clearly a point of $L$ as its coordinates are a linear combination of the coordinates of $P_1$ and $P_2$. Now $Q_1+\gamma Q_2+\mu (Q_3'+ \gamma Q_4')=(Q_1+\mu Q_3')+\gamma (Q_2+\mu  Q_4')$. Since $\mu\in \F_q$, we see that $(Q_1+\mu Q_3')$ determines a point of $\Sigma$, and similarly,  $Q_2+\mu  Q_4'$ determines a point of $\Sigma$. Hence, $R_\mu$ is a point of rank $2$ for all $\mu\in \F_q$. Furthermore, since $\mu\in \F_q$, the $q$ points $\{R_{\mu}\mid\mu\in\F_q\}$ form together with $P_2$ an $\F_q$-subline. So, we have found $(q+1)$ points of rank $2$ on $L$. 
\end{proof}

\begin{theorem} \label{allesiser}
	Let $L$ be a $(q^2+q+1)$-secant to $\Omega_2$, where $L$ is disjoint from $\Sigma$. Then there are $q+1$ points of $\Omega_2\cap L$ that have the same type, say $t_0$, and the other $q^2$ points of $\Omega_2\cap L$ have each a mutually different type, different from $t_0$.
\end{theorem}
\begin{proof}
Recall from Theorem \ref{intersectionwithline} that all points of rank $2$ on the line $L$ arise from extended lines that lie in a plane $\pi$ in $\Sigma$. We have seen that there are precisely $q^2+1$ possible types for a point of rank $2$ on $L$. By the pigeonhole principle, there are two points, say $R_1$ and $R_2$ that have the same type. This implies that there are points $Q_0$, $Q_1$ and $Q_2$ in $\pi$, such that $R_1=Q_0+\gamma_1 Q_1$ and $R_2=Q_0+\gamma_2Q_2$ where $\gamma_1$ and $\gamma_2$ are in the same $G$-orbit. This implies that $\gamma_2= \frac{a\gamma_1+b}{c\gamma_1+d}$ for some $a,b,c,d\in \F_q$ with $ad-bc\neq 0$.  If $b=0$, then
\[
	-\frac{d}{\gamma_{1}}(Q_0+\gamma_1Q_1)+\frac{a}{\gamma_{2}}(Q_0+\gamma_2Q_2)=cQ_0-dQ_1+aQ_2\;,	
\]
hence $L$ would not be disjoint from $\Sigma$. We conclude that $b\neq 0$. Write $Q_3=(dQ_0+bQ_2)$ and $Q_4=(cQ_0+aQ_2)$, then $R_2=Q_3+\gamma_1 Q_4$.

Consider a point $R$ of rank $2$ on $L$, different from $R_1,R_2$.  Then $R$ can be written as
$\xi_1(Q_0+\gamma_1 Q_1)+\xi_2(Q_0+(a\gamma_1+b)/(c\gamma_1+d) Q_2)$ for some $\xi_1,\xi_2\in \F_{q^5}^*$. Moreover, since $R$ is a point of rank $2$, it can be written as $(\la_0Q_0+\la_1Q_1)+\gamma(\mu_0Q_0+\mu_1Q_2)$ for some $\gamma\in\F_{q^{5}}\setminus\F_{q}$ and some $\la_0,\la_1,\mu_{0},\mu_{1}\in\F_{q}$. Note that $\lambda_1
\neq 0$ since $R\neq R_2$ and $\mu_1
\neq 0$ since $R\neq R_1$. This gives rise to the following system in $\xi_{1}$, $\xi_{2}$ and $\gamma$:
\begin{align*}
	\begin{cases}
		\xi_1+\xi_2=\la_0+\gamma\mu_0\\
		\gamma_1\xi_1=\la_1\\
		\xi_2\left(\frac{a\gamma_1+b}{c\gamma_1+d}\right)=\gamma\mu_1
	\end{cases}.
\end{align*}
It follows that $$\gamma=\frac{(a\gamma_1+b)(\la_0\gamma_1-\la_1)}{\gamma_1((\mu_1c-\mu_0a)\gamma_1+\mu_1d-\mu_0b)}\;.$$

Now suppose that $\gamma$ lies in the $G$-orbit of $\gamma_1$. This implies that (as quadratic functions of $\gamma_1$) the denominator and numerator have a common root. If $\gamma_1=0$ is a common root, then either $\lambda_1=0$ or $b=0$, a contradiction.

First suppose that $a=0$. If $\lambda_0=0$, then, in order for $G(\gamma)$ to be $G(\gamma_1)$, we need to have that $\mu_1c=0$. Since $ad-bc\neq 0$, we see that $c\neq 0$ and hence, $\mu_1=0$,  a contradiction. So $\lambda_0\neq 0$.  Expressing that $\lambda_1/\lambda_0$ is a root of $\mu_1c\gamma_{1}+\mu_1d-\mu_0b$ yields that $\frac{\mu_0}{\mu_1}=\frac{c\la_1+d\lambda_0}{b\lambda_0}$. So for each of the $q-1$ values of $\lambda_1/\lambda_0\in \F_q^*$, we find one value for $\mu_0/\mu_1$ such that $\lambda_1/\lambda_0$ is a root of $\mu_1c\gamma+\mu_1d-\mu_0b$, and hence, $G(\gamma)=G(\gamma_1)$. This shows that, together with $R_1$ and $R_2$, there are precisely $q+1$ points of rank $2$ on $R_1R_2$ with the same type as $R_1$.

Now suppose that $a\neq 0$. If $\gamma_{1}=-b/a$ is the common root of the denominator and the numerator, then $(\mu_1c-\mu_0a)(-b/a)+\mu_1d-\mu_0b=0$, and hence, $\mu_1=0$, or $ad-bc=0$, a contradiction. This implies that, if $G(\gamma_1)=G(\gamma)$, then $\lambda_0\neq0$.
Now, as before, let $\lambda_1/\lambda_0$ be a fixed element in $\F_{q}^{*}$, different from $-b/a$. Then, expressing that $\lambda_1/\lambda_0$ is a root of $(\mu_1c-\mu_0a)\gamma+\mu_1d-\mu_0b$ yields that $\frac{\mu_0}{\mu_1}=\frac{\la_1c+d\la_0}{b\la_0+a\la_1}$. So for each of the $q-1$ values of $\lambda_1/\lambda_0\in \F_q^*$, we find one value for $\mu_0/\mu_1$ such that $\lambda_1/\lambda_0$ is a root of $(\mu_1c-\mu_0a)\gamma+\mu_1d-\mu_0b$, and hence, $G(\gamma)=G(\gamma_1)$. This shows that, together with $R_1$ and $R_2$, there are precisely $q+1$ points of rank $2$ on $R_1R_2$ with the same type as $R_1$.

Recall that there are $q^2$ other points of rank $2$ on $L$, and that there are precisely $q^2+1$ different types of rank $2$ points. We claim that all $q^2$ types that are different from $G(\gamma_{1})$, appear exactly once amongst these $q^2$ points of rank $2$. 

We have seen above that all the rank $2$ points on $R_1R_2$ that do not have type $G(\gamma_1)$ have type $\gamma$ of the form

$$\frac{(a\gamma_1+b)(\la_0\gamma_1-\la_1)}{\gamma_1((\mu_1c-\mu_0a)\gamma_1+\mu_1d-\mu_0b)}$$ 

for some $(\la_0,\la_1,\mu_0,\mu_1)\in \F_q^4$ such that the denominator and numerator do not have a common factor. Now suppose that two points with type $\gamma$ as above are in the same orbit, then there are $A,B,C,D\in \F_q$ with $AD-BC\neq 0$ and $(\la'_0,\la'_1,\mu'_0,\mu'_1)\in \F_q^4$ such that the denominator and numerator do not have a common factor and
$$\frac{(a\gamma_1+b)(\la_0\gamma_1-\la_1)}{\gamma_1((\mu_1c-\mu_0a)\gamma_1+\mu_1d-\mu_0b)}=\frac{A\frac{(a\gamma_1+b)(\la_0'\gamma_1-\la'_1)}{\gamma_1((\mu'_1c-\mu'_0a)\gamma_1+\mu'_1d-\mu'_0b)}-B}{C\frac{(a\gamma_1+b)(\la'_0\gamma_1-\la'_1)}{\gamma_1((\mu'_1c-\mu'_0a)\gamma_1+\mu'_1d-\mu'_0b)}-D}\;.$$
It follows that
\begin{multline*}
	\frac{(a\gamma_1+b)(\la_0\gamma_1-\la_1) }{\gamma_1((\mu_1c-\mu_0a)\gamma_1+\mu_1d-\mu_0b)}\\=\frac{A(a\gamma_1+b)(\la'_0\gamma_1-\la'_1)-B\gamma_1((\mu'_1c-\mu'_0a)\gamma_1+\mu'_1d-\mu_0'b)}{C(a\gamma_1+b)(\la_0'\gamma_1-\la'_1)-D\gamma_1((\mu_1'c-\mu'_0a)\gamma_1+\mu_1'd-\mu'_0b)}\;,
\end{multline*}
and
\begin{multline*}
(a\gamma_1+b)(\la_0\gamma_1-\la_1) (C(a\gamma_1+b)(\la_0'\gamma_1-\la'_1)-D\gamma_1((\mu_1'c-\mu'_0a)\gamma_1+\mu_1'd-\mu'_0b))=\\ (A(a\gamma_1+b)(\la'_0\gamma_1-\la'_1)-B\gamma_1((\mu'_1c-\mu'_0a)\gamma_1+\mu'_1d-\mu_0'b))\gamma_1((\mu_1c-\mu_0a)\gamma_1+\mu_1d-\mu_0b)\;.
\end{multline*}

Since $\gamma_1\in \F_{q^5}\setminus \F_q$, it follows that the expressions on the right and left hand side, as polynomials in $\gamma_1$ with coefficients in $\F_q$ have to coincide. In particular, we see that $(a\gamma_1+b)$ divides the left hand side, and since $\gamma_1((\mu_1c-\mu_0a)\gamma_1+\mu_1d-\mu_0b)$ is not a multiple of $(a\gamma_1+b)$, and $\gamma_1((\mu'_1c-\mu'_0a)\gamma_1+\mu'_1d-\mu'_0b)$ is not a multiple of $(a\gamma_1+b)$ we find that $B=0$. 
Similarly, we see that $\gamma_1$ divides the right hand side, and hence, also the left hand side. Since $\gamma_1$ is not a divisor of $(a\gamma_1+b)(\la_0'\gamma_1-\la'_1)$, we have that $C=0$.

This implies that if two rank $2$ points, say $R_3,R_4$, on $R_1R_2$ having a type different from $\gamma_1$, are in the same $G$-orbit, then they have exactly the same expression for $\gamma$ (up to an $\F_q$-scalar multiple). Going back to our expression for $\gamma$, we find that there are $\la_0,\la_1,\mu_0,\mu_1,\la_0',\la_1',\mu_0',\mu_1'$ such that $R_3=(\la_0Q_0+\la_1Q_1)+\gamma(\mu_0Q_0+\mu_1Q_2)$ and $R_4=(\la'_0Q_0+\la'_1Q_1)+\gamma(\mu'_0Q_0+\mu'_1Q_2)$ with
$$\gamma=\frac{(a\gamma_1+b)(\la_0\gamma_1-\la_1)}{\gamma_1((\mu_1c-\mu_0a)\gamma_1+\mu_1d-\mu_0b)}=\frac{(a\gamma_1+b)(\la_0'\gamma_1-\la_1')}{\gamma_1((\mu_1'c-\mu_0'a)\gamma_1+\mu_1'd-\mu_0'b)}\;,$$
and hence,
$$\frac{\la_0\gamma_1-\la_1}{(\mu_1c-\mu_0a)\gamma_1+\mu_1d-\mu_0b}=\frac{\la_0'\gamma_1-\la_1'}{(\mu_1'c-\mu_0'a)\gamma_1+\mu_1'd-\mu_0'b}\;.$$

We claim that $\la_0=z\la_0'$ and $\la_1=z\la_1'$, $\mu_0=z\mu_0$ and $\mu_1=z\mu_1'$ for some $z$, and hence, $R_3=R_4$ which finishes the proof. Suppose that $\la_0=z\la_0'$ and $ \la_1=z\la_1'$, then it follows that $c(\mu_1-\mu_1'z)=a(\mu_0-z\mu_0')$ and $d(\mu_1-\mu_1'z)=b(\mu_0-z\mu_0')$. If $(\mu_0-z\mu_0')\neq 0$, then it follows from $b\neq 0$ that $(\mu_1-\mu_1'z)\neq 0$, and it follows that $ad-bc=0$, a contradiction. Hence, we have that if $\la_0=z\la_0'$ and $\la_1=z\la_1'$, then $\mu_0=z\mu_0'$ and $\mu_1=z\mu_1'$.

We see that for all $\phi_1,\phi_2\in \F_q$ the point
\begin{align*}
	\phi_1R_3+\phi_2 R_4 &=\phi_1((\la_0Q_0+\la_1Q_1)+\gamma(\mu_0Q_0+\mu_1Q_2))\\&\qquad+\phi_2((\la'_0Q_0+\la'_1Q_1)+\gamma(\mu'_0Q_0+\mu'_1Q_2))\\
	&=(\phi_1\la_0+\phi_2\la_0')Q_0+(\phi_1\la_1+\phi_2\la_1')Q_1\\&\qquad+\gamma((\phi_1\mu_0+\phi_2\mu_0')Q_0+(\phi_1\mu_1+\phi_2\mu_1')Q_2)\;	
\end{align*}
is a point of rank $2$ with type $\gamma$. Now if it is not true that $\la_0=z\la_0'$ and $\la_1=z\la_1'$ for some $z$, then we can choose $\phi_1,\phi_2$ such that  $\phi_1\la_1+\phi_2\la_1'=0$ and $\phi_1\la_0+\phi_2\la_0'\neq 0$, but then we find that the point $\phi_1R_3+\phi_2 R_4$ lies on $Q_0Q_2$, so $\phi_1R_3+\phi_2 R_4=R_2$ and has type $G(\gamma)$, a contradiction since $R_2$ has type $G(\gamma_1)\neq G(\gamma)$.
\end{proof}

\subsection{The intersection of a plane with \texorpdfstring{$\Omega_2$}{Omega2}}

\begin{lemma}\label{tweelange}
	Let $\Pi$ be a plane, disjoint from $\Sigma$ and suppose that there are two distinct $(q^2+q+1)$-secants, say $L_1$ and $L_2$ to $\Omega_2$ in $\Pi$. Then there is a $3$-space of $\Sigma$ whose extension contains $\Pi$. Furthermore, $\Pi$ contains exactly $(q^2+1)(q^2+q+1)$ points of rank $2$.
\end{lemma}
\begin{proof}
	From Theorem \ref{intersectionwithline} we know that there are two planes, say $\pi_1$ and $\pi_2$ of $\Sigma$ whose extensions, say $\Pi_1$ and $\Pi_2$ meet $\Pi$ in $L_1$ and $L_2$ respectively. Now $L_1,L_2$ intersect in a point $P$ of $\Pi$ and since $\Sigma$ is $4$-dimensional, $\pi_1$ and $\pi_2$ have (at least) a point of $\Sigma$ in common, necessarily different from $P$. It follows that $\Pi_1$ and $\Pi_2$ have a line in common, and hence, they span a $3$-dimensional space $\Xi$. It follows that $\pi_1$ and $\pi_2$ intersect in a line, and that $\Pi$ is contained in $\Xi$ which is the extension of the $3$-space $\langle \pi_1,\pi_2\rangle$. It follows that every line of $\langle \pi_1,\pi_2\rangle$ intersects $\Pi$ which shows that there are at least $(q^2+1)(q^2+q+1)$ points of rank $2$ in $\Pi$. Suppose to the contrary that there is a point $R$ of rank $2$ in $\Pi$ that lies on the extension of the line $\ell$ in $\Sigma$ such that $\ell$ is not contained in $\langle \pi_1,\pi_2\rangle$. Since $\ell$ intersects $\langle \pi_1,\pi_2\rangle$ it follows that $\Sigma$ is contained in the $3$-dimensional space $\Xi$, a contradiction since $\Sigma$ is a $4$-dimensional space.
\end{proof}
\begin{remark}
	Suppose that $\Pi$ is a plane as in Lemma \ref{tweelange} above. We will see in Subsection \ref{trace} that the linear set obtained from projecting $\Sigma$ from $\Pi$ has a point of weight $4$, and hence, is a $4$-club.
\end{remark}

\begin{lemma}\label{langmeetskort} Let $\Pi$ be a plane, disjoint from $\Sigma$ and suppose that there is a unique $(q^2+q+1)$-secant $L$ to $\Omega_2$ in $\Pi$. Then the following hold. 
\begin{enumerate}
	\item If there is a $(q+1)$-secant $M$ to $\Omega_2$ in $\Pi$, then $L$ and $M$ meet in a point of $\Omega_2$.
	\item If there is a $(q+1)$-secant $M$ to $\Omega_2$ in $\Pi$, then every point of rank $2$ in $\Pi$ that is not on $L$ has the same type as the point $L\cap M$.
\end{enumerate}
\end{lemma}
\begin{proof}
	We know from Theorem \ref{intersectionwithline} that the points of rank $2$ on $L$ arise from extended lines in a plane $\pi$ of $\Sigma$, and that the points of rank $2$ on $M$ arise from $q+1$ lines of $\Sigma$, contained in a $3$-space, say $\mu$, of $\Sigma$. Since $\Sigma$ is $4$-dimensional, $\pi\cap \mu$ meet either in the plane $\pi$ or in a line of the plane $\pi$. Suppose that $\mu$ contains $\pi$. Then $\Pi$ is contained in the extension of the $3$-space $\mu$. But then every plane of $\mu$ extends to a plane meeting $\Pi$ in a $(q^2+q+1)$-secant, contradicting our assumption that there is a unique $(q^2+q+1)$-secant to $\Omega_2$ in $\Pi$. Hence, $\pi\cap \mu$ is a line $\ell$. The extension of the line $\ell$ meets $L$ in a point of rank $2$ that also lies on the line $M$, proving our first statement.
	\par Now suppose that there is a $(q+1)$-secant, say $M$, to $\Omega_2$, then by the first part, we know that $L$ and $M$ meet in a point of rank $2$, say $P_1$. The points of rank $2$ on $M$ all have the same type by Theorem \ref{qplusonesecant}, say $G(\gamma)$.  Let $P_{3}$ be a rank 2 point on $M$ different from $P_{1}$. As before, let $\pi$ be the plane of $\Sigma$ whose extension meets $\Pi$ in the line $L$, and $\mu$ be the 3-space generated by the lines of $\Sigma$ giving rise to the points on $M$. Using Lemma \ref{gammatype} and Theorem \ref{qplusonesecant} we can take four points $Q_1,Q_2,Q_4,Q_{5}$ spanning $\mu$ such that $P_1=Q_1+\gamma Q_2$ and $P_3=Q_4+\gamma Q_5$. Let $Q_3$ be a point in $\pi\setminus Q_{1}Q_{2}$. Using Theorem \ref{intersectionwithline} there is a $\gamma'\in\F_{q^{5}}$ such that $P_2=Q_1+\gamma'Q_3$ is a point on $L$. Now consider a point $P_4$ of rank $2$ in $\Pi$ not on $L$. This point can be written as $\xi_1P_1+\xi_2P_2+\xi_3P_3$ and as $(\la_1Q_1+\la_2Q_2+\la_3Q_3+\la_4Q_4)+\gamma''(\mu_1Q_1+\mu_2Q_2+\mu_3Q_3+\mu_5Q_5)$  for some $\xi_1,\xi_2,\xi_3\in\F_{q^{5}}$ and $\gamma''\in\F_{q^{5}}\setminus\F_{q}$. Note that $\xi_3\neq0$ since $P_{4}$ is not on $L$. It follows that $\xi_3=\la_4$ and $\xi_3\gamma=\mu_5\gamma''$.  Note that $\mu_{5}\neq0$ since $\xi_3\neq0$. Hence, $\gamma''=\frac{\la_4}{\mu_5}\gamma$, so $G(\gamma)=G(\gamma'')$.
\end{proof}

\section{Linear sets containing a point of weight at least 3}
\subsection{Linear sets containing a point of weight 4}\label{trace}

Let $S$ be a linear set of rank $5$ containing a point of weight $4$. We have seen in Subsection \ref{constructions} that this is a $4$-club and always exists. In terms of projections, we see that by Result \ref{linearsetprojection}, $S$ corresponds to the projection of $\Sigma$ from a plane $\Pi$ such that there is a $3$-space $\mu$ of $\Sigma$ whose extension contains $\Pi$. 


\subsection{Linear sets containing a point of weight 3}
It is easy to see (e.g. using Subsection \ref{opmerkingspread}) that if an $\F_q$-linear set $S$ of rank $5$ in $\PG(1,q^5)$ contains a point of weight $3$, it cannot contain any point of weight $4$ nor any additional point of weight $3$. We now show that if there is a point of weight $3$, the number of points of weight $2$ can only take $3$ different values.

\begin{theorem} \label{weight3}Consider an $\F_q$-linear set $\S$ of rank $5$ in $\PG(1,q^5)$. If $S$ contains a point of weight $3$, then there are exactly $0$, $q$ or $q^2$ points in $\S$ of weight $2$.
\end{theorem}
\begin{proof}
Since $\S$ contains a point of weight $3$, by Result \ref{linearsetprojection}, we know that $\S$ corresponds to the projection of $\Sigma$ from a plane $\Pi$ such that there is a plane of $\Sigma$ whose extension meets $\Pi$ in a line, which then is a $(q^2+q+1)$-secant $L$ to $\Omega_2$. It is clear that there is a unique such plane, and hence, that $L$ is the only $(q^2+q+1)$-secant in $\Pi$. By Theorem \ref{allesiser}, we know that there are $q+1$ points of $L\cap \Omega_2$ with the same type, say $t_0$ and all $q^2$ other points of $L\cap \Omega_2$ have a mutually different type. Recall that there are precisely $q^2+1$ different types. 

Now assume that there is a rank $2$ point $P$ in $\Pi$, not on $L$. Let $t$ be the type of $P$. For every point $Q$ of $L\cap \Omega_2$ with type $t$, by Theorem \ref{qplusonesecant}, the line through $Q$ and $P$ is a $(q+1)$-secant to $\Omega_2$, all of whose points have type $t$. 

Suppose first that $Q$ is the unique point of type $t$ on $L$ and that there is a point $R$ of rank $2$ in $\Pi$ that does not lie on $L$ nor on the line $PQ$. By Lemma \ref{langmeetskort}, $R$ has the same type as $P$, which implies that $PR$ is a $(q+1)$-secant which then would meet $L$ in a point of type $t$, different from $Q$, a contradiction. We conclude that if $Q$ is the unique point of type $t$ on $L$, $S$ has one point of weight $3$ and precisely $q$ of weight $2$. 

Suppose now that there are $q+1$ points of type $t$ on $L$, then we find $q+1$ concurrent $(q+1)$-secants to $\Omega_2$ through $P$. Using the same reasoning as above we find that there cannot be additional points of rank $2$ in $\Pi$ because they would give rise to an extra $(q+1)$-secant meeting $L$ in a rank $2$ point of type $t$. So we find that $S$ has one point of weight $3$ and precisely $q^2$ of weight $2$.
\end{proof}

\begin{remark}\label{remarkconstruction}
	The proof of Theorem \ref{weight3} shows us how to construct linear sets of rank $5$ with one point of weight $3$ and precisely $q$ or $q^2$ points of weight $2$. Let $Q_0,Q_1,Q_2,Q_3,Q_4$ be five points spanning $\Sigma$ and consider the points $P_1=Q_0+\alpha Q_1, P_2=Q_0+\frac{1}{\alpha} Q_2$, $\alpha\in  \F_{q^5}\setminus \F_q$. If $P_3=Q_3+\alpha Q_4$, then the projection of $\Sigma$ from $\Pi=\langle P_1,P_2,P_3\rangle$ is a linear set of rank $5$ with one point of weight $3$ and $q^2$ points of weight $2$. If $P_3'=Q_3+\beta Q_4$, with $G(\beta)\neq G(\alpha)$, then the projection of $\Sigma$ from $\Pi=\langle P_1,P_2,P_3'\rangle$ is a linear set of rank $5$ with one point of weight $3$ and $q$ points of weight $2$. When we take $Q_i$ to be the point determined by the $i$-th standard vector, these point sets are precisely those given in Subsection \ref{constructions}. 
\end{remark}

\section{Linear sets with only points of weight 1 and 2}

\subsection{When there is a \texorpdfstring{$(q+1)$}{(q+1)}-secant to \texorpdfstring{$\Omega_2$}{Omega2}}

\begin{lemma} \label{notwointersectingsublines}
	Let $\Pi$ be a plane, disjoint from $\Sigma$. Suppose that there are no $(q^2+q+1)$-secants to $\Omega_2$ in $\Pi$. Then it is impossible for $\Pi$ to contain two $(q+1)$-secants to $\Omega_{2}$ meeting in a point of $\Omega_2$.
\end{lemma}
\begin{proof} Suppose that there are two $(q+1)$-secants, say $L$ and $M$, intersecting in a point $R$ of rank $2$. It follows from Theorems \ref{intersectionwithline} and \ref{qplusonesecant} that all points of $\Omega_2$ on these two $(q+1)$-secants have the same type, and both of them correspond to a hyperplane, say $\pi_L,\pi_M$ of $\Sigma$ containing $q+1$ lines extending to the points of $\Omega_2$ on $L$ and $M$. By Theorem \ref{intersectionwithline} and Lemma \ref{gammatype}, it follows that we can find four points, say $P_1,P_2,P_3,P_4$ spanning $\pi_L$, and a point $P_5\in \pi_M\setminus\pi_L$, and elements $\la_1,\la_2,\la_3,\la_4,\la_5\in \F_q$ and $\gamma\in\F_{q^{5}}\setminus\F_{q}$ such that $R=P_3+\gamma P_4$, $R'=P_1+\gamma P_2\in L$, and $Q=\gamma P_5+(\la_1P_1+\la_2P_2+\la_3P_3+\la_4P_4+\la_5P_5)\in M$.

We claim that we can always find a plane of $\Sigma$ extending to a plane meeting $\Pi$ in a line and hence, we have a $(q^2+q+1)$-secant to $\Omega_2$ in $\Pi$, contradicting the assumption in the lemma. For this it is sufficient to find a point in $\Sigma$ on two different lines that extend to lines meeting $\Pi$. We distinguish between two cases.
\par If $(\la_2+\lambda_{1}\la_5,\la_4+\la_3\la_5)=(0,0)$, then the point
\begin{align*}
	P&=-\la_5P_5+(\la_1P_1+\la_2P_2+\la_3P_3+\la_4P_4+\la_5P_5)\\
	&=\la_1(P_1-\la_5P_2)+\la_3(P_3-\la_5P_4)\\
	&=(\la_1P_1+\la_3P_3)-\la_5(\la_1P_2+\la_3P_4)
\end{align*}
lies on the line through  $\la_1P_1+\la_3P_3$ and $\la_1P_2+\la_3P_4$ containing the point $\lambda_{1}R'+\lambda_{3}R$ of $\Omega_2$. Furthermore, $P$ also lies on the line through $\la_1P_1+\la_2P_2+\la_3P_3+\la_4P_4+\la_5P_5$ and $P_5$ containing the point $Q$ of $\Omega_2$.
\par If $(\la_2+\lambda_{1}\la_5,\la_4+\la_3\la_5)\neq(0,0)$, then the point 
\[
	P'=(\la_2+\lambda_{1}\la_5)P_{1}+(\la_4+\la_3\la_5)P_{3}
\]
lies on the line through $P'$ and $(\la_2+\lambda_{1}\la_5)P_{2}+(\la_4+\la_3\la_5)P_{4}$ containing the point $(\la_2+\lambda_{1}\la_5)R'+(\la_4+\la_3\la_5)R$ of $\Omega_2$. Furthermore, it is easy to check that $P'$ also lies on the line through $P'$ and $P_5-\la_1P_{2}-\la_3P_{4}$ containing the point
\[
P'+\gamma(\gamma+\la_5)(\la_1P_{2}+\la_3P_{4}-P_5)=(\la_2+\la_1(\gamma+\la_5))R'+(\la_4+\la_3(\gamma+\la_5))R-\gamma Q\in\Pi\cap \Omega_2\;. \qedhere
\]
\end{proof}

\begin{lemma}\label{geentweekortesecanten}
	If $\Pi$ is a plane disjoint from $\Sigma$ that does not contain a $(q^{2}+q+1)$-secant to $\Omega_{2}$, then it contains at most one $(q+1)$-secant to $\Omega_{2}$. 
\end{lemma}
\begin{proof}
	Suppose that $L$ and $M$ are two $(q+1)$-secants to $\Omega_{2}$, with $L$ and $M$ lines in $\Pi$. Let $S$ be the intersection point of $L$ and $M$. From Lemma \ref{notwointersectingsublines} it follows that $S\notin\Omega_{2}$.
	\par By Theorems \ref{intersectionwithline} and \ref{qplusonesecant} and Lemma \ref{gammatype} there are $\gamma,\gamma'\in\F_{q^{5}}\setminus\F_{q}$ and points $Q_{1},Q_{2},Q_{3},Q_{4},Q'_{1},Q'_{2},Q'_{3},Q'_{4}\in\Sigma$ with $\dim\left\langle Q_{1},Q_{2},Q_{3},Q_{4}\right\rangle=3=\dim\left\langle Q'_{1},Q'_{2},Q'_{3},Q'_{4}\right\rangle$ such that $P_{1}=Q_{1}+\gamma Q_{2}$ and $P_{2}=Q_{3}+\gamma Q_{4}$ are points on $L$, and $P'_{1}=Q'_{1}+\gamma' Q'_{2}$ and $P'_{2}=Q'_{3}+\gamma'Q'_{4}$ are points on $M$. Note that $P_{1}$, $P_{2}$, $P'_{1}$ and $P'_{2}$ are points on $\Omega_{2}$. If the solids $\left\langle Q_{1},Q_{2},Q_{3},Q_{4}\right\rangle$ and $\left\langle Q'_{1},Q'_{2},Q'_{3},Q'_{4}\right\rangle$ would coincide,  $\Pi$ would be contained in the extension of this $3$-space so any plane of $\left\langle Q_{1},Q_{2},Q_{3},Q_{4}\right\rangle$ would give rise to a $(q^2+q+1)$-secant in $\Omega_2$, a contradiction. So at least one of the points $Q'_{i}$ is not contained in $\left\langle Q_{1},Q_{2},Q_{3},Q_{4}\right\rangle$. Without loss of generality, and by redefining $\gamma'$, we may assume $Q'_{2}\notin\left\langle Q_{1},Q_{2},Q_{3},Q_{4}\right\rangle$ and $Q'_{1}\in\left\langle Q_{1},Q_{2},Q_{3},Q_{4}\right\rangle$.
	\par Again without loss of generality, we can choose a basis for the underlying vector space such that $Q_{1}=\langle(1,0,0,0,0)\rangle$, $Q_{2}=\langle(0,1,0,0,0)\rangle$, $Q_{3}=\langle(0,0,1,0,0)\rangle$, $Q_{4}=\langle(0,0,0,1,0)\rangle$ and $Q'_{2}=\langle(0,0,0,0,1)\rangle$. There are $\lambda_{i}\in\F_{q}$, with $i=1,\dots4$ and $\mu_{i},\nu_{i}\in\F_{q}$, with $i=1,\dots,5$, such that $Q'_{1}=\langle(\lambda_{1},\lambda_{2},\lambda_{3},\lambda_{4},0)\rangle$, $Q'_{3}=\langle(\mu_{1},\mu_{2},\mu_{3},\mu_{4},\mu_{5})\rangle$ and $Q'_{4}=\langle(\nu_{1},\nu_{2},\nu_{3},\nu_{4},\nu_{5})\rangle$. It follows that 
	\begin{align*}
		P_{1}&=\langle(1,\gamma,0,0,0)\rangle\;,&P_{2}&=\langle(0,0,1,\gamma,0)\rangle\;,\\ P'_{1}&=\langle(\lambda_{1},\lambda_{2},\lambda_{3},\lambda_{4},\gamma')\rangle\;,& P'_{2}&=\langle(\mu_{1}+\nu_{1}\gamma',\mu_{2}+\nu_{2}\gamma',\mu_{3}+\nu_{3}\gamma',\mu_{4}+\nu_{4}\gamma',\mu_{5}+\nu_{5}\gamma')\rangle\;.
	\end{align*}
	\par Now, we calculate the coordinates of the point $S$. We know that there exist $\alpha_{1},\alpha_{2},\beta_{1}\in\F^{*}_{q^{5}}$ such that $\alpha_{1}P_{1}+\alpha_{2}P_{2}=\beta_{1}P'_{1}+P'_{2}$. We find the following system of equations (one equation for each coordinate):
	\begin{align}\label{eq:2keerq+1-vgl1}
		\begin{cases}
			\alpha_{1}=\lambda_{1}\beta_{1}+\mu_{1}+\nu_{1}\gamma'\\
			\alpha_{1}\gamma=\lambda_{2}\beta_{1}+\mu_{2}+\nu_{2}\gamma'\\
			\alpha_{2}=\lambda_{3}\beta_{1}+\mu_{3}+\nu_{3}\gamma'\\
			\alpha_{2}\gamma=\lambda_{4}\beta_{1}+\mu_{4}+\nu_{4}\gamma'\\
			0=\beta_{1}\gamma'+\mu_{5}+\nu_{5}\gamma'
		\end{cases}
		\Leftrightarrow\quad
		\begin{cases}
			\beta_{1}=-\frac{\mu_{5}+\nu_{5}\gamma'}{\gamma'}\\
			\alpha_{1}=-\lambda_{1}\frac{\mu_{5}+\nu_{5}\gamma'}{\gamma'}+\mu_{1}+\nu_{1}\gamma'\\
			\alpha_{2}=-\lambda_{3}\frac{\mu_{5}+\nu_{5}\gamma'}{\gamma'}+\mu_{3}+\nu_{3}\gamma'\\
			\alpha_{1}\gamma=-\lambda_{2}\frac{\mu_{5}+\nu_{5}\gamma'}{\gamma'}+\mu_{2}+\nu_{2}\gamma'\\
			\alpha_{2}\gamma=-\lambda_{4}\frac{\mu_{5}+\nu_{5}\gamma'}{\gamma'}+\mu_{4}+\nu_{4}\gamma'
		\end{cases}\;.
	\end{align}
	Note that $\gamma'\neq0$ since $\gamma'\notin\F_{q}$. The system of equations in \eqref{eq:2keerq+1-vgl1} has a solution in $\alpha_{1},\alpha_{2},\beta_{1}$ if $\gamma,\gamma'$ and $\lambda_{i},\mu_{i},\nu_{i}\in\F_{q}$, with $i=1,\dots,5$, fulfil
	\begin{align}\label{eq:2keerq+1-vgl2}
		&&&\begin{cases}
			\left(-\lambda_{1}\frac{\mu_{5}+\nu_{5}\gamma'}{\gamma'}+\mu_{1}+\nu_{1}\gamma'\right)\gamma=-\lambda_{2}\frac{\mu_{5}+\nu_{5}\gamma'}{\gamma'}+\mu_{2}+\nu_{2}\gamma'\\
			\left(-\lambda_{3}\frac{\mu_{5}+\nu_{5}\gamma'}{\gamma'}+\mu_{3}+\nu_{3}\gamma'\right)\gamma=-\lambda_{4}\frac{\mu_{5}+\nu_{5}\gamma'}{\gamma'}+\mu_{4}+\nu_{4}\gamma'
		\end{cases}\nonumber\\
		&\Leftrightarrow 
		&&\begin{cases}
			\left(\nu_{1}\gamma'^{2}+\left(\mu_{1}-\lambda_{1}\nu_{5}\right)\gamma'-\mu_{5}\lambda_{1}\right)\gamma=\nu_{2}\gamma'^{2}+\left(\mu_{2}-\lambda_{2}\nu_{5}\right)\gamma'-\mu_{5}\lambda_{2}\\
			\left(\nu_{3}\gamma'^{2}+\left(\mu_{3}-\lambda_{3}\nu_{5}\right)\gamma'-\mu_{5}\lambda_{3}\right)\gamma=\nu_{4}\gamma'^{2}+\left(\mu_{4}-\lambda_{4}\nu_{5}\right)\gamma'-\mu_{5}\lambda_{4}
		\end{cases}\;,
	\end{align}
	which can only be the case if 
	\begin{align*}
		&\left(\nu_{1}\gamma'^{2}+\left(\mu_{1}-\lambda_{1}\nu_{5}\right)\gamma'-\mu_{5}\lambda_{1}\right)\left(\nu_{4}\gamma'^{2}+\left(\mu_{4}-\lambda_{4}\nu_{5}\right)\gamma'-\mu_{5}\lambda_{4}\right)\\=\:&\left(\nu_{2}\gamma'^{2}+\left(\mu_{2}-\lambda_{2}\nu_{5}\right)\gamma'-\mu_{5}\lambda_{2}\right)	\left(\nu_{3}\gamma'^{2}+\left(\mu_{3}-\lambda_{3}\nu_{5}\right)\gamma'-\mu_{5}\lambda_{3}\right)\;,
	\end{align*}
	which is equivalent to
	\begin{align*}
		0=\:&\left[\nu_{1}\nu_{4}-\nu_{2}\nu_{3}\right]\gamma'^{4}+\left[\mu_{4}\nu_{1}+\mu_{1}\nu_{4}-\mu_{2}\nu_{3}-\mu_{3}\nu_{2}+\nu_{5}\left(\lambda_{2}\nu_{3}+\lambda_{3}\nu_{2}-\lambda_{1}\nu_{4}-\lambda_{4}\nu_{1}\right)\right]\gamma'^{3}\\
		&+\left[\mu_{1}\mu_{4}-\mu_{2}\mu_{3}+\nu_{5}\left(\lambda_{2}\mu_{3}+\lambda_{3}\mu_{2}-\lambda_{1}\mu_{4}-\lambda_{4}\mu_{1}\right)+\nu^{2}_{5}\left(\lambda_{1}\lambda_{4}-\lambda_{2}\lambda_{3}\right)\right.\\
		&\left.+\mu_{5}\left(\lambda_{2}\nu_{3}+\lambda_{3}\nu_{2}-\lambda_{1}\nu_{4}-\lambda_{4}\nu_{1}\right)\right]\gamma'^{2}+\left[\mu_{5}\left(\lambda_{2}\mu_{3}+\lambda_{3}\mu_{2}-\lambda_{1}\mu_{4}-\lambda_{4}\mu_{1}\right)\right.\\&\left.+2\mu_{5}\nu_{5}\left(\lambda_{1}\lambda_{4}-\lambda_{2}\lambda_{3}\right)\right]\gamma'+\mu^{2}_{5}\left(\lambda_{1}\lambda_{4}-\lambda_{2}\lambda_{3}\right)\;.
	\end{align*}
	However, as $\{1,\gamma',\gamma'^{2},\gamma'^{3},\gamma'^{4}\}$ is a linearly independent set over $\F_{q}$, we have that all coefficients in the right hand side of the previous expression equal zero. In particular, we have
	\begin{align}
		0&=\nu_{1}\nu_{4}-\nu_{2}\nu_{3}\;,\label{eq:2keerq+1-coeff1}\\
		0&=\mu_{4}\nu_{1}+\mu_{1}\nu_{4}-\mu_{2}\nu_{3}-\mu_{3}\nu_{2}+\nu_{5}\left(\lambda_{2}\nu_{3}+\lambda_{3}\nu_{2}-\lambda_{1}\nu_{4}-\lambda_{4}\nu_{1}\right)\;.\label{eq:2keerq+1-coeff2}
	\end{align}
	Now, subtracting $\nu_{1}$ times the second equation in \eqref{eq:2keerq+1-vgl2} from $\nu_{3}$ times the first, and using \eqref{eq:2keerq+1-coeff1}, we find
	\begin{align}\label{eq:2keerq+1-vriendje1}
		0&=\left(\mu_{1}\nu_{3}-\mu_{3}\nu_{1}+\nu_{5}\left(\lambda_{3}\nu_{1}-\lambda_{1}\nu_{3}\right)\right)\gamma\gamma'+\mu_{5}\left(\lambda_{3}\nu_{1}-\lambda_{1}\nu_{3}\right)\gamma\nonumber\\&\quad+\left(\mu_{4}\nu_{1}-\mu_{2}\nu_{3}+\nu_{5}\left(\lambda_{2}\nu_{3}-\lambda_{4}\nu_{1}\right)\right)\gamma'+\mu_{5}\left(\lambda_{2}\nu_{3}-\lambda_{4}\nu_{1}\right)\;.
	\end{align}
	Similarly, subtracting $\nu_{2}$ times the second equation in \eqref{eq:2keerq+1-vgl2} from $\nu_{4}$ times the first, and using \eqref{eq:2keerq+1-coeff1}, we find
	\begin{align}\label{eq:2keerq+1-vriendje2}
		0&=\left(\mu_{1}\nu_{4}-\mu_{3}\nu_{2}+\nu_{5}\left(\lambda_{3}\nu_{2}-\lambda_{1}\nu_{4}\right)\right)\gamma\gamma'+\mu_{5}\left(\lambda_{3}\nu_{2}-\lambda_{1}\nu_{4}\right)\gamma\nonumber\\&\quad+\left(\mu_{2}\nu_{4}-\mu_{4}\nu_{2}+\nu_{5}\left(\lambda_{4}\nu_{2}-\lambda_{2}\nu_{4}\right)\right)\gamma'+\mu_{5}\left(\lambda_{4}\nu_{2}-\lambda_{2}\nu_{4}\right)\;.
	\end{align}
	Unless all coefficients in both \eqref{eq:2keerq+1-vriendje1} and \eqref{eq:2keerq+1-vriendje2} equal zero, it follows from these equalities that $G(\gamma)=G(\gamma')$ by Lemma \ref{gammas}, and then by Theorem \ref{qplusonesecant} each line through a point of $L\cap\Omega_{2}$ and a point of $M\cap\Omega_{2}$ is a $(q+1)$-secant to $\Omega_{2}$, contradicting Lemma \ref{notwointersectingsublines}. So, we find that all coefficients in both \eqref{eq:2keerq+1-vriendje1} and \eqref{eq:2keerq+1-vriendje2} equal zero. In particular we find that
	\begin{align}
		0&=\mu_{4}\nu_{1}-\mu_{2}\nu_{3}+\nu_{5}\left(\lambda_{2}\nu_{3}-\lambda_{4}\nu_{1}\right)\;,\label{eq:2keerq+1-coeff3}\\
		0&=\mu_{5}\left(\lambda_{3}\nu_{2}-\lambda_{1}\nu_{4}\right)\;.\label{eq:2keerq+1-coeff4}
	\end{align}
	Considering the expression for $\alpha_1$ in \eqref{eq:2keerq+1-vgl1}, we deduce that
	\begin{align*}
		\nu_{4}\gamma'\alpha_{1}&=\nu_{1}\nu_{4}\gamma'^{2}+\left(\mu_{1}\nu_{4}-\lambda_{1}\nu_{4}\nu_{5}\right)\gamma'-\lambda_{1}\mu_{5}\nu_{4}\\
		&=\nu_{2}\nu_{3}\gamma'^{2}+\left(\mu_{2}\nu_{3}+\mu_{3}\nu_{2}-\mu_{4}\nu_{1}-\nu_{5}\left(\lambda_{2}\nu_{3}+\lambda_{3}\nu_{2}-\lambda_{4}\nu_{1}\right)\right)\gamma'-\lambda_{1}\mu_{5}\nu_{4}\\
		&=\nu_{2}\nu_{3}\gamma'^{2}+\left(\mu_{3}\nu_{2}-\lambda_{3}\nu_{2}\nu_{5}\right)\gamma'-\lambda_{3}\mu_{5}\nu_{2}\\
		&=\nu_{2}\gamma'\alpha_{2}\;.
	\end{align*}
	Here we used \eqref{eq:2keerq+1-coeff1} and \eqref{eq:2keerq+1-coeff2} in the first transition, \eqref{eq:2keerq+1-coeff3} and \eqref{eq:2keerq+1-coeff4} in the second transition and finally, the expression for $\alpha_2$ from \eqref{eq:2keerq+1-vgl1} in the last transition. We conclude that $\nu_{4}\alpha_{1}=\nu_{2}\alpha_{2}$ since $\gamma'\neq0$. Hence, the point $S$ is given by $\nu_{2}P_{1}+\nu_{4}P_{2}=(\nu_{2}Q_{1}+\nu_{4}Q_{3})+\gamma(\nu_{2}Q_{2}+\nu_{4}Q_{4})$. However, this implies that $S$ is a rank 2 point, contradicting the assumption that $S\notin\Omega_{2}$.
\end{proof}

\begin{theorem}\label{eriseensecant}
	If $\Pi$ is a plane disjoint from $\Sigma$ that contains a $(q+1)$-secant $L$ to $\Omega_{2}$, but that does not contain a $(q^{2}+q+1)$-secant to $\Omega_{2}$, then $q-1\leq |(\Pi\setminus L)\cap\Omega_{2}|\leq q+1$ and the points in $\Pi\cap\Omega_{2}$ not on $L$, form an arc. 
\end{theorem}
\begin{proof}
	By Theorems \ref{intersectionwithline} and \ref{qplusonesecant} and Lemma \ref{gammatype} there is a $\gamma\in\F_{q^{5}}\setminus\F_{q}$ and there are points $Q_{1},Q_{2},Q_{3},Q_{4}\in\Sigma$ with $\dim\left\langle Q_{1},Q_{2},Q_{3},Q_{4}\right\rangle=3$ such that $P_{1}=Q_{1}+\gamma Q_{2}$ and $P_{2}=Q_{3}+\gamma Q_{4}$ are points on $L$. Note that $P_{1}$ and $P_{2}$ are points on $\Omega_{2}$. Without loss of generality we can choose a basis for the underlying vector space such that $Q_{1}=\langle(1,0,0,0,0)\rangle$, $Q_{2}=\langle(0,1,0,0,0)\rangle$, $Q_{3}=\langle(0,0,1,0,0)\rangle$ and $Q_{4}=\langle(0,0,0,1,0)\rangle$. Let $Q_{5}$ be the point $\langle(0,0,0,0,1)\rangle$.
	\par The line $\left\langle Q_{1},Q_{3}\right\rangle$ cannot contain a point of $\Pi$ since the plane $\left\langle Q_{1},Q_{2},Q_{3}\right\rangle$ cannot give rise to a $(q^{2}+q+1)$-secant on $\Pi$ by the assumption (see Theorem \ref{intersectionwithline}). So, the plane $\left\langle Q_{1},Q_{3},Q_{5}\right\rangle$ meets $\Pi$ in a point, and there are $\delta_{1},\delta_{2}\in\F_{q^{5}}$ such that $P_{3}=\delta_{1}Q_{1}+\delta_{2}Q_{3}+Q_{5}$ is a point of $\Pi$.
	\par By the assumption we know that $L$ is a $(q+1)$-secant to $\Omega_{2}$, so we now look for points in $(\Pi\setminus L)\cap\Omega_{2}$. It is clear that any point $P$ of $\Omega_{2}$ can be written as $(\mu_{1},\mu_{2},\mu_{3},\mu_{4},\mu_{5})+\varphi(\nu_{1},\nu_{2},\nu_{3},\nu_{4},0)$ for some $\mu_{i},\nu_{i}\in\F_{q}$ and $\varphi\in\F_{q^{5}}\setminus\F_{q}$. Since we assume that $P$ is not on $L$ we know that $\mu_{5}\neq0$. Clearly, each point $P$ can in many ways be written as such a sum. However, it is easy to see that for each $P$ in $(\Pi\setminus L)\cap\Omega_{2}$ there are either unique $\mu_{i},\nu_{i}\in \F_{q}$, $i=1,2,3$, such that $P=\langle(\mu_{1},\mu_{2},\mu_{3},0,1)+\varphi(\nu_{1},\nu_{2},\nu_{3},1,0)\rangle$, or unique $\mu_{1},\mu_{2},\mu_{4},\nu_{1},\nu_{2}\in \F_{q}$ such that $P=\langle(\mu_{1},\mu_{2},0,\mu_{4},1)+\varphi(\nu_{1},\nu_{2},1,0,0)\rangle$. Here we used again that $\Pi$ cannot contain a $(q^{2}+q+1)$-secant. As the point $P$ is contained in $\Pi$ there are $\alpha_{1},\alpha_{2},\alpha_{3}\in\F_{q^{5}}$ such that $P=\alpha_{1}P_{1}+\alpha_{2}P_{2}+\alpha_{3}P_{3}$. Comparing both expressions for $P$ we find the following system of equations (one equation for each coordinate):
	\begin{align}\label{eq:q+1alg}
		\begin{cases}
			\alpha_{1}+\alpha_{3}\delta_{1}=\mu_{1}+\varphi\nu_{1}\\
			\alpha_{1}\gamma=\mu_{2}+\varphi\nu_{2}\\
			\alpha_{2}+\alpha_{3}\delta_{2}=\mu_{3}+\varphi\nu_{3}\\
			\alpha_{2}\gamma=\mu_{4}+\varphi\nu_{4}\\
			\alpha_{3}=1
		\end{cases}
		\Leftrightarrow\quad
		\begin{cases}
			\alpha_{1}=\mu_{1}+\varphi\nu_{1}-\delta_{1}\\
			\alpha_{2}=\mu_{3}+\varphi\nu_{3}-\delta_{2}\\
			\alpha_{3}=1\\
			\varphi(\nu_{1}\gamma-\nu_{2})=\delta_{1}\gamma-\mu_{1}\gamma+\mu_{2}\\
			\varphi(\nu_{3}\gamma-\nu_{4})=\delta_{2}\gamma-\mu_{3}\gamma+\mu_{4}
		\end{cases}\;.
	\end{align}
	Each solution in the $\alpha_{i}$'s, $\mu_{i}$'s, $\nu_{i}$'s and $\varphi$ with either $(\mu_{4},\nu_{4})=(0,1)$ or $(\mu_{3},\nu_{3},\nu_{4})=(0,1,0)$ of this system of equations corresponds to a unique point of $(\Pi\setminus L)\cap\Omega_{2}$. The first three equations in \eqref{eq:q+1alg} describe $\alpha_{1},\alpha_{2},\alpha_{3}$ as functions of the other unknowns so can be disregarded from now on. So we consider the system of equations
	\begin{align}\label{eq:q+1algbis}
		\begin{cases}
			\varphi(\nu_{1}\gamma-\nu_{2})=\delta_{1}\gamma-\mu_{1}\gamma+\mu_{2}\\	\varphi(\nu_{3}\gamma-\nu_{4})=\delta_{2}\gamma-\mu_{3}\gamma+\mu_{4}
		\end{cases}.
	\end{align}
	Note that $\nu_{3}$ or $\nu_{4}$ is nonzero, hence the coefficients of $\varphi$ in the equations cannot simultaneously be zero. Consequently, \eqref{eq:q+1algbis} has a unique solution in $\varphi$ iff
	\begin{align}\label{eq:q+1standaard}
		(\nu_{1}\gamma-\nu_{2})(\delta_{2}\gamma-\mu_{3}\gamma+\mu_{4})=(\nu_{3}\gamma-\nu_{4})(\delta_{1}\gamma-\mu_{1}\gamma+\mu_{2})\;,
	\end{align}
	and no solution if \eqref{eq:q+1standaard} is not satisfied. We only need to find the solutions of \eqref{eq:q+1standaard} in case either $(\mu_{4},\nu_{4})=(0,1)$ or $(\mu_{3},\nu_{3},\nu_{4})=(0,1,0)$, so we will look at the two following equations:
	\begin{align}
		-\gamma\delta_{1}&=\mu_{2}+(\mu_{3}\nu_{2}-\mu_{2}\nu_{3}-\mu_{1})\gamma+(\mu_{1}\nu_{3}-\mu_{3}\nu_{1})\gamma^{2}-\nu_{2}\gamma\delta_{2}+\nu_{1}\gamma^{2}\delta_{2}-\nu_{3}\gamma^{2}\delta_{1}\;,\label{eq:q+1standaardA}\\
		\gamma^{2}\delta_{1}&=-\mu_{4}\nu_{2}+(\mu_{4}\nu_{1}-\mu_{2})\gamma+\mu_{1}\gamma^{2}-\nu_{2}\gamma\delta_{2}+\nu_{1}\gamma^{2}\delta_{2} \;.\label{eq:q+1standaardB}
	\end{align}
	We conclude that we need to count the total number of solutions to \eqref{eq:q+1standaardA} and \eqref{eq:q+1standaardB} in order to find the number of points of $\Omega_2\cap (\Pi\setminus L)$. We will distinguish between several cases, depending on the relation between $\gamma$, $\delta_{1}$ and $\delta_{2}$, when discussing these equations.
	\paragraph*{Intermezzo:} Before analysing the two equations above, we will show that $$\dim\left\langle 1,\gamma,\gamma\delta_{1},\gamma\delta_{2}\right\rangle_q=4.$$ Assume to the contrary that $\dim\left\langle 1,\gamma,\gamma\delta_{1},\gamma\delta_{2}\right\rangle_q\leq 3$, then either there are $\lambda,\lambda_{3},\lambda_{4}\in\F_{q}$ such that $\gamma\delta_{2}=\lambda\gamma\delta_{1}+\lambda_{3}\gamma+\lambda_{4}$ or else there are $\lambda_{1},\lambda_{2}\in\F_{q}$ such that $\gamma\delta_{1}=\lambda_{1}\gamma+\lambda_{2}$. In the former case the point $\langle(\delta_{1},0,\lambda_{3}+\lambda\delta_{1},-\lambda_{4},1)\rangle=(\lambda_{3}+\lambda\delta_{1}-\delta_{2})P_{2}+P_{3}$ is a point of $\Omega_{2}\cap\Pi$. However, then the plane $\left\langle\langle(1,0,\lambda,0,0)\rangle,\langle(0,1,0,\lambda,0)\rangle,\langle(0,0,\lambda_{3},-\lambda_{4},1)\rangle\right\rangle$ contains two points of $\Pi$ -- next to the point just described, there is also $P_{1}+\lambda P_{2}=\langle(1,\gamma,\lambda,\lambda\gamma,0)\rangle$ on $L\cap \Omega_2$ --, hence meets $\Pi$ in a line and consequently gives rise to a $(q^{2}+q+1)$-secant in $\Pi$, a contradiction. In the latter case the point $\langle(\lambda_{1},-\lambda_{2},\delta_{2},0,1)\rangle=(\lambda_{1}-\delta_{1})P_{1}+P_{3}$ is a point of $\Omega_{2}\cap\Pi$. However, then the plane $\left\langle\langle(0,0,1,0,0)\rangle,\langle(0,0,0,1,0)\rangle,\langle(\lambda_{1},-\lambda_{2},0,0,1)\rangle\right\rangle$ contains two points of $\Pi$ -- apart from the point just described, there is also $P_2=\langle(0,0,1,\gamma,0)\rangle$ --, hence meets $\Pi$ in a line and consequently gives rise to a $(q^{2}+q+1)$-secant in $\Pi$, a contradiction.
\par We now distinguish between several cases and subcases. For cases A.1, A.2 and B.1.1 we present the details. The arguments in the other subcases are similar, and can be found in Appendix \ifthenelse{\equal{\versie}{arxiv}}{\ref{ap:th4.3}}{A in the arXiv version of this paper}.
	\paragraph*{Case A:} We assume that $\dim\left\langle 1,\gamma,\gamma^{2},\gamma\delta_{1},\gamma\delta_{2}\right\rangle_q=5$, in other words $\left\{1,\gamma,\gamma^{2},\gamma\delta_{1},\gamma\delta_{2}\right\}$ is an $\F_{q}$-basis for $\F_{q^{5}}$. Then there are $a_{i},b_{i}\in\F_{q}$, $i=1,\dots,5$, such that
	\begin{align}
		\gamma^{2}\delta_{1}&=b_{1}+b_{2}\gamma+b_{3}\gamma^{2}+b_{4}\gamma\delta_{1}+b_{5}\gamma\delta_{2}\quad\text{and}\label{gammakwadraatdelta_1}\\
		\gamma^{2}\delta_{2}&=a_{1}+a_{2}\gamma+a_{3}\gamma^{2}+a_{4}\gamma\delta_{1}+a_{5}\gamma\delta_{2}\label{gammakwadraatdelta_2}\;.
	\end{align}
	Note that $\dim\left\langle \gamma,\gamma^{2},\gamma^{2}\delta_{1},\gamma^{2}\delta_{2}\right\rangle_q=4$ since $\dim\left\langle 1,\gamma,\gamma\delta_{1},\gamma\delta_{2}\right\rangle_q=4$, and hence we have $\rk\left(\begin{smallmatrix}a_{1}&a_{4}&a_{5}\\b_{1}&b_{4}&b_{5}\end{smallmatrix}\right)=2$. We also show that it is not possible that simultaneously $a_{4}=b_{5}=0$ and $a_{5}=b_{4}\neq0$. Indeed, if $a_{4}=b_{5}=0$ and $a_{5}=b_{4}\neq0$, then
	\begin{align*}
		0&=\left(a_{1}b^{-1}_{4}+a_{2}+a_{3}b_{4}\right)\left(\gamma^{2}\delta_{1}-b_{1}-b_{2}\gamma-b_{3}\gamma^{2}-b_{4}\gamma\delta_{1}\right)\\&\qquad-\left(b_{1}b^{-1}_{4}+b_{2}+b_{3}b_{4}\right)\left(\gamma^{2}\delta_{2}-a_{1}-a_{2}\gamma-a_{3}\gamma^{2}-b_{4}\gamma\delta_{2}\right)\\
		&=\left(\gamma-b_{4}\right)\left[\left(a_{1}b^{-1}_{4}+a_{2}+a_{3}b_{4}\right)\gamma\delta_{1}-\left(b_{1}b^{-1}_{4}+b_{2}+b_{3}b_{4}\right)\gamma\delta_{2}\right.\\&\qquad\left.+\left(\left(a_{3}b_{1}-a_{1}b_{3}\right)b^{-1}_{4}+a_{3}b_{2}-a_{2}b_{3}\right)\gamma+a_{3}b_{1}-a_{1}b_{3}+\left(a_{2}b_{1}-a_{1}b_{2}\right)b^{-1}_{4}\right]\;.
	\end{align*}
	The first factor in this expression cannot be zero as $\gamma\notin\F_{q}$. Hence, the second factor must be zero. However, as $\{1,\gamma,\gamma\delta_{1},\gamma\delta_{2}\}$ is independent over $\F_{q}$, we have that $a_{1}+a_{2}b_{4}+a_{3}b^{2}_{4}=0=b_{1}+b_{2}b_{4}+b_{3}b^{2}_{4}$ (and two more equalities but they are redundant). But, then it also follows that
	\begin{align*}
		0&=\left(\gamma^{2}\delta_{1}-b_{1}-b_{2}\gamma-b_{3}\gamma^{2}-b_{4}\gamma\delta_{1}\right)+\left(b_{1}+b_{2}b_{4}+b_{3}b^{2}_{4}\right)\\
		&=\left(\gamma-b_{4}\right)\left(\gamma\delta_{1}-b_{2}-b_{3}(\gamma+b_{4})\right)\;.
	\end{align*}
	The first factor in this expression cannot be zero as $\gamma\notin\F_{q}$, and the second factor cannot be zero as $\{1,\gamma,\gamma\delta_{1}\}$ is independent over $\F_{q}$. So, we find a contradiction, and we conclude that it is not possible that simultaneously $a_{4}=b_{5}=0$ and $a_{5}=b_{4}\neq0$.
	\par Since $\left\{1,\gamma,\gamma^{2},\gamma\delta_{1},\gamma\delta_{2}\right\}$ is a linearly independent set over $\F_{q}$, Equation \eqref{eq:q+1standaardA} is equivalent to the following system of equations:
	\begin{align}\label{eq:q+1standaardA1}
		\begin{cases}
			0=\mu_{2}+a_{1}\nu_{1}-b_{1}\nu_{3}\\
			0=\mu_{3}\nu_{2}-\mu_{2}\nu_{3}-\mu_{1}+a_{2}\nu_{1}-b_{2}\nu_{3}\\
			0=\mu_{1}\nu_{3}-\mu_{3}\nu_{1}+a_{3}\nu_{1}-b_{3}\nu_{3}\\
			0=a_{4}\nu_{1}-b_{4}\nu_{3}+1\\
			0=-\nu_{2}+a_{5}\nu_{1}-b_{5}\nu_{3}
		\end{cases}
		\Leftrightarrow\quad
		\begin{cases}
			\mu_{2}=b_{1}\nu_{3}-a_{1}\nu_{1}\\
			\nu_{2}=a_{5}\nu_{1}-b_{5}\nu_{3}\\
			0=\mu_{3}\nu_{2}-\mu_{2}\nu_{3}-\mu_{1}+a_{2}\nu_{1}-b_{2}\nu_{3}\\
			0=\mu_{1}\nu_{3}-\mu_{3}\nu_{1}+a_{3}\nu_{1}-b_{3}\nu_{3}\\
			0=a_{4}\nu_{1}-b_{4}\nu_{3}+1
		\end{cases}.
	\end{align}
	
	It is straightforward to see that there is a one-to-one correspondence between the solutions in $(\mu_{1},\mu_{2},\mu_{3},\nu_{1},\nu_{2},\nu_{3})$ of Equation \eqref{eq:q+1standaardA1} and the solutions in $(\mu_{1},\mu_{3},\nu_{1},\nu_{3})$ of
	\begin{align}\label{eq:q+1standaardA1bis}
		\begin{cases}
			0=\mu_{3}(a_{5}\nu_{1}-b_{5}\nu_{3})-(b_{1}\nu_{3}-a_{1}\nu_{1})\nu_{3}-\mu_{1}+a_{2}\nu_{1}-b_{2}\nu_{3}\\
			0=\mu_{1}\nu_{3}-\mu_{3}\nu_{1}+a_{3}\nu_{1}-b_{3}\nu_{3}\\
			0=a_{4}\nu_{1}-b_{4}\nu_{3}+1
		\end{cases}.
	\end{align}
	Equation \eqref{eq:q+1standaardB} is equivalent to the following system of equations:
	\begin{align}\label{eq:q+1standaardB1}
		\begin{cases}
			0=-\mu_{4}\nu_{2}+a_{1}\nu_{1}-b_{1}\\
			0=\mu_{4}\nu_{1}-\mu_{2}+a_{2}\nu_{1}-b_{2}\\
			0=\mu_{1}+a_{3}\nu_{1}-b_{3}\\
			0=a_{4}\nu_{1}-b_{4}\\
			0=-\nu_{2}+a_{5}\nu_{1}-b_{5}
		\end{cases}
		\Leftrightarrow\quad
		\begin{cases}
			\mu_{1}=b_{3}-a_{3}\nu_{1}\\
			\mu_{2}=\mu_{4}\nu_{1}+a_{2}\nu_{1}-b_{2}\\
			\nu_{2}=a_{5}\nu_{1}-b_{5}\\
			0=-\mu_{4}\nu_{2}+a_{1}\nu_{1}-b_{1}\\
			0=a_{4}\nu_{1}-b_{4}
		\end{cases}.
	\end{align}
	Again, it is straightforward that there is a one-to-one correspondence between the solutions in $(\mu_{1},\mu_{2},\mu_{4},\nu_{1},\nu_{2})$ of Equation \eqref{eq:q+1standaardB1} and the solutions in $(\mu_{4},\nu_{1})$ of
	\begin{align}\label{eq:q+1standaardB1bis}
		\begin{cases}
			0=\mu_{4}(b_{5}-a_{5}\nu_{1})+a_{1}\nu_{1}-b_{1}\\
			0=a_{4}\nu_{1}-b_{4}
		\end{cases}.
	\end{align}
	We see that in this case, we have reduced the problem of counting the solutions to \eqref{eq:q+1standaardA} and \eqref{eq:q+1standaardB}, to the problem to counting the number of solutions to \eqref{eq:q+1standaardA1bis} and \eqref{eq:q+1standaardB1bis}. Yet again, we will discuss several cases, now depending on the couple $(a_{4},b_{4})$.
	\par \textit{Case A.1: $(a_{4},b_{4})=(0,0)$.} 	 In this case Equation \eqref{eq:q+1standaardA1bis} clearly has no solutions; Equation \eqref{eq:q+1standaardB1bis} reduces to the equation of a conic in the $(\mu_{4},\nu_{1})$-plane $\pi\cong\AG(2,q)$ with line at infinity $\ell_{\infty}$. This conic is either a non-degenerate conic which has two points on $\ell_{\infty}$ (if $a_{5}\neq0$), or a degenerate conic consisting of one affine line and $\ell_{\infty}$ (if $a_{5}=0$). Since a conic in $\PG(2,q)$ has $q+1$ points, we find that Equation \eqref{eq:q+1standaardB1bis} has $q-1$ or $q$ solutions. So, we find $q-1$ or $q$ points in $(\Pi\setminus L)\cap\Omega_{2}$ in this case.
	\par \textit{Case A.2: $a_{4}\neq0$.} Clearly, in this case Equation \eqref{eq:q+1standaardB1bis} has 0, 1 or $q$ solutions and it can only have $q$ solutions if $\left|\begin{smallmatrix}a_{5}&a_{4}\\b_{5}&b_{4}\end{smallmatrix}\right|=0=\left|\begin{smallmatrix}a_{1}&a_{4}\\b_{1}&b_{4}\end{smallmatrix}\right|$, hence if $\rk\left(\begin{smallmatrix}a_{1}&a_{4}&a_{5}\\b_{1}&b_{4}&b_{5}\end{smallmatrix}\right)=1$, a contradiction. So, Equation \eqref{eq:q+1standaardB1bis} has 0 solutions or 1 solution in this case. Note that it can only have 0 solutions if $\left|\begin{smallmatrix}a_{5}&a_{4}\\b_{5}&b_{4}\end{smallmatrix}\right|=0$ and simultaneously $\left|\begin{smallmatrix}a_{1}&a_{4}\\b_{1}&b_{4}\end{smallmatrix}\right|\neq0$.
	\par We now look at Equation \eqref{eq:q+1standaardA1bis}. From the third equation in \eqref{eq:q+1standaardA1bis} we then have $\nu_{1}=\frac{b_{4}\nu_{3}-1}{a_{4}}$, so $\nu_{1}$ is uniquely determined by $\nu_{3}$, and we can look at the following system of equations:
	\begin{align}\label{eq:q+1standaardA1.2}
		\begin{cases}
			\mu_{1}-\mu_{3}\left(a_{5}\frac{b_{4}\nu_{3}-1}{a_{4}}-b_{5}\nu_{3}\right)=-\left(b_{1}\nu_{3}-a_{1}\frac{b_{4}\nu_{3}-1}{a_{4}}\right)\nu_{3}+a_{2}\frac{b_{4}\nu_{3}-1}{a_{4}}-b_{2}\nu_{3}\\
			\mu_{1}\nu_{3}-\mu_{3}\frac{b_{4}\nu_{3}-1}{a_{4}}=b_{3}\nu_{3}-a_{3}\frac{b_{4}\nu_{3}-1}{a_{4}}
		\end{cases}.
	\end{align}
	For a given value of $\nu_{3}$ Equation \eqref{eq:q+1standaardA1.2} is a linear system of equations in $\mu_{1}$ and $\mu_{3}$ and has either 0, 1 or $q$ solutions. It has 0 or $q$ solutions iff
	\begin{align}\label{eq:q+1-1.2ontaard}
		\left(a_{5}\frac{b_{4}\nu_{3}-1}{a_{4}}-b_{5}\nu_{3}\right)\nu_{3}-\frac{b_{4}\nu_{3}-1}{a_{4}}=0\quad\Leftrightarrow\quad(a_{5}b_{4}-a_{4}b_{5})\nu^{2}_{3}-(a_{5}+b_{4})\nu_{3}+1=0\;.
	\end{align}
	This is a non-vanishing quadratic or linear equation, so it has at most two solutions, hence for at least $q-2$ values of $\nu_{3}$ Equation \eqref{eq:q+1standaardA1.2} has precisely one solution. If $\nu$ is a solution of \eqref{eq:q+1-1.2ontaard}, then $\nu\neq0$. For $\nu_{3}=\nu$ Equation \eqref{eq:q+1standaardA1.2} has $q$ solutions iff
	\begin{align}
		&&&b_{3}\nu-a_{3}\frac{b_{4}\nu-1}{a_{4}}=\nu\left(-\left(b_{1}\nu-a_{1}\frac{b_{4}\nu-1}{a_{4}}\right)\nu+a_{2}\frac{b_{4}\nu-1}{a_{4}}-b_{2}\nu\right)\nonumber\\
		&\Leftrightarrow&0&=\left(a_{1}b_{4}-a_{4}b_{1}\right)\nu^{3}+\left(a_{2}b_{4}-a_{4}b_{2}-a_{1}\right)\nu^{2}+\left(a_{3}b_{4}-a_{4}b_{3}-a_{2}\right)\nu-a_{3}\nonumber\\
		&\Leftrightarrow& 0&=\left(a_{1}b_{4}-a_{4}b_{1}\right)+\left(a_{2}b_{4}-a_{4}b_{2}-a_{1}\right)\nu^{-1}+\left(a_{3}b_{4}-a_{4}b_{3}-a_{2}\right)\nu^{-2}-a_{3}\nu^{-3}\;.\label{eq:q+1-1.2qopl}
	\end{align}	
	Now, we also have that
	\begin{align*}
		&&0&=\gamma\delta_{2}(\gamma-a_{5})-a_{1}-a_{2}\gamma-a_{3}\gamma^{2}-a_{4}\gamma\delta_{1}\\
		&\Leftrightarrow&0&=\gamma\delta_{2}(\gamma-a_{5})(\gamma-b_{4})-(a_{1}+a_{2}\gamma+a_{3}\gamma^{2})(\gamma-b_{4})-a_{4}\gamma(\gamma-b_{4})\delta_{1}\\
		&&&=\gamma\delta_{2}\left((\gamma-a_{5})(\gamma-b_{4})-a_{4}b_{5}\right)-(a_{1}+a_{2}\gamma+a_{3}\gamma^{2})(\gamma-b_{4})-a_{4}\left(b_{1}+b_{2}\gamma+b_{3}\gamma^{2}\right)
	\end{align*}
	where we have used \eqref{gammakwadraatdelta_1}. Now subtracting \eqref{eq:q+1-1.2qopl} from this, we find that
	\begin{align}
		0&=\gamma\delta_{2}\left(\gamma^{2}-(a_{5}+b_{4})\gamma+a_{5}b_{4}-a_{4}b_{5}\right)+\left(a_{2}b_{4}-a_{4}b_{2}-a_{1}\right)\left(\gamma-\nu^{-1}\right)\nonumber\\&\qquad+\left(a_{3}b_{4}-a_{4}b_{3}-a_{2}\right)\left(\gamma^{2}-\nu^{-2}\right)-a_{3}\left(\gamma^{3}-\nu^{-3}\right)\nonumber\\
		&=\left(\gamma-\nu^{-1}\right)\left[\gamma\delta_{2}\left(\gamma-(a_{5}+b_{4})+\nu^{-1}\right)+\left(a_{2}b_{4}-a_{4}b_{2}-a_{1}\right)\right.\nonumber\\&\qquad\left.+\left(a_{3}b_{4}-a_{4}b_{3}-a_{2}\right)\left(\gamma+\nu^{-1}\right)-a_{3}\left(\gamma^{2}+\nu^{-1}\gamma+\nu^{-2}\right)\right]\;.\label{eq:q+1-1.2qoplbis}
	\end{align}
	The first factor in \eqref{eq:q+1-1.2qoplbis} cannot be zero as $\gamma\notin\F_{q}$. Hence, the second factor in \eqref{eq:q+1-1.2qoplbis} must be zero. It follows that
	\begin{align*}
		0&=\gamma^{2}\delta_{2}-(a_{5}+b_{4}-\nu^{-1})\gamma\delta_{2}+\left(a_{2}b_{4}-a_{4}b_{2}-a_{1}+\left(a_{3}b_{4}-a_{4}b_{3}-a_{2}\right)\nu^{-1}-a_{3}\nu^{-2}\right)\\&\qquad+\left(a_{3}b_{4}-a_{4}b_{3}-a_{2}-a_{3}\nu^{-1}\right)\gamma-a_{3}\gamma^{2}\\
		&=\left(a_{2}b_{4}-a_{4}b_{2}+\left(a_{3}b_{4}-a_{4}b_{3}-a_{2}\right)\nu^{-1}-a_{3}\nu^{-2}\right)+\left(a_{3}b_{4}-a_{4}b_{3}-a_{3}\nu^{-1}\right)\gamma\\&\qquad+a_{4}\gamma\delta_{1}-(a_{5}+b_{4}-\nu^{-1})\gamma\delta_{2}\;.
	\end{align*}
	However, as $\{1,\gamma,\gamma\delta_{1},\gamma\delta_{2}\}$ is independent over $\F_{q}$, we have that $a_{4}=0$, contradicting our assumption. So, if $\nu_{3}$ is a solution of \eqref{eq:q+1-1.2ontaard}, then \eqref{eq:q+1standaardA1.2} has no solutions.
	\par We conclude that in Case A.2, we have $q-2$, $q-1$ or $q$ solutions of Equation \eqref{eq:q+1standaardA1bis} and at most 1 solution of Equation \eqref{eq:q+1standaardB1bis}. Now, recall that \eqref{eq:q+1standaardB1bis} has no solutions if and only if $\left|\begin{smallmatrix}a_{5}&a_{4}\\b_{5}&b_{4}\end{smallmatrix}\right|=0$ and $\left|\begin{smallmatrix}a_{1}&a_{4}\\b_{1}&b_{4}\end{smallmatrix}\right|\neq0$, but in this case Equation \eqref{eq:q+1-1.2ontaard} is linear, and so there are at least $q-1$ solutions of Equation \eqref{eq:q+1standaardA1bis}. So, in total there are at least $q-1$ and at most $q+1$ points in $(\Pi\setminus L)\cap\Omega_{2}$ in this case.
	\par \textit{Case A.3: $a_{4}=0$ and $b_{4}\neq0$.} The arguments in this case are similar to the arguments in Case A.2. We find that there are $q-1$ or $q$ solutions of Equation \eqref{eq:q+1standaardA1bis} and no solutions of Equation \eqref{eq:q+1standaardB1bis}, so in total there are $q-1$ or $q$ points in $(\Pi\setminus L)\cap\Omega_{2}$. Details can be found in Appendix \ifthenelse{\equal{\versie}{arxiv}}{\ref{ap:th4.3}, see page \pageref{apA:A3}}{A in the arXiv version of this paper}.
	\paragraph*{Case B:} Now, we assume that $\dim\left\langle 1,\gamma,\gamma^2,\gamma\delta_{1},\gamma\delta_{2}\right\rangle_q\neq 5$. We showed in the beginning of the proof that $\dim\left\langle 1,\gamma,\gamma\delta_{1},\gamma\delta_{2}\right\rangle_q=4$. So, we have that $\gamma^{2}\in\left\langle 1,\gamma,\gamma\delta_{1},\gamma\delta_{2}\right\rangle_q$. Since $\gamma^{2}\notin\left\langle1,\gamma\right\rangle$, we know that $\gamma\delta_{1}\in\left\langle 1,\gamma,\gamma^{2},\gamma\delta_{2}\right\rangle_q$ or $\gamma\delta_{2}\in\left\langle 1,\gamma,\gamma^{2},\gamma\delta_{1}\right\rangle_q$.
	\par \textit{Case B.1: $\gamma\delta_{1}\in\left\langle 1,\gamma,\gamma^{2},\gamma\delta_{2}\right\rangle_q$.} By the assumption we can find $c_{0},c_{1},c_{2},c_{3}\in\F_{q}$ such that
	\begin{align}
		\gamma\delta_{1}=c_{0}+c_{1}\gamma+c_{2}\gamma^{2}+c_{3}\gamma\delta_{2}\label{gammadelta1B1}\;,
	\end{align}
	and we know that $c_{2}\neq0$ since $\dim\left\langle 1,\gamma,\gamma\delta_{1},\gamma\delta_{2}\right\rangle_q=4$. We now show that it is not possible that both $\gamma^{2}\delta_{1}$ and $\gamma^{2}\delta_{2}$ are contained in $\left\langle 1,\gamma,\gamma^{2},\gamma\delta_{2}\right\rangle_q$. Assume to the contrary they are, and let $a_{0},a_{1},a_{2},a_{3}\in\F_{q}$ and $b_{0},b_{1},b_{2},b_{3}\in\F_{q}$ be such that
	\begin{align*}
		\gamma^{2}\delta_{1}&=b_{0}+b_{1}\gamma+b_{2}\gamma^{2}+b_{3}\gamma\delta_{2}\;,\\
		\gamma^{2}\delta_{2}&=a_{0}+a_{1}\gamma+a_{2}\gamma^{2}+a_{3}\gamma\delta_{2}\;.
	\end{align*}
	It follows immediately that
	\begin{align*}
		0&=\gamma^{2}\delta_{1}-\gamma(\gamma\delta_{1})=b_{0}+(b_{1}-c_{0})\gamma+(b_{2}-c_{1})\gamma^{2}+b_{3}\gamma\delta_{2}-c_{2}\gamma^{3}-c_{3}\gamma^{2}\delta_{2}\\
		&=(b_{0}-c_{3}a_{0})+(b_{1}-c_{0}-c_{3}a_{1})\gamma+(b_{2}-c_{1}-c_{3}a_{2})\gamma^{2}+(b_{3}-c_{3}a_{3})\gamma\delta_{2}-c_{2}\gamma^{3}\;.
	\end{align*}
	This is a non-vanishing expression since $c_{2}\neq0$. So, as $\{1,\gamma,\gamma^{2},\gamma^{3}\}$ is a linearly independent set over $\F_{q}$, it follows that $\gamma\delta_{2}\in\left\langle1,\gamma,\gamma^{2},\gamma^{3}\right\rangle_q$. More precisely, $\gamma\delta_{2}\in\left\langle1,\gamma,\gamma^{2},\gamma^{3}\right\rangle_q\setminus\left\langle1,\gamma,\gamma^{2}\right\rangle_q$ since $c_{2}\neq0$. However, then $\gamma^{2}\delta_{2}\in\left\langle\gamma,\gamma^{2},\gamma^{3},\gamma^{4}\right\rangle_q\setminus\left\langle\gamma,\gamma^{2},\gamma^{3}\right\rangle_q$. Since 
	$\gamma^{2}\delta_{2}=a_{0}+a_{1}\gamma+a_{2}\gamma^{2}+a_{3}\gamma\delta_{2}$ and $\gamma\delta_{2}\in\left\langle1,\gamma,\gamma^{2},\gamma^{3}\right\rangle_q\setminus\left\langle1,\gamma,\gamma^{2}\right\rangle_q$, it follows that $\gamma^{2}\delta_{2}\in\left\langle1,\gamma,\gamma^{2},\gamma^{3}\right\rangle_q$, a contradiction since $\{1,\gamma,\gamma^{2},\gamma^{3},\gamma^{4}\}$ is a linearly independent set over $\F_{q}$.
	\par Having excluded the possibility that both $\gamma^{2}\delta_{1}$ and $\gamma^{2}\delta_{2}$ are contained in $\left\langle 1,\gamma,\gamma^{2},\gamma\delta_{2}\right\rangle_q$, we now distinguish between two cases.
	\par \textit{Case B.1.1: $\dim\left\langle 1,\gamma,\gamma^{2},\gamma\delta_{2},\gamma^{2}\delta_{1}\right\rangle_q=5$.} In other words, $\left\{1,\gamma,\gamma^{2},\gamma\delta_{2},\gamma^{2}\delta_{1}\right\}$ is an $\F_{q}$-basis for $\F_{q^{5}}$. Then, there are $a_{i}\in\F_{q}$, $i=1,\dots,5$, such that
	\begin{align}
		\gamma^{2}\delta_{2}&=a_{0}+a_{1}\gamma+a_{2}\gamma^{2}+a_{3}\gamma\delta_{2}+a_{4}\gamma^{2}\delta_{1}\label{gammakwadraatdelta2B11}\;.
	\end{align}
	Note that $\dim\left\langle \gamma,\gamma^{2},\gamma^{2}\delta_{1},\gamma^{2}\delta_{2}\right\rangle_q=4$ since $\dim\left\langle 1,\gamma,\gamma\delta_{1},\gamma\delta_{2}\right\rangle_q=4$, and hence $(a_{0},a_{3})\neq(0,0)$. We also show that it is not possible that simultaneously $a_{3}=0$ and $a_{4}c_{3}=1$. Indeed, if $a_{3}=0$ and $a_{4}c_{3}=1$, then
	\begin{align*}
		a_{0}+a_{1}\gamma+a_{2}\gamma^{2}=\gamma^{2}\left(\delta_{2}-c^{-1}_{3}\delta_{1}\right)=c^{-1}_{3}\gamma\left(\gamma\left(c_{3}\delta_{2}-\delta_{1}\right)\right)=-c^{-1}_{3}\gamma(c_{0}+c_{1}\gamma+c_{2}\gamma^{2})\;.
	\end{align*}
	However, as $\{1,\gamma,\gamma^{2},\gamma^{3}\}$ is independent over $\F_{q}$, we then find that $a_{0}=0$, a contradiction since $a_{3}=0$.
	\par We are now ready to discuss the number of solutions to \eqref{eq:q+1standaardA} and \eqref{eq:q+1standaardB} in Case B.1.1. We see that Equation \eqref{eq:q+1standaardB}, is equivalent to the following system of equations:
	\begin{align}\label{eq:q+1standaardB211}
		\begin{cases}
			0=-\mu_{4}\nu_{2}+a_{0}\nu_{1}\\
			0=\mu_{4}\nu_{1}-\mu_{2}+a_{1}\nu_{1}\\
			0=\mu_{1}+a_{2}\nu_{1}\\
			0=-\nu_{2}+a_{3}\nu_{1}\\
			1=a_{4}\nu_{1}
		\end{cases}
		\Leftrightarrow\quad
		\begin{cases}
			\mu_{1}=-a_{2}\nu_{1}\\
			\mu_{2}=\mu_{4}\nu_{1}+a_{1}\nu_{1}\\
			\nu_{2}=a_{3}\nu_{1}\\
			0=a_{0}\nu_{1}-a_{3}\mu_{4}\nu_{1}\\
			1=a_{4}\nu_{1}
		\end{cases}.
	\end{align}
	Again, it is straightforward that there is a one-to-one correspondence between the solutions in $(\mu_{1},\mu_{2},\mu_{4},\nu_{1},\nu_{2})$ of Equation \eqref{eq:q+1standaardB211} and the solutions in $(\mu_{4},\nu_{1})$ of
	\begin{align}\label{eq:q+1standaardB211bis}
		\begin{cases}
			0=\nu_{1}(a_{0}-a_{3}\mu_{4})\\
			1=a_{4}\nu_{1}
		\end{cases}
		\Leftrightarrow\quad
		\begin{cases}
			0=a_{0}-a_{3}\mu_{4}\\
			1=a_{4}\nu_{1}
		\end{cases}.
	\end{align}
	Clearly, Equation \eqref{eq:q+1standaardB211bis} has 0, 1 or $q$ solutions and it can only have $q$ solutions if $a_{0}=a_{3}=0$, a contradiction. So, Equation \eqref{eq:q+1standaardB211bis} has 0 solutions or 1 solution in this case, and it only has 0 solutions if $a_{3}=0$ and $a_{0}\neq0$ or if $a_{4}=0$. 	\par Equation \eqref{eq:q+1standaardA} is equivalent to the following system of equations:
	\begin{align}\label{eq:q+1standaardA211}
		&\begin{cases}
			0=\mu_{2}+a_{0}\nu_{1}+c_{0}\\
			0=\mu_{3}\nu_{2}-\mu_{2}\nu_{3}-\mu_{1}+a_{1}\nu_{1}+c_{1}\\
			0=\mu_{1}\nu_{3}-\mu_{3}\nu_{1}+a_{2}\nu_{1}+c_{2}\\
			0=-\nu_{2}+a_{3}\nu_{1}+c_{3}\\
			0=a_{4}\nu_{1}-\nu_{3}
		\end{cases}\nonumber\\
		\Leftrightarrow\quad&
		\begin{cases}
			\mu_{2}=-a_{0}\nu_{1}-c_{0}\\
			\nu_{2}=a_{3}\nu_{1}+c_{3}\\
			\nu_{3}=a_{4}\nu_{1}\\
			\mu_{1}=\mu_{3}(a_{3}\nu_{1}+c_{3})+(a_{0}\nu_{1}+c_{0})a_{4}\nu_{1}+a_{1}\nu_{1}+c_{1}\\
			0=\mu_{1}a_{4}\nu_{1}-\mu_{3}\nu_{1}+a_{2}\nu_{1}+c_{2}
		\end{cases}\!\!.
	\end{align}
	 It is straightforward to see that there is a one-to-one correspondence between the solutions in $(\mu_{1},\mu_{2},\mu_{3},\nu_{1},\nu_{2},\nu_{3})$ of Equation \eqref{eq:q+1standaardA211} and the solutions in $(\mu_{1},\mu_{3},\nu_{1})$ of
	\begin{align}\label{eq:q+1standaardA211bis}
		\begin{cases}
			\mu_{1}-\mu_{3}(a_{3}\nu_{1}+c_{3})=(a_{0}\nu_{1}+c_{0})a_{4}\nu_{1}+a_{1}\nu_{1}+c_{1}\\
			\mu_{1}a_{4}\nu_{1}-\mu_{3}\nu_{1}=-a_{2}\nu_{1}-c_{2}
		\end{cases}.
	\end{align}
	For a given value of $\nu_{1}$ Equation \eqref{eq:q+1standaardA211bis} is a linear system of equations in $\mu_{1}$ and $\mu_{3}$ and has either 0, 1 or $q$ solutions. It has 0 or $q$ solutions iff
	\begin{align}\label{eq:q+1-2.1.1ontaard}
		\left(a_{3}\nu_{1}+c_{3}\right)a_{4}\nu_{1}-\nu_{1}=0\quad\Leftrightarrow\quad\nu_{1}\left(a_{3}a_{4}\nu_{1}+c_{3}a_{4}-1\right)=0\;.
	\end{align}
	This is a non-vanishing quadratic or linear equation since it is not possible that simultaneously $a_{3}=0$ and $a_{4}c_{3}=1$. More precisely, \eqref{eq:q+1-2.1.1ontaard} has solutions $\nu_{1}=0$ and $\nu_{1}=\frac{1-a_{4}c_{3}}{a_{3}a_{4}}$. Note that the latter solution only exists if $a_{3}\neq0\neq a_{4}$. Hence for $q-2$ or $q-1$ values of $\nu_{3}$ Equation \eqref{eq:q+1standaardA211bis} has precisely one solution. Denote $\frac{1-a_{4}c_{3}}{a_{3}a_{4}}$ by $\nu$. If $\nu_{1}=0$, then \eqref{eq:q+1standaardA211bis} has no solutions since $c_{2}\neq0$. If $\nu_{1}=\nu\neq0$, then \eqref{eq:q+1standaardA211bis} has $q$ solutions iff
	\begin{align}
		&&-a_{2}\nu-c_{2}&=a_{4}\nu\left((a_{0}\nu+c_{0})a_{4}\nu+a_{1}\nu+c_{1}\right)\nonumber\\
		&\Leftrightarrow& 0&=a_{0}a^{2}_{4}\nu^{3}+a_{4}\left(a_{1}+a_{4}c_{0}\right)\nu^{2}+\left(a_{2}+a_{4}c_{1}\right)\nu+c_{2}\nonumber\\
		&\Leftrightarrow& 0&=a_{0}+\left(a_{1}+a_{4}c_{0}\right)(a_{4}\nu)^{-1}+\left(a_{2}+a_{4}c_{1}\right)(a_{4}\nu)^{-2}+a_{4}c_{2}(a_{4}\nu)^{-3}\;.\label{eq:q+1-2.1.1qopl}
	\end{align}
	Now, from \eqref{gammadelta1B1} and \eqref{gammakwadraatdelta2B11}, we also have that
	\begin{align*}
		0&=a_{0}+a_{1}\gamma+a_{2}\gamma^{2}+a_{3}\gamma\delta_{2}+a_{4}\gamma(\gamma\delta_{1})-\gamma^{2}\delta_{2}\\
		&=a_{0}+(a_{1}+a_{4}c_{0})\gamma+(a_{2}+a_{4}c_{1})\gamma^{2}+a_{4}c_{2}\gamma^{3}+a_{3}\gamma\delta_{2}-(1-a_{4}c_{3})\gamma^{2}\delta_{2}\;,
	\end{align*}
	and subtracting \eqref{eq:q+1-2.1.1qopl} from this, we find that
	\begin{align}
		0&=(a_{1}+a_{4}c_{0})\left(\gamma-(a_{4}\nu)^{-1}\right)+(a_{2}+a_{4}c_{1})\left(\gamma^{2}-(a_{4}\nu)^{-2}\right)\nonumber\\&\qquad+a_{4}c_{2}\left(\gamma^{3}-(a_{4}\nu)^{-3}\right)-(1-a_{4}c_{3})\gamma\delta_{2}\left(\gamma-(a_{4}\nu)^{-1}\right)\nonumber\\
		&=\left(\gamma-(a_{4}\nu)^{-1}\right)\left[(a_{1}+a_{4}c_{0})+(a_{2}+a_{4}c_{1})\left(\gamma+(a_{4}\nu)^{-1}\right)\right.\nonumber\\&\qquad\qquad\qquad\qquad\left.+a_{4}c_{2}\left(\gamma^{2}+(a_{4}\nu)^{-2}\gamma+(a_{4}\nu)^{-2}\right)-(1-a_{4}c_{3})\gamma\delta_{2}\right]\;.\label{eq:q+1-2.1.1qoplbis}
	\end{align}
	The first factor in \eqref{eq:q+1-2.1.1qoplbis} cannot be zero as $\gamma\notin\F_{q}$. Hence, the second factor in \eqref{eq:q+1-2.1.1qoplbis} must be zero. However, as $\{1,\gamma,\gamma^{2},\gamma\delta_{2}\}$ is independent over $\F_{q}$, we have that $a_{4}c_{3}=1$, a contradiction since we assumed $\nu\neq0$. So, if $\nu_{1}=\nu=\frac{1-a_{4}c_{3}}{a_{3}a_{4}}$ then \eqref{eq:q+1standaardA211bis} has no solutions.
	\par We conclude that in Case B.1.1, we have $q-2$, $q-1$ or $q$ solutions of Equation \eqref{eq:q+1standaardA211bis} and at most 1 solution of Equation \eqref{eq:q+1standaardB211bis}. Now, recall that \eqref{eq:q+1standaardB211bis} has no solutions if and only if $a_{3}=0$ and $a_{0}\neq0$, or if $a_{4}=0$, but in both cases \eqref{eq:q+1-2.1.1ontaard} is a linear equation, and so there are at least $q-1$ solutions of \eqref{eq:q+1standaardA211bis}. So, in total there are at least $q-1$ and at most $q+1$ points in $(\Pi\setminus L)\cap\Omega_{2}$ in Case B.1.1.
	\par \textit{Case B.1.2: $\gamma^{2}\delta_{1}\in\left\langle 1,\gamma,\gamma^{2},\gamma\delta_{2}\right\rangle_q$ and $\dim\left\langle 1,\gamma,\gamma^{2},\gamma\delta_{2},\gamma^{2}\delta_{2}\right\rangle_q=5$.} The arguments in this case are similar to the arguments in Case B.1.1. We find that Equation \eqref{eq:q+1standaardA} has $q-2$, $q-1$ or $q$ solutions and Equation \eqref{eq:q+1standaardB} has at most 1 solution. However it cannot happen that simultaneously Equation \eqref{eq:q+1standaardA} has $q-2$ solutions and Equation \eqref{eq:q+1standaardB} has no solutions. So, in total there are at least $q-1$ and at most $q+1$ points in $(\Pi\setminus L)\cap\Omega_{2}$ in this case. Details can be found in Appendix \ifthenelse{\equal{\versie}{arxiv}}{\ref{ap:th4.3}, see page \pageref{apA:B1.2}}{A in the arXiv version of this paper}.
	\par \textit{Case B.2: $\gamma\delta_{2}\in\left\langle 1,\gamma,\gamma^{2},\gamma\delta_{1}\right\rangle_q$ and $\gamma\delta_{1}\notin\left\langle 1,\gamma,\gamma^{2},\gamma\delta_{2}\right\rangle_q$.} It follows from the assumption that $\gamma\delta_{2}\in\left\langle 1,\gamma,\gamma^{2}\right\rangle_q$. So, we can find $d_{0},d_{1},d_{2}\in\F_{q}$ such that
	\begin{align}
		\gamma\delta_{2}=d_{0}+d_{1}\gamma+d_{2}\gamma^{2}\label{gammadelta2B2}\;,
	\end{align}
	and we know that $d_{2}\neq0$ since $\dim\left\langle 1,\gamma,\gamma\delta_{1},\gamma\delta_{2}\right\rangle_q=4$. An argument, very similar to the one in Case B.1 shows that it is not possible that both $\gamma^{2}\delta_{1}$ and $\gamma^{2}\delta_{2}$ are contained in $\left\langle 1,\gamma,\gamma^{2},\gamma\delta_{1}\right\rangle_q$. We can now distinguish between two cases.
	\par \textit{Case B.2.1: $\dim\left\langle 1,\gamma,\gamma^{2},\gamma\delta_{1},\gamma^{2}\delta_{1}\right\rangle_q=5$.} The arguments in this case are similar to the arguments in Case B.1.1. We find that in total there are $q-1$ or $q$ points in $(\Pi\setminus L)\cap\Omega_{2}$ in this case. Details can be found in Appendix \ifthenelse{\equal{\versie}{arxiv}}{\ref{ap:th4.3}, see page \pageref{apA:B2.1}}{A in the arXiv version of this paper}.
	\par \textit{Case B.2.2: $\dim\left\langle 1,\gamma,\gamma^{2},\gamma\delta_{1},\gamma^2\delta_1\right\rangle_q\neq 5$.} The arguments in this case are similar to the arguments in Case B.1.1. Equation \eqref{eq:q+1standaardA} has $q-1$ solutions and \eqref{eq:q+1standaardB} has no solutions, or vice versa. We find that in total there are $q-1$ points in $(\Pi\setminus L)\cap\Omega_{2}$ in this case. Details can be found in Appendix \ifthenelse{\equal{\versie}{arxiv}}{\ref{ap:th4.3}, see page \pageref{apA:B2.2}}{A in the arXiv version of this paper}.
	\paragraph*{Conclusion:} We find in each of the cases that we described above that $(\Pi\setminus L)\cap\Omega_{2}$ contains at least $q-1$ and at most $q+1$ points. If three points of $(\Pi\setminus L)\cap\Omega_{2}$ would be collinear, then $\Pi\cap\Omega_{2}$ contains a $(q+1)$-secant different from $L$ by Theorem \ref{intersectionwithline}, and this contradicts Lemma \ref{geentweekortesecanten}. Hence, $|\Pi\cap\Omega_{2}|\in\{2q,2q+1,2q+2\}$, $q+1$ points of $\Pi\cap\Omega_{2}$ are collinear, and the remaining points of $\Pi\cap\Omega_{2}$ form an arc.
\end{proof}

\subsection{When there are no \texorpdfstring{$(q+1)$}{(q+1)}-secants to \texorpdfstring{$\Omega_2$}{Omega2}}

We need the following lemma.

\begin{lemma}\label{cubic}
	Let $f_{1},f_{2},g_{1},g_{2}$ be nonzero homogeneous polynomials in $\F_{q}[x,y,z]$ such that $\deg(f_{i})=1$ and $\deg(g_{i})=2$, $i=1,2$ where $g_i$ is possibly reducible. Let $L_i$ be the line of $\PG(2,q)$ defined by $f_i(x,y,z)=0$ and let $C_i$ be the conic defined by $g_i(x,y,z)=0$. Assume that $f_{1}g_{2}-f_{2}g_{1}$ does not vanish on every point and let $C$ be the cubic curve defined by $f_{1}g_{2}-f_{2}g_{1}=0$ in $\PG(2,q)$. Then either the cubic $C$ contains $N$ $\F_q$-rational points, where:
	\begin{enumerate}
		\item $N\in [q-2\sqrt{q}+1,q+2\sqrt{q}+1]$, or
		\item $N\in \{2q,2q+1,2q+2\}$, or
		\item $N \in \{3q,3q+1\}$,
	\end{enumerate}
	or else $C$ splits into three conjugate lines over a cubic extension, the lines $L_1$ and $L_2$ are distinct and both $C_{1}$ and $C_{2}$ contain the point $L_1\cap L_2$. 
\end{lemma}
\begin{proof}
	 We assume that $f_{1}g_{2}-f_{2}g_{1}$ does not vanish on each point, so it is not possible that each of the $q^2+q+1$ points of $\PG(2,q)$ satisfies $f_{1}g_{2}-f_{2}g_{1}=0$. In particular $f_{1}g_{2}-f_{2}g_{1}$ is not identically zero. In what follows, we use some facts about general cubic curves that can be found e.g.~in \cite[Chapter 11]{hirschfeld}. If a cubic plane curve is non-singular, then it is an elliptic curve and the Hasse bound implies that $|N-(q+1)|\leq 2\sqrt{q}$ (see \cite{hasse}). If a cubic curve is singular but irreducible, there is exactly one singular point (which is either a node, a cusp or an isolated double point). The number of points on $C$ is $q$ if there is a node, $q+1$ if there is a cusp and $q+2$ if there is an isolated double point. If a cubic curve is reducible, it either splits into three lines (possibly over an extension field) or in one line and an irreducible conic. In the latter case, the number of points on the curve is either $2q,2q+1$ or $2q+2$ depending on whether the line is secant, tangent or external to the conic.
	\par Now assume a cubic splits into three lines, say $\ell_1,\ell_2,\ell_3$. If all three are defined over $\F_q$, then $C$ contains $q+1$, $2q+1$, $3q$ or $3q+1$ points. If exactly one line, say $\ell_1$, is defined over $\F_q$ and the other two are conjugate lines over a quadratic extension of $\F_q$, then $C$ contains either $q+1$ or $q+2$ points depending on whether or not $\ell_2\cap \ell_3$ is on $\ell_1$.
	\par So now assume that the cubic curve $C:f_{1}g_{2}-f_{2}g_{1}=0$ splits in three lines, none of which are defined over $\F_q$. Then the curve splits into three conjugate lines over a cubic extension of $\F_q$, so it defines three lines $\ell_1,\ell_2,\ell_3$ in $\PG(2,q^3)$. Without loss of generality, we have $\ell_2=\ell_1^q,\ell_3=\ell_1^{q^2}$.  Note that in this case necessarily $L_1\neq L_2$. We denote the intersection point of $L_1$ and $L_2$ by $R$.
	\par It is clear that $C$ always contains the point $R$. If $\ell_1,\ell_2,\ell_3$ would be non-concurrent, then $C$ would not contain any points of $\PG(2,q)$. So we may now assume that the lines $\ell_1,\ell_1^q,\ell_1^{q^2}$ are concurrent in $R$ which implies that $C$ has only the point $R$ in $\PG(2,q)$.
	\par We will now show that this implies that $L_1\cap C_1=\{R\}$ and $L_2\cap C_2=\{R\}$ which then finishes the proof. Note that a point of $\PG(2,q)$ that lies on both $L_1$ and $C_1$ also lies on $C$. Similarly, any intersection point of $L_2$ and $C_2$ is contained in $C$.
		\par Since $C$ splits in three conjugate lines, we see that $L_1$ does not lie on $C_1$. Since $C_1$ is quadratic, either the intersection points of $L_1$ and $C_1$ are points of $\PG(2,q)$ or of a quadratic extension $\PG(2,q^2)$. 
	Since $C=\ell_1\cup\ell_1^q\cup \ell_1^{q^2}$ does not have points of $\PG(2,q^2)\setminus \PG(2,q)$. We find that $L_1$ meets $C_1$ in $2$ (possibly coinciding) points of $\PG(2,q)$. But since $C$ has only the point $R$ in $\PG(2,q)$, we find that $L_1\cap C_1=\{R\}$. 
	The same reasoning for $L_2$ and $C_2$ shows that $L_2\cap C_2=\{R\}$.
\end{proof}

\begin{theorem}\label{arcs}
	If $\Pi$ is a plane disjoint from $\Sigma$ that does not contain a $(q+1)$-secant or a $(q^{2}+q+1)$-secant to $\Omega_{2}$, then the points of $\Pi\cap \Omega_2$ form an $s$-arc with $s\in[q-2\sqrt{q}+1,q+2\sqrt{q}+1]$ or $s\in\{0,2q,2q+1,3q,3q+1,q^{2}+1\}$.
\end{theorem}
\begin{proof}
	We assume that $P_{0}$ is a point of $\Pi\cap\Omega_{2}$. Then, there exists a $\gamma_{0}\in\F_{q^{5}}\setminus\F_{q}$ and points $Q_{1},Q_{2}\in\Sigma$ such that $P_{0}=Q_{1}+\gamma_{0}Q_{2}$. 
	Let $\pi_{1}$ and $\pi_{2}$ be two planes of $\Sigma$ not through $\left\langle Q_{1},Q_{2}\right\rangle$, meeting this line in $Q_{1}$ and $Q_{2}$, respectively, and such that $\pi_1$ and $\pi_2$ do not have a line in common. Thus the intersection of the planes $\pi_{1}$ and $\pi_{2}$ is a point $Q_{3}$. Let $Q_{4}$ be a point of $\pi_{1}\setminus\left\langle Q_{1},Q_{3}\right\rangle$ and let $Q_{5}$ be a point of $\pi_{2}\setminus\left\langle Q_{2},Q_{3}\right\rangle$. Without loss of generality we can choose a basis for the underlying vector space such that $Q_{1}=\langle(1,0,0,0,0)\rangle$, $Q_{2}=\langle(0,1,0,0,0)\rangle$, $Q_{3}=\langle(0,0,1,0,0)\rangle$, $Q_{4}=\langle(0,0,0,1,0)\rangle$ and $Q_{5}=\langle(0,0,0,0,1)\rangle$.
	\par The planes $\Pi_{1}$ and $\Pi_{2}$ that are the extensions of $\pi_{1}$ and $\pi_{2}$ in $\PG(4,q^{5})$, respectively, cannot meet $\Pi$ in a line since otherwise the planes $\pi_{1}$ and $\pi_{2}$ would give rise to a $(q^{2}+q+1)$-secant to $\Omega_2$ in $\Pi$. So, the planes $\Pi_{1}$ and $\Pi_{2}$ each meet $\Pi$ in a point, and there are $\gamma_{1},\gamma'_{1},\gamma_{2},\gamma'_{2}\in\F_{q^{5}}$ such that $P_{1}=\gamma_{1}Q_{1}+\gamma'_{1}Q_{3}+Q_{4}$ and $P_{2}=\gamma_{2}Q_{2}+\gamma'_{2}Q_{3}+Q_{5}$ are points of $\Pi$. Note that $\gamma_{1}$ and $\gamma'_{1}$ cannot be both elements of $\F_{q}$ since $\Pi$ is disjoint to $\Sigma$. Similarly, $\gamma_{2}$ and $\gamma'_{2}$ cannot both be elements of $\F_{q}$ since $\Pi$ is disjoint to $\Sigma$.
	\par We now look for points in $\Pi\cap\Omega_{2}$. It is clear that any point $P$ of $\Omega_{2}$ can be written as $\langle(\mu_{1},\mu_{2},\mu_{3},\mu_{4},\mu_{5})+\varphi(\nu_{1},\nu_{2},\nu_{3},\nu_{4},\nu_{5})\rangle$ for some $\mu_{i},\nu_{i}\in\F_{q}$ and $\varphi\in\F_{q^{5}}\setminus\F_{q}$. Clearly, each point $P$ can in many ways be written as such a sum. However, it is easy to see that for each $P$ in $\Pi\cap\Omega_{2}$ different from $P_{0}$ there is a unique presentation in exactly one of the following forms:
	\begin{enumerate}[(a)]
		\item there are unique $\mu_{i},\nu_{i}\in \F_{q}$, $i=1,2,3$, such that $P=\langle(\mu_{1},\mu_{2},\mu_{3},0,1)+\varphi(\nu_{1},\nu_{2},\nu_{3},1,0)\rangle$,
		\item there are unique $\mu_{1},\mu_{2},\nu_{1},\nu_{2}\in \F_{q}$ and $\mu_{4}\in \F^{*}_{q}$ such that $P=\langle(\mu_{1},\mu_{2},0,\mu_{4},1)+\varphi(\nu_{1},\nu_{2},1,0,0)\rangle$,
		\item there are unique $\mu_{1},\mu_{2},\nu_{1},\nu_{2}\in \F_{q}$ such that $P=\langle(\mu_{1},\mu_{2},1,0,0)+\varphi(\nu_{1},\nu_{2},0,1,0)\rangle$,
		\item there are unique $\mu_{1},\mu_{2},\nu_{1},\nu_{2}\in \F_{q}$ such that $P=\langle(\mu_{1},\mu_{2},1,0,0)+\varphi(\nu_{1},\nu_{2},0,0,1)\rangle$.
	\end{enumerate}
	In this distinction of four cases we used again that $\Pi$ cannot contain a $(q^{2}+q+1)$-secant. As the point $P$ is contained in $\Pi$ there are $\alpha_{0},\alpha_{1},\alpha_{2}\in\F_{q^{5}}$ such that $P=\alpha_{0}P_{0}+\alpha_{1}P_{1}+\alpha_{2}P_{2}$. Comparing both expressions for $P$ we find the following system of equations (one equation for each coordinate):
	\begin{align}\label{eq:1alg}
		\begin{cases}
			\alpha_{0}+\alpha_{1}\gamma_{1}=\mu_{1}+\varphi\nu_{1}\\
			\alpha_{0}\gamma_{0}+\alpha_{2}\gamma_{2}=\mu_{2}+\varphi\nu_{2}\\
			\alpha_{1}\gamma'_{1}+\alpha_{2}\gamma'_{2}=\mu_{3}+\varphi\nu_{3}\\
			\alpha_{1}=\mu_{4}+\varphi\nu_{4}\\
			\alpha_{2}=\mu_{5}+\varphi\nu_{5}
		\end{cases}
		\Leftrightarrow\quad
		\begin{cases}
			\alpha_{0}=\mu_{1}+\varphi\nu_{1}-\left(\mu_{4}+\varphi\nu_{4}\right)\gamma_{1}\\
			\alpha_{1}=\mu_{4}+\varphi\nu_{4}\\
			\alpha_{2}=\mu_{5}+\varphi\nu_{5}\\
			\alpha_{0}\gamma_{0}+\alpha_{2}\gamma_{2}=\mu_{2}+\varphi\nu_{2}\\
			\alpha_{1}\gamma'_{1}+\alpha_{2}\gamma'_{2}=\mu_{3}+\varphi\nu_{3}
		\end{cases}.
	\end{align}
	Each solution in the $\alpha_{i}$'s, $\mu_{i}$'s, $\nu_{i}$'s and $\varphi$ with  $(\mu_{4},\nu_{4},\mu_{5},\nu_{5})=(0,1,1,0)$, or with $(\mu_{3},\nu_{3},\nu_{4},\mu_{5},\nu_{5})=(0,1,0,1,0)$ and $\mu_{4}\neq0$, or with $(\mu_{3},\nu_{3},\mu_{4},\nu_{4},\mu_{5},\nu_{5})=(1,0,0,1,0,0)$, or with $(\mu_{3},\nu_{3},\mu_{4},\nu_{4},\mu_{5},\nu_{5})=(1,0,0,0,0,1)$, of this system of equations corresponds to a unique point of $\Pi\cap\Omega_{2}$. The first three equations in \eqref{eq:1alg} describe $\alpha_{0},\alpha_{1},\alpha_{2}$ as functions of the other unknowns so can be disregarded from now on. So we consider the system of equations
	\begin{align}\label{eq:1algbis}
		&\begin{cases}
			\left(\mu_{1}+\varphi\nu_{1}-\left(\mu_{4}+\varphi\nu_{4}\right)\gamma_{1}\right)\gamma_{0}+\left(\mu_{5}+\varphi\nu_{5}\right)\gamma_{2}=\mu_{2}+\varphi\nu_{2}\\
			\left(\mu_{4}+\varphi\nu_{4}\right)\gamma'_{1}+\left(\mu_{5}+\varphi\nu_{5}\right)\gamma'_{2}=\mu_{3}+\varphi\nu_{3}
		\end{cases}\nonumber\\
		\Leftrightarrow\quad
		&\begin{cases}
			\varphi\left(\nu_{1}\gamma_{0}-\nu_{4}\gamma_{0}\gamma_{1}+\nu_{5}\gamma_{2}-\nu_{2}\right)=\mu_{2}-\mu_{1}\gamma_{0}+\mu_{4}\gamma_{0}\gamma_{1}-\mu_{5}\gamma_{2}\\
			\varphi\left(\nu_{4}\gamma'_{1}+\nu_{5}\gamma'_{2}-\nu_{3}\right)=\mu_{3}-\mu_{4}\gamma'_{1}-\mu_{5}\gamma'_{2}
		\end{cases}.
	\end{align}
	Given the $\mu_{i}$'s and $\nu_{i}$'s, this system of equations has $0$, $1$ or $q^{5}$ solutions for $\varphi$. However, if it has $q^{5}$ solutions for $\varphi$, then the line $\left\langle(\mu_{1},\mu_{2},\mu_{3},\mu_{4},\mu_{5}),(\nu_{1},\nu_{2},\nu_{3},\nu_{4},\nu_{5})\right\rangle$ is a line of $\Pi$ that meets $\Sigma$ in an $\F_{q}$-subline, contradicting the assumption that $\Pi$ and $\Sigma$ are disjoint. So, \eqref{eq:1algbis} has either 0 solutions or a unique solution in $\varphi$, and the latter occurs iff
	\begin{multline}\label{eq:1standaard}
		\left(\nu_{1}\gamma_{0}-\nu_{4}\delta+\nu_{5}\gamma_{2}-\nu_{2}\right)\left(\mu_{3}-\mu_{4}\gamma'_{1}-\mu_{5}\gamma'_{2}\right)\\=\left(\nu_{4}\gamma'_{1}+\nu_{5}\gamma'_{2}-\nu_{3}\right)\left(\mu_{2}-\mu_{1}\gamma_{0}+\mu_{4}\delta-\mu_{5}\gamma_{2}\right)\;,
	\end{multline}
	where we substituted $\gamma_{0}\gamma_{1}$ by $\delta$. We only need to find the solutions of \eqref{eq:1standaard} in the four cases for the $\mu_{i}$'s and $\nu_{i}$'s described above, so we will look at the four following equations:
	\begin{align}
		\delta\gamma'_{2}+\gamma'_{1}\gamma_{2}&=\left(\mu_{3}\nu_{2}-\mu_{2}\nu_{3}\right)+\left(\mu_{1}\nu_{3}-\mu_{3}\nu_{1}\right)\gamma_{0}+\mu_{3}\delta+\mu_{2}\gamma'_{1}+\nu_{3}\gamma_{2}-\nu_{2}\gamma'_{2}\nonumber\\&\qquad-\mu_{1}\gamma_{0}\gamma'_{1}+\nu_{1}\gamma_{0}\gamma'_{2}\;,\label{eq:1standaardA}\\
		-\gamma_{2}&=-\mu_{2}+\mu_{1}\gamma_{0}-\mu_{4}\delta-\nu_{2}\mu_{4}\gamma'_{1}-\nu_{2}\gamma'_{2}+\mu_{4}\nu_{1}\gamma_{0}\gamma'_{1}+\nu_{1}\gamma_{0}\gamma'_{2} \;,\label{eq:1standaardB}\\
		-\delta&=\nu_{2}-\nu_{1}\gamma_{0}+\mu_{2}\gamma'_{1}-\mu_{1}\gamma_{0}\gamma'_{1}\;,\label{eq:1standaardC}\\
		\gamma_{2}&=\nu_{2}-\nu_{1}\gamma_{0}+\mu_{2}\gamma'_{2}-\mu_{1}\gamma_{0}\gamma'_{2}\;.\label{eq:1standaardD}
	\end{align}
	We will distinguish between several cases, depending on the relation between $\gamma_{0}$, $\gamma_{2}$, $\gamma'_{1}$, $\gamma'_{2}$ and $\delta$, when discussing these equations. To do this we define $U_{1}=\left\langle1,\gamma_{0},\gamma'_{1},\gamma_{0}\gamma'_{1}\right\rangle_q$ and $U_{2}=\left\langle1,\gamma_{0},\gamma'_{2},\gamma_{0}\gamma'_{2}\right\rangle_q$.
	\paragraph*{Intermezzo:} 
	\begin{itemize}
		\item We first show that if $\dim U_{1}=2$ or $\dim U_{2}=2$, then there is a $(q^2+q+1)$-secant to $\Omega_2$ in $\Pi$. Suppose that $\dim U_1=2$. Since $\gamma_0\notin \F_q$, $U_1=\langle 1,\gamma_0\rangle$. This implies that $\gamma_1'=a+b\gamma_0$ for some $a,b\in \F_q$ and $\gamma_0\gamma'_1=a\gamma_0+b\gamma_0^2=c+d\gamma_0$ for some $c,d\in \F_q$. Since $\gamma_0\in \F_{q^5}\setminus \F_q$, the set $\{1,\gamma_0,\gamma_0^2\}$ is $\F_q$-independent. It follows that $b=0$ and hence, $\gamma'_1=a\in \F_q$. This implies that $P_1=\gamma_{1}Q_{1}+aQ_{3}+Q_{4}=\gamma_1Q_1+Q_3'$, where $Q_3'=aQ_3+Q_4\in\Sigma$. We see that both $Q_1Q_2$ and $Q_1Q_3'$ extend to a line containing a point of $\Pi$. Hence, the line $P_0P_1$ is a $(q^2+q+1)$-secant. A similar reasoning shows that if $\dim U_{2}=2$, the line $P_0P_2$ is a $(q^2+q+1)$-secant.
		\item Now, we show that if $\dim U_2=3$ and $\gamma_2\in U_2$, then $P_0P_2$ is a $(q+1)$-secant. 
		If $\dim U_2=3$ and $\gamma_2\in U_2$, then Equation \eqref{eq:1standaardD} has clearly $q$ solutions. Each of these solutions corresponds to a distinct point on the line $P_0P_2$ in $\Pi$, different from $P_{0}$. So, the line $P_0P_2$ is a $(q+1)$-secant. Analogously, if $\dim U_1=3$ and $\delta\in U_1$, the line $P_0P_1$ is a $(q+1)$-secant.
		\item Next, we show that there is a $(q^{2}+q+1)$-secant in $\Pi$ if 1, $\gamma'_{1}$ and $\gamma'_{2}$ are linearly dependent over $\F_{q}$. Assume that there are $a,b,c\in\F_{q}$, with $(a,b,c)\neq(0,0,0)$, such that $a+b\gamma'_{1}+c\gamma'_{2}=0$. Consider Equation \eqref{eq:1algbis}, and put $\mu_3=-\nu_{3}=a$, $\mu_4=-\nu_4=-b$ and $\mu_5=-\nu_5=-c$. We see that the second equation vanishes. We can pick $(\nu_{1},\nu_{2})=(\overline{\nu_1},\overline{\nu_2})$ such that $\overline{\nu_{1}}\gamma_{0}-b\delta+c\gamma_{2}-\overline{\nu_{2}}\neq 0$, hence such that the coefficient of $\varphi$ in the first equation is different from zero. It follows that the line through the points $(\mu_1,\mu_2,a,-b,-c)$ and $(\overline{\nu_1},\overline{\nu_2},-a,b,c)$ gives rise to a point of rank $2$ for any $\mu_{1},\mu_{2}\in\F_{q}$ with $(\mu_1,\mu_2)\neq (\overline{\nu_1},\overline{\nu_2})$. In particular, the point $Q_1=\langle(1,0,0,0,0)\rangle$ lies on a line of this form, different from $Q_1Q_2$, and hence, we find a $(q^2+q+1)$-secant through $P_0$.
		\item Finally, we show that if $U_{1}=U_{2}$ and $\dim U_{1}=3$, then there is a $(q^2+q+1)$-secant to $\Omega_2$ in $\Pi$. Suppose that $U_1=U_2$ and $\dim U_1=3$. First assume that there exist $e_{1},e_{2}\in\F_{q}$ such that $\gamma_1'=e_1+e_2\gamma_0$. We have that $e_2\neq 0$ since otherwise $\dim U_1=2$. We have that $\gamma_0\gamma_1'=e_1\gamma_0+e_2\gamma_0^2$ and $U_1=\langle 1,\gamma_0,\gamma_0^2\rangle$. Since $U_2=U_1=\langle 1,\gamma_0,\gamma_0^2\rangle$, we have that $\gamma_2'=f_1+f_2\gamma_0+f_3\gamma_0^2$ and $\gamma_0\gamma_2'=f_1\gamma_0+f_2\gamma_0^2+f_3\gamma_0^3\in U_2$. Since $1$, $\gamma_0$, $\gamma_0^2$ and $\gamma_0^3$ are $\F_q$-independent, we have that $f_3=0$. This implies that $\gamma_2'=f_1+f_2\gamma_0$ and hence that $f_2\gamma_1'-e_2\gamma_2'+e_2f_1-e_1f_2=0$. It follows that $\{1,\gamma_1',\gamma_2'\}$ is a linearly dependent set over $\F_q$. We have seen in the previous bullet point that this implies that there is a $(q^2+q+1)$-secant in $\Pi$.
		\par Now suppose that $1,\gamma_0,\gamma_1'$ are $\F_q$-independent. This implies that there are $a_{i},b_{i}c_{i}\in\F_{q}$, $i=1,2,3$, such that
		\begin{align*}
			\gamma_0\gamma_1'&=a_1+a_2\gamma_0+a_3\gamma_1'\;,\\
			\gamma_2'&=b_1+b_2\gamma_0+b_3\gamma_1'\quad\text{and}\\
			\gamma_0\gamma_2'&=c_1+c_2\gamma_0+c_3\gamma_1'\;.
		\end{align*}
		First note that $1$, $\gamma_0$, $\gamma_0^2$ and $\gamma_1'$ are $\F_{q}$-independent. Indeed, if $\gamma_1'=d_0+d_1\gamma_0+d_2\gamma_0^2$, then $d_0\gamma_0+d_1\gamma_0^2+d_2\gamma_0^3=\gamma_0\gamma_1'=a_1+a_2\gamma_{0}+a_3(d_0+d_1\gamma_0+d_2\gamma_0^2)$. Since $1,\gamma_0,\gamma_0^2,\gamma_0^3$ are linearly independent, $d_2=0$ and hence $\gamma_1'=d_0+d_1\gamma_0$, a contradiction since $1,\gamma_0,\gamma_1'$ are independent. We now have that
		\[
		c_1+c_2\gamma_0+c_3\gamma_1'=\gamma_0\gamma'_2=\gamma_0(b_1+b_2\gamma_0+b_3\gamma_1')=\gamma_0b_1+\gamma_0^2b_2+b_3(a_1+a_2\gamma_0+a_3\gamma_1')\;.
		\]
		This implies that $c_1=a_1b_3$, $c_2=b_1+a_2b_3$, $c_3=a_3b_3$ and $b_2=0$ since $1$, $\gamma_0$, $\gamma_0^2$ and $\gamma_1'$ are $\F_{q}$-independent. We see that $b_3\gamma_1'-\gamma_2'+b_1=0$. It follows that $\{1,\gamma_1',\gamma_2'\}$ is a linearly dependent set over $\F_q$. We have seen in the previous bullet point that this implies that there is a $(q^2+q+1)$-secant in $\Pi$.
	\end{itemize}
	\par Keeping the results of this intermezzo in mind, we have that either (A) at least one of the subspaces $U_{1}$ and $U_{2}$ has dimension 4, or else (B) $\dim U_{1}=\dim U_{2}=3$ and $\dim\left\langle U_{1},U_{2}\right\rangle=4$ since $\Pi$ does not contain a $(q^2+q+1)$-secant by the assumption of the theorem. So, we will distinguish between these two cases. 
	\paragraph*{Case A:} We assume that $\dim U_{1}=4$ or $\dim U_{2}=4$. Now, note that Equation \eqref{eq:1standaard} is invariant when interchanging $(\gamma'_{1},\delta,\mu_{4},\nu_{4})$ and $(\gamma'_{2},-\gamma_{2},\mu_{5},\nu_{5})$. It follows that the set of Equations \eqref{eq:1standaardA}--\eqref{eq:1standaardD} when interchanging $(\gamma'_{1},\delta)$ and $(\gamma'_{2},-\gamma_{2})$ yields the same system of equations after renaming some of the variables $\mu_{i}$ and $\nu_{i}$. Recall that $\mu_{4}$ in Equation \eqref{eq:1standaardB} cannot be zero. For this reason $U_{1}$ and  $U_{2}$ are interchangeable, so we may assume without loss of generality that $\dim U_{1}=4$. We now distinguish between several subcases. For cases A.1 and A.3 we present the details. The arguments in the other subcases are similar and can be found in Appendix \ifthenelse{\equal{\versie}{arxiv}}{\ref{ap:th4.5}}{B in the arXiv version of this paper}.
	\par \textit{Case A.1: $\gamma'_{2}\notin U_{1}$}. Hence, we assume that $\dim\left\langle 1,\gamma_{0},\gamma'_{1},\gamma_{0}\gamma'_{1},\gamma'_{2}\right\rangle_q=5$, in other words $\left\{1,\gamma_{0},\gamma'_{1},\gamma_{0}\gamma'_{1},\gamma'_{2}\right\}$ is an $\F_{q}$-basis for $\F_{q^{5}}$. Then, there are $a_{i},b_{i},c_{i},d_{i}\in\F_{q}$, $i=1,\dots,5$, such that
	\begin{align*}
		\gamma_{0}\gamma'_{2}&=a_{1}+a_{2}\gamma_{0}+a_{3}\gamma'_{1}+a_{4}\gamma_{0}\gamma'_{1}+a_{5}\gamma'_{2}\;,\\
		\delta&=b_{1}+b_{2}\gamma_{0}+b_{3}\gamma'_{1}+b_{4}\gamma_{0}\gamma'_{1}+b_{5}\gamma'_{2}\;,\\
		\gamma_{2}&=c_{1}+c_{2}\gamma_{0}+c_{3}\gamma'_{1}+c_{4}\gamma_{0}\gamma'_{1}+c_{5}\gamma'_{2}\quad\text{and}\\
		\delta\gamma'_{2}+\gamma'_{1}\gamma_{2}&=d_{1}+d_{2}\gamma_{0}+d_{3}\gamma'_{1}+d_{4}\gamma_{0}\gamma'_{1}+d_{5}\gamma'_{2}\;.
	\end{align*}
	As seen in the intermezzo, $\left\langle P_{0},P_{2}\right\rangle$ is a $(q+1)$-secant if $\gamma_{2}\in U_{2}$ and $\dim U_{2}=3$. In this case, these conditions are fulfilled if and only if $a_{3}=a_{4}=0$ and $c_{3}=c_{4}=0$. So, we may assume that $(a_{3},a_{4},c_{3},c_{4})\neq(0,0,0,0)$. We also show that we cannot have simultaneously $a_{1}=-a_{2}a_{5}$ and $a_{3}=-a_{4}a_{5}$, since otherwise it follows that
	\begin{align*}
		\gamma_{0}\gamma'_{2}&=-a_{2}a_{5}+a_{2}\gamma_{0}-a_{4}a_{5}\gamma'_{1}+a_{4}\gamma_{0}\gamma'_{1}+a_{5}\gamma'_{2}\quad\Leftrightarrow\quad 0=(\gamma_{0}-a_{5})(\gamma'_{2}-a_{4}\gamma'_{1}-a_{2})\;,
	\end{align*}
	which is not possible since $\gamma_{0}\notin\F_{q}$ and $\{1,\gamma'_{1},\gamma'_{2}\}$ is independent over $\F_{q}$.
	\par Considering $\F_{q^{5}}$ as a vector space over $\F_{q}$, Equation \eqref{eq:1standaardA} is equivalent to the following system of equations:
	\begin{align}\label{eq:1standaardA11}
		\begin{cases}
			d_{1}=\mu_{3}\nu_{2}-\mu_{2}\nu_{3}+\nu_{1}a_{1}+\mu_{3}b_{1}+\nu_{3}c_{1}\\
			d_{2}=\mu_{1}\nu_{3}-\mu_{3}\nu_{1}+\nu_{1}a_{2}+\mu_{3}b_{2}+\nu_{3}c_{2}\\
			d_{3}=\mu_{2}+\nu_{1}a_{3}+\mu_{3}b_{3}+\nu_{3}c_{3}\\
			d_{4}=-\mu_{1}+\nu_{1}a_{4}+\mu_{3}b_{4}+\nu_{3}c_{4}\\
			d_{5}=-\nu_{2}+\nu_{1}a_{5}+\mu_{3}b_{5}+\nu_{3}c_{5}
		\end{cases}
		\!\Leftrightarrow\ 
		\begin{cases}
			\mu_{1}=\nu_{1}a_{4}+\mu_{3}b_{4}+\nu_{3}c_{4}-d_{4}\\
			\mu_{2}=-\nu_{1}a_{3}-\mu_{3}b_{3}-\nu_{3}c_{3}+d_{3}\\
			\nu_{2}=\nu_{1}a_{5}+\mu_{3}b_{5}+\nu_{3}c_{5}-d_{5}\\
			d_{1}=\mu_{3}\nu_{2}-\mu_{2}\nu_{3}+\nu_{1}a_{1}+\mu_{3}b_{1}+\nu_{3}c_{1}\\
			d_{2}=\mu_{1}\nu_{3}-\mu_{3}\nu_{1}+\nu_{1}a_{2}+\mu_{3}b_{2}+\nu_{3}c_{2}
		\end{cases}\!\!\!.
	\end{align}
	It is straightforward to see that there is a one-to-one correspondence between the solutions in $(\mu_{1},\mu_{2},\mu_{3},\nu_{1},\nu_{2},\nu_{3})$ of Equation \eqref{eq:1standaardA11} and the solutions in $(\mu_{3},\nu_{1},\nu_{3})$ of
	\begin{align}\label{eq:1standaardA11bis}
		&\begin{cases}
			d_{1}=\mu_{3}(\nu_{1}a_{5}+\mu_{3}b_{5}+\nu_{3}c_{5}-d_{5})+(\nu_{1}a_{3}+\mu_{3}b_{3}+\nu_{3}c_{3}-d_{3})\nu_{3}+\nu_{1}a_{1}+\mu_{3}b_{1}+\nu_{3}c_{1}\\
			d_{2}=(\nu_{1}a_{4}+\mu_{3}b_{4}+\nu_{3}c_{4}-d_{4})\nu_{3}-\mu_{3}\nu_{1}+\nu_{1}a_{2}+\mu_{3}b_{2}+\nu_{3}c_{2}
		\end{cases}\nonumber\\
		\Leftrightarrow\ 
		&\begin{cases}
			-L_{1}(\mu_{3},\nu_{3})\nu_{1}=C_{1}(\mu_{3},\nu_{3})\\
			-L_{2}(\mu_{3},\nu_{3})\nu_{1}=C_{2}(\mu_{3},\nu_{3})
		\end{cases}
	\end{align}
	with
	\begin{align*}
		L_{1}(\mu_{3},\nu_{3})&=a_{5}\mu_{3}+a_{3}\nu_{3}+a_{1}\;,\\
		L_{2}(\mu_{3},\nu_{3})&=-\mu_{3}+a_{4}\nu_{3}+a_{2}\;,\\
		C_{1}(\mu_{3},\nu_{3})&=b_{5}\mu^{2}_{3}+(c_{5}+b_{3})\mu_{3}\nu_{3}+c_{3}\nu^{2}_{3}+(b_{1}-d_{5})\mu_{3}+(c_{1}-d_{3})\nu_{3}-d_{1}\;\text{ and}\\
		C_{2}(\mu_{3},\nu_{3})&=b_{4}\mu_{3}\nu_{3}+c_{4}\nu^{2}_{3}+b_{2}\mu_{3}+(c_{2}-d_{4})\nu_{3}-d_{2}\;.
	\end{align*}
	Given $\mu_{3}$ and $\nu_{3}$, the system of equations in \eqref{eq:1standaardA11bis} has $0$, $1$ or $q$ solutions for $\nu_{1}$. Assume that for $(\mu_{3},\nu_{3})=(\overline{\mu_{3}},\overline{\nu_{3}})$ the system of equations in \eqref{eq:1standaardA11bis} would have $q$ solutions. Then, looking at \eqref{eq:1algbis} with $(\mu_4,\mu_5,\nu_4,\nu_5)=(0,1,1,0)$, we see that for the $q$ corresponding points, we have $\varphi=\frac{\overline{\mu_{3}}-\gamma'_{2}}{\gamma'_{1}-\overline{\nu_{3}}}$. Hence, any two of these $q$ points determine a $(q+1)$-secant by Theorem \ref{qplusonesecant}, contradicting the assumption on $\Pi$. So, \eqref{eq:1standaardA11bis} has either 0 solutions or a unique solution in $\nu_{1}$, and the latter occurs iff
	\begin{multline}\label{eq:1cubic11}
		F(\mu_{3},\nu_{3})=L_{1}(\mu_{3},\nu_{3})C_{2}(\mu_{3},\nu_{3})-L_{2}(\mu_{3},\nu_{3})C_{1}(\mu_{3},\nu_{3})=0\\ \wedge\quad\left(L_{1}(\mu_{3},\nu_{3}),L_{2}(\mu_{3},\nu_{3})\right)\neq(0,0)\;.
	\end{multline}
	The equation $F(\mu_{3},\nu_{3})=0$ determines a cubic curve $C$ in the $(\mu_{3},\nu_{3})$-plane $\pi\cong\AG(2,q)$. We embed this affine plane in the projective plane $\PG(2,q)$ by adding the line at infinity $\ell_{\infty}$ and extend $C$ to the cubic curve $\overline{C}$ by going to a homogeneous equation $\overline{F}(\mu_{3},\nu_{3},\rho)=0$. Note that the equation of $\ell_\infty$ is $\rho=0$. Analogously, we define the homogeneous functions $\overline{L_{1}}(\mu_{3},\nu_{3},\rho)$, $\overline{L_{2}}(\mu_{3},\nu_{3},\rho)$, $\overline{C_{1}}(\mu_{3},\nu_{3},\rho)$, and $\overline{C_{2}}(\mu_{3},\nu_{3},\rho)$. Note that neither $\overline{L_{1}}$ nor $\overline{L_{2}}$ can be identically zero; the latter is obvious, and in case $L_{1}\equiv0$ we would have that $\gamma_{0}\gamma'_{2}=a_{2}\gamma_{0}+a_{4}\gamma_{0}\gamma'_{1}$, hence that $\{1,\gamma'_{1},\gamma'_{2}\}$ is not a linearly independent set over $\F_{q}$, contradicting the assumption of Case A.1. The lines defined by $\overline{L_{1}}=0$ and $\overline{L_{2}}=0$ in the projective plane $\PG(2,q)$, do not coincide since we cannot simultaneously have $a_{1}=-a_{2}a_{5}$ and $a_{3}=-a_{4}a_{5}$ (as we noted before). So, these lines have precisely one intersection point $R$, which is on $\ell_{\infty}$ if and only if $a_{3}+a_{4}a_{5}=0$. It is clear that $R$ is on the cubic curve $\overline{F}=0$.
	\par We denote the number of points on $\overline{C}$ by $N$ and the number of points on $\overline{C}\cap\ell_{\infty}$ by $N_{\infty}$. Furthermore, we set $\varepsilon=1$ if $R$ is an affine point, and $\varepsilon=0$ if $R\in\ell_{\infty}$. We find that Equation \eqref{eq:1cubic11}, and hence also Equation \eqref{eq:1standaardA11}, has $N-N_{\infty}-\varepsilon$ solutions.
	\par Now, we look at Equation \eqref{eq:1standaardB}; it is equivalent to the following system of equations:
	\begin{align}\label{eq:1standaardB11}
		\begin{cases}
			-c_{1}=-\mu_{2}+a_{1}\nu_{1}-b_{1}\mu_{4}\\
			-c_{2}=\mu_{1}+a_{2}\nu_{1}-b_{2}\mu_{4}\\
			-c_{3}=-\nu_{2}\mu_{4}+a_{3}\nu_{1}-b_{3}\mu_{4}\\
			-c_{4}=\mu_{4}\nu_{1}+a_{4}\nu_{1}-b_{4}\mu_{4}\\
			-c_{5}=-\nu_{2}+a_{5}\nu_{1}-b_{5}\mu_{4}
		\end{cases}
		\Leftrightarrow\quad
		\begin{cases}
			\mu_{1}=-a_{2}\nu_{1}+b_{2}\mu_{4}-c_{2}\\
			\mu_{2}=a_{1}\nu_{1}-b_{1}\mu_{4}+c_{1}\\
			\nu_{2}=a_{5}\nu_{1}-b_{5}\mu_{4}+c_{5}\\
			c_{3}=\nu_{2}\mu_{4}-a_{3}\nu_{1}+b_{3}\mu_{4}\\
			-c_{4}=\mu_{4}\nu_{1}+a_{4}\nu_{1}-b_{4}\mu_{4}
		\end{cases}\;.
	\end{align}
	Recall that $\mu_{4}\in\F^{*}_{q}$. It is straightforward that there is a one-to-one correspondence between the solutions in $(\mu_{1},\mu_{2},\mu_{4},\nu_{1},\nu_{2})$ of Equation \eqref{eq:1standaardB11} and the solutions in $(\mu_{4},\nu_{1})$ of
	\begin{align}\label{eq:1standaardB11bis}
		&\begin{cases}
			c_{3}=(a_{5}\nu_{1}-b_{5}\mu_{4}+c_{5})\mu_{4}-a_{3}\nu_{1}+b_{3}\mu_{4}\\
			-c_{4}=\mu_{4}\nu_{1}+a_{4}\nu_{1}-b_{4}\mu_{4}
		\end{cases}\nonumber\\
		\Leftrightarrow\quad
		&\begin{cases}
			\left(a_{5}\mu_{4}-a_{3}\right)\nu_{1}=b_{5}\mu^{2}_{4}-(c_{5}+b_{3})\mu_{4}+c_{3}\\
			(\mu_{4}+a_{4})\nu_{1}=b_{4}\mu_{4}-c_{4}
		\end{cases}\nonumber\\
		\Leftrightarrow\quad
		&\begin{cases}
			-\overline{L_{1}}\left(-\mu_{4},1,0\right)\nu_{1}=\overline{C_{1}}\left(-\mu_{4},1,0\right)\\
			-\overline{L_{2}}\left(-\mu_{4},1,0\right)\nu_{1}=\overline{C_{2}}\left(-\mu_{4},1,0\right)
		\end{cases}.
	\end{align}
	Given $\mu_{4}$, the system of equations in \eqref{eq:1standaardB11bis} has $0$, $1$ or $q$ solutions for $\nu_{1}$. Assume that for $\mu_{4}=\overline{\mu_{4}}$ the system of equations in \eqref{eq:1standaardB11bis} has $q$ solutions. Then, looking at \eqref{eq:1algbis} with $(\mu_{3},\nu_{3},\nu_{4},\mu_{5},\nu_{5})=(0,1,0,1,0)$, we see that for the $q$ corresponding points, we have $\varphi=\overline{\mu_{4}}\gamma'_{1}+\gamma'_{2}$. Hence, any two of these $q$ points determine a $(q+1)$-secant by Theorem \ref{qplusonesecant} contradicting the assumption on $\Pi$. So, \eqref{eq:1standaardB11bis} has either 0 solutions or a unique solution in $\nu_{1}$, and the latter occurs if and only if $\mu_{4}\in\F^{*}_{q}$ fulfils
	\begin{align}\label{eq:1oneindig11}
		0&=\overline{L_{1}}\left(-\mu_{4},1,0\right)\overline{C_{2}}\left(-\mu_{4},1,0\right)-\overline{L_{2}}\left(-\mu_{4},1,0\right)\overline{C_{1}}\left(-\mu_{4},1,0\right)\nonumber\\
		&=\overline{F}(-\mu_{4},1,0)\;,
	\end{align}
	and simultaneously $(\overline{L_{1}}\left(-\mu_{4},1,0\right),\overline{L_{2}}\left(-\mu_{4},1,0\right))\neq(0,0)$. However, if 
	$$(\overline{L_{1}}\left(-\mu_{4},1,0\right),\overline{L_{2}}\left(-\mu_{4},1,0\right))=(0,0),$$ then $R=\langle(-\mu_{4},1,0)\rangle\in\ell_{\infty}$. So, the solutions of \eqref{eq:1oneindig11} correspond to the points of $\overline{C}\cap(\ell_{\infty}\setminus\{R,\langle(1,0,0)\rangle,\langle(0,1,0)\rangle\})$.
	\par Now we look at Equations \eqref{eq:1standaardC} and \eqref{eq:1standaardD}. It is immediately clear that Equation \eqref{eq:1standaardC} has 1 solution if $b_{5}=0$ and no solutions otherwise, so it has a solution if and only if $\langle(1,0,0)\rangle\in\overline{C}$. Equation \eqref{eq:1standaardD} is equivalent to
	\begin{multline*}
		c_{1}+c_{2}\gamma_{0}+c_{3}\gamma'_{1}+c_{4}\gamma_{0}\gamma'_{1}+c_{5}\gamma'_{2}\\=(\nu_{2}-a_{1}\mu_{1})-(\nu_{1}+a_{2}\mu_{1})\gamma_{0}-a_{3}\mu_{1}\gamma'_{1}-a_{4}\mu_{1}\gamma_{0}\gamma'_{1}+(\mu_{2}-a_{5}\mu_{1})\gamma'_{2}\;.
	\end{multline*}
	This equation has one solution if $a_{3}c_{4}=a_{4}c_{3}$ and $(a_{3},a_{4})\neq(0,0)$ and no solutions otherwise; recall (from the beginning of this case) that it is not possible that $a_{3}=a_{4}=c_{3}=c_{4}=0$. So, it has a solution if and only if $\langle(0,1,0)\rangle\in\overline{C}$, but $R\neq\langle(0,1,0)\rangle$.
	\par We note that $R$ cannot be the point $\langle(1,0,0)\rangle$, and we conclude that regardless of the behaviour of $b_{5}$ and $a_{3}c_{4}-a_{4}c_{3}$ and the position of $R$, the total number of solutions of the Equations \eqref{eq:1standaardB}, \eqref{eq:1standaardC} and \eqref{eq:1standaardD} together equals $N_{\infty}-(1-\varepsilon)$. Including the solutions from Equation \eqref{eq:1standaardA} and the point $P_{0}$, we find that $\Pi\cap\Omega_{2}$ contains $N$ points.
	\par We have seen during the analysis of Equations \eqref{eq:1standaardA11bis} and \eqref{eq:1standaardB11bis} that we could only find at most one solution for $\nu_1$ for any choice of the parameters $(\mu_{3},\nu_{3})$ or $\mu_{4}$. From this it follows that $R$ cannot be on both conics $\overline{C_{1}}=0$ and $\overline{C_{2}}=0$ if it is an affine point or a point on $\ell_{\infty}\setminus\{\langle(1,0,0)\rangle,\langle(0,1,0)\rangle\}$. If $R=\langle(0,1,0)\rangle$ this point cannot be on both conics since it is not possible that $a_{3}=a_{4}=c_{3}=c_{4}=0$. Recall that $R\neq\langle(1,0,0)\rangle$ since $\langle(1,0,0)\rangle\notin \overline{L_2}$. So, $R$ cannot be on both conics $\overline{C_{1}}=0$ and $\overline{C_{2}}=0$, regardless of its coordinates.
	\par Note that if $\Pi\cap\Omega_{2}$ contains $N=q^2+q+1$ points, then there are at least $2$ points with the same type (recall that there are $q^{2}+1$ $G$-orbits), and hence, there is a $(q+1)$-secant by Theorem \ref{qplusonesecant}. So, $\overline{F}$ does not vanish. This implies that the theorem follows from applying Lemma \ref{cubic} to the cubic $\overline{C}$.
	\par \textit{Case A.2: $\gamma'_{2}\in U_{1}$, but $\gamma_{0}\gamma'_{2}\notin U_{1}$}.  The arguments in this case are similar to the arguments in Case A.1. We find that the points on $\Pi\cap\Omega_{2}$ correspond to the points on a cubic and we apply Lemma \ref{cubic}. Details can be found in Appendix \ifthenelse{\equal{\versie}{arxiv}}{\ref{ap:th4.5}, see page \pageref{apB:A2}}{B in the arXiv version of this paper}.
	\par \textit{Case A.3: $\gamma'_{2},\gamma_{0}\gamma'_{2}\in U_{1}$, but $\delta\notin U_{1}$}. Hence, we assume that $\dim\left\langle 1,\gamma_{0},\gamma'_{1},\gamma_{0}\gamma'_{1},\delta\right\rangle_q=5$, in other words $\left\{1,\gamma_{0},\gamma'_{1},\gamma_{0}\gamma'_{1},\delta\right\}$ is an $\F_{q}$-basis for $\F_{q^{5}}$. Note that in this case $U_{2}\leq U_{1}$. Then, there are $a_{i},b_{i},c_{i},d_{i}\in\F_{q}$, $i=1,\dots,5$, such that
	\begin{align*}
		\gamma'_{2}&=a_{1}+a_{2}\gamma_{0}+a_{3}\gamma'_{1}+a_{4}\gamma_{0}\gamma'_{1}\;,\\
		\gamma_{0}\gamma'_{2}&=b_{1}+b_{2}\gamma_{0}+b_{3}\gamma'_{1}+b_{4}\gamma_{0}\gamma'_{1}\;,\\
		\gamma_{2}&=c_{1}+c_{2}\gamma_{0}+c_{3}\gamma'_{1}+c_{4}\gamma_{0}\gamma'_{1}+c_{5}\delta\quad\text{and}\\
		\delta\gamma'_{2}+\gamma'_{1}\gamma_{2}&=d_{1}+d_{2}\gamma_{0}+d_{3}\gamma'_{1}+d_{4}\gamma_{0}\gamma'_{1}+d_{5}\delta\;.
	\end{align*}
	As seen in the intermezzo, $\left\langle P_{0},P_{2}\right\rangle$ is a $(q+1)$-secant if $\gamma_{2}\in U_{2}$ and $\dim U_{2}=3$. In this case, these conditions are fulfilled if and only if $a_{3}b_{4}=a_{4}b_{3}$ and $c_{5}=0$. So, we may assume that $(a_{3}b_{4}-a_{4}b_{3},c_{5})\neq(0,0)$. Note also that we cannot have $a_{2}=a_{4}=0$ or $b_{1}=b_{3}=0$: in both cases we would have that $\{1,\gamma'_{1},\gamma'_{2}\}$ is not a linearly independent set over $\F_{q}$. However, then the intermezzo shows that there is a $(q^{2}+q+1)$-secant, a contradiction.
	\par Considering $\F_{q^{5}}$ as a vector space over $\F_{q}$, Equation \eqref{eq:1standaardA} is equivalent to the following system of equations:
	\begin{align}\label{eq:1standaardA13}
		\begin{cases}
			d_{1}=\mu_{3}\nu_{2}-\mu_{2}\nu_{3}-\nu_{2}a_{1}+\nu_{1}b_{1}+\nu_{3}c_{1}\\
			d_{2}=\mu_{1}\nu_{3}-\mu_{3}\nu_{1}-\nu_{2}a_{2}+\nu_{1}b_{2}+\nu_{3}c_{2}\\
			d_{3}=\mu_{2}-\nu_{2}a_{3}+\nu_{1}b_{3}+\nu_{3}c_{3}\\
			d_{4}=-\mu_{1}-\nu_{2}a_{4}+\nu_{1}b_{4}+\nu_{3}c_{4}\\
			d_{5}=\mu_{3}+\nu_{3}c_{5}
		\end{cases}
		\Leftrightarrow\quad
		\begin{cases}
			\mu_{1}=-\nu_{2}a_{4}+\nu_{1}b_{4}+\nu_{3}c_{4}-d_{4}\\
			\mu_{2}=\nu_{2}a_{3}-\nu_{1}b_{3}-\nu_{3}c_{3}+d_{3}\\
			\mu_{3}=-\nu_{3}c_{5}+d_{5}\\
			d_{1}=\mu_{3}\nu_{2}-\mu_{2}\nu_{3}-\nu_{2}a_{1}+\nu_{1}b_{1}+\nu_{3}c_{1}\\
			d_{2}=\mu_{1}\nu_{3}-\mu_{3}\nu_{1}-\nu_{2}a_{2}+\nu_{1}b_{2}+\nu_{3}c_{2}
		\end{cases}\;.
	\end{align}
	It is straightforward to see that there is a one-to-one correspondence between the solutions in $(\mu_{1},\mu_{2},\mu_{3},\nu_{1},\nu_{2},\nu_{3})$ of Equation \eqref{eq:1standaardA13} and the solutions in $(\nu_{1},\nu_{2},\nu_{3})$ of
	\begin{align}\label{eq:1standaardA13bis}
		&\begin{cases}
			d_{1}=(d_{5}-\nu_{3}c_{5})\nu_{2}+(\nu_{1}b_{3}-\nu_{2}a_{3}+\nu_{3}c_{3}-d_{3})\nu_{3}-\nu_{2}a_{1}+\nu_{1}b_{1}+\nu_{3}c_{1}\\
			d_{2}=(\nu_{1}b_{4}-\nu_{2}a_{4}+\nu_{3}c_{4}-d_{4})\nu_{3}+(\nu_{3}c_{5}-d_{5})\nu_{1}-\nu_{2}a_{2}+\nu_{1}b_{2}+\nu_{3}c_{2}
		\end{cases}\nonumber\\
		\Leftrightarrow\quad
		&\begin{cases}
			L_{1}(\nu_{2},\nu_{3})\nu_{1}=C_{1}(\nu_{2},\nu_{3})\\
			L_{2}(\nu_{2},\nu_{3})\nu_{1}=C_{2}(\nu_{2},\nu_{3})
		\end{cases}
	\end{align}
	with
	\begin{align*}
		L_{1}(\nu_{2},\nu_{3})&=b_{3}\nu_{3}+b_{1}\;,\\
		L_{2}(\nu_{2},\nu_{3})&=(b_{4}+c_{5})\nu_{3}+b_{2}-d_{5}\;,\\
		C_{1}(\nu_{2},\nu_{3})&=(a_{3}+c_{5})\nu_{2}\nu_{3}-c_{3}\nu^{2}_{3}+(a_{1}-d_{5})\nu_{2}+(d_{3}-c_{1})\nu_{3}+d_{1}\;\text{ and}\\
		C_{2}(\nu_{2},\nu_{3})&=a_{4}\nu_{2}\nu_{3}-c_{4}\nu^{2}_{3}+a_{2}\nu_{2}+(d_{4}-c_{2})\nu_{3}+d_{2}\;.
	\end{align*}
	Given $\nu_{2}$ and $\nu_{3}$, the system of equations in \eqref{eq:1standaardA13bis} has $0$, $1$ or $q$ solutions for $\nu_{1}$. Assume that for $(\nu_{2},\nu_{3})=(\overline{\nu_{2}},\overline{\nu_{3}})$ the system of equations in \eqref{eq:1standaardA13bis} would have $q$ solutions. Then, looking at \eqref{eq:1algbis} with $(\mu_4,\mu_5,\nu_4,\nu_5)=(0,1,1,0)$, we see that, for the $q$ corresponding points, we have $\varphi=\frac{d_{5}-c_{5}\overline{\nu_{3}}-\gamma'_{2}}{\gamma'_{1}-\overline{\nu_{3}}}$. Hence any two of these $q$ points determine a $(q+1)$-secant by Theorem \ref{qplusonesecant}, contradicting the assumption on $\Pi$. So, \eqref{eq:1standaardA13bis} has either 0 solutions or a unique solution in $\nu_{1}$, and the latter occurs iff 
	\begin{multline}\label{eq:1cubic13}
		F(\nu_{2},\nu_{3})=L_{1}(\nu_{2},\nu_{3})C_{2}(\nu_{2},\nu_{3})-L_{2}(\nu_{2},\nu_{3})C_{1}(\nu_{2},\nu_{3})=0\\ \wedge\quad\left(L_{1}(\nu_{2},\nu_{3}),L_{2}(\nu_{2},\nu_{3})\right)\neq(0,0)\;.
	\end{multline}
	The equation $F(\nu_{2},\nu_{3})=0$ determines a cubic curve $C$ in the $(\nu_{2},\nu_{3})$-plane $\pi\cong\AG(2,q)$. We embed this affine plane in the projective plane $\overline{\pi}\cong\PG(2,q)$ by adding the line at infinity $\ell_{\infty}$ and extend $C$ to the cubic curve $\overline{C}$ by going to a homogeneous equation $\overline{F}(\nu_{2},\nu_{3},\rho)=0$. Analogously, we define the homogeneous functions $\overline{L_{1}}(\nu_{2},\nu_{3},\rho)$, $\overline{L_{2}}(\nu_{2},\nu_{3},\rho)$, $\overline{C_{1}}(\nu_{2},\nu_{3},\rho)$, and $\overline{C_{2}}(\nu_{2},\nu_{3},\rho)$.
	\par Note that $\overline{L_{1}}$ cannot be identically zero: in case $L_{1}\equiv0$ we would have that $b_{1}=b_{3}=0$, which is not possible as we have seen above. Furthermore, if $(\overline{\nu_{2}},\overline{\nu_{3}})$ and $(\overline{\nu_{2}}',\overline{\nu_{3}})$ are two points on $C$, then the corresponding rank 2 points of $\Pi$ both have $\varphi=\frac{d_{5}-c_{5}\overline{\nu_{3}}-\gamma'_{2}}{\gamma'_{1}-\overline{\nu_{3}}}$, so these two points determine a $(q+1)$-secant by Theorem \ref{qplusonesecant}, contradicting the assumption on $\Pi$. So, no two affine points of $C$ can be on the same line through $R=\langle(1,0,0)\rangle$ unless one of them is contained in both $L_{1}=0$ and $L_{2}=0$. Note that $R \in \overline{C}\cap \ell_\infty$.
	\par Now, we look at Equation \eqref{eq:1standaardB}; it is equivalent to the following system of equations:
	\begin{align}\label{eq:1standaardB13}
		\begin{cases}
			-c_{1}=-\mu_{2}-a_{1}\nu_{2}+b_{1}\nu_{1}\\
			-c_{2}=\mu_{1}-a_{2}\nu_{2}+b_{2}\nu_{1}\\
			-c_{3}=-\nu_{2}\mu_{4}-a_{3}\nu_{2}+b_{3}\nu_{1}\\
			-c_{4}=\mu_{4}\nu_{1}-a_{4}\nu_{2}+b_{4}\nu_{1}\\
			-c_{5}=-\mu_{4}
		\end{cases}
		\Leftrightarrow\quad
		\begin{cases}
			\mu_{1}=a_{2}\nu_{2}-b_{2}\nu_{1}-c_{2}\\
			\mu_{2}=-a_{1}\nu_{2}+b_{1}\nu_{1}+c_{1}\\
			\mu_{4}=c_{5}\\
			-c_{3}=-\nu_{2}\mu_{4}-a_{3}\nu_{2}+b_{3}\nu_{1}\\
			-c_{4}=\mu_{4}\nu_{1}-a_{4}\nu_{2}+b_{4}\nu_{1}
		\end{cases}.
	\end{align}
	Recall that $\mu_{4}\in\F^{*}_{q}$. So, the system of equations in \eqref{eq:1standaardB13} has no solutions if $c_{5}=0$. Hence, we assume in the discussion of this system of equations that $c_{5}\neq0$. It is straightforward that there is a one-to-one correspondence between the solutions in $(\mu_{1},\mu_{2},\mu_{4},\nu_{1},\nu_{2})$ of Equation \eqref{eq:1standaardB13} and the solutions in $(\nu_{1},\nu_{2})$ of
	\begin{align}\label{eq:1standaardB13bis}
		&\begin{cases}
			-c_{3}=-c_{5}\nu_{2}-a_{3}\nu_{2}+b_{3}\nu_{1}\\
			-c_{4}=c_{5}\nu_{1}-a_{4}\nu_{2}+b_{4}\nu_{1}
		\end{cases}
		&\Leftrightarrow\quad
		&\begin{cases}
			-c_{3}=b_{3}\nu_{1}-(a_{3}+c_{5})\nu_{2}\\
			-c_{4}=(b_{4}+c_{5})\nu_{1}-a_{4}\nu_{2}
		\end{cases}\nonumber\\
		&&\Leftrightarrow\quad
		&\begin{cases}
			\overline{L_{1}}\left(\nu_{2},1,0\right)\nu_{1}=\overline{C_{1}}\left(\nu_{2},1,0\right)\\
			\overline{L_{2}}\left(\nu_{2},1,0\right)\nu_{1}=\overline{C_{2}}\left(\nu_{2},1,0\right)
		\end{cases}.
	\end{align}
	The system of equations in \eqref{eq:1standaardB13bis} has $0$, $1$ or $q$ solutions for $\nu_{1}$. Assume that it would have $q$ solutions. Then, looking at \eqref{eq:1algbis} with $(\mu_{3},\nu_{3},\nu_{4},\mu_{5},\nu_{5})=(0,1,0,1,0)$, we see that for the $q$ corresponding points, we have $\varphi=c_{5}\gamma'_{1}+\gamma'_{2}$. Hence any two of these $q$ points determine a $(q+1)$-secant by Theorem \ref{qplusonesecant}, contradicting the assumption on $\Pi$. So, \eqref{eq:1standaardB13bis} has either 0 solutions or a unique solution in $\nu_{2}$, and the latter occurs if and only if $\nu_{2}$ fulfils
	\begin{align}\label{eq:1oneindig13}
		0&=\overline{L_{1}}\left(\nu_{2},1,0\right)\overline{C_{2}}\left(\nu_{2},1,0\right)-\overline{L_{2}}\left(\nu_{2},1,0\right)\overline{C_{1}}\left(\nu_{2},1,0\right)\nonumber\\
		&=\overline{F}(\nu_{2},1,0)\nonumber\\
		&=\left(a_{4}b_{3}-a_{3}b_{4}-c_{5}\left(a_{3}+b_{4}\right)-c^{2}_{5}\right)\nu_{2}+\left(b_{4}c_{3}-b_{3}c_{4}+c_{3}c_{5}\right)\;,
	\end{align}
	and simultaneously $(\overline{L_{1}}\left(\nu_{2},1,0\right),\overline{L_{2}}\left(\nu_{2},1,0\right))\neq(0,0)$. Furthermore, if $\langle(\overline{\nu_{2}},1,0)\rangle$ and $\langle(\overline{\nu_{2}}',1,0)\rangle$ are two points on $\overline{C}$, then the corresponding rank 2 points of $\Pi$ both have $\varphi=c_{5}\gamma'_{1}+\gamma'_{2}$, so these two points determine a $(q+1)$-secant by Theorem \ref{qplusonesecant}, contradicting the assumption on $\Pi$. So, Equation \eqref{eq:1standaardA13bis} has no solutions or one solution; the former occurs if $a_{4}b_{3}-a_{3}b_{4}-c_{5}\left(a_{3}+b_{4}\right)-c^{2}_{5}=0$ and the latter otherwise.
	\par Now we look at Equations \eqref{eq:1standaardC} and \eqref{eq:1standaardD}. It is immediately clear that Equation \eqref{eq:1standaardC} has no solutions by the assumption of Case A.3. Equation \eqref{eq:1standaardD} is equivalent to
	\begin{multline}\label{eq:1standaardD13}
		c_{1}+c_{2}\gamma_{0}+c_{3}\gamma'_{1}+c_{4}\gamma_{0}\gamma'_{1}+c_{5}\delta\\=(\nu_{2}+a_{1}\mu_{2}-b_{1}\mu_{1})-(\nu_{1}-a_{2}\mu_{2}+b_{2}\mu_{1})\gamma_{0}+(a_{3}\mu_{2}-b_{3}\mu_{1})\gamma'_{1}+(a_{4}\mu_{2}-b_{4}\mu_{1})\gamma_{0}\gamma'_{1}\;.
	\end{multline}
	This equation has no solutions if $c_{5}\neq0$ and one solution otherwise; recall (from the beginning of this case) that it is not possible that $(a_{3}b_{4}-a_{4}b_{3},c_{5})=(0,0)$.
	{\par \textit{Interlude:} Before we continue, we will show that it is impossible that the cubic $\overline{C}$ decomposes in three lines through $R$ (possibly defined over an extension field). We will assume it does decompose and derive a contradiction. Then, the equation $F=0$ enjoys zero coefficients for $\nu_{2}\nu^{2}_{3}$, $\nu_{2}\nu_{3}$ and $\nu_{2}$. Hence, we have
	\begin{align}
		0&=a_{4}b_{3}-\overline{a}\overline{b}\;,\label{eq:cubicontaard1}\\
		0&=a_{2}b_{3}+a_{4}b_{1}-\overline{a}\hat{b}-\hat{a}\overline{b}\;\text{ and}\label{eq:cubicontaard2}\\
		0&=a_{2}b_{1}-\hat{a}\hat{b}\label{eq:cubicontaard3}\;,
	\end{align}
	with $\overline{a}=a_{3}+c_{5}$, $\overline{b}=b_{4}+c_{5}$, $\hat{a}=a_{1}-d_{5}$ and $\hat{b}=b_{2}-d_{5}$. It follows that
	\begin{align}
		\left(a_{4}\hat{a}-a_{2}\overline{a}\right)\left(a_{4}\hat{b}-a_{2}\overline{b}\right)&=a^{2}_{4}\hat{a}\hat{b}+a^{2}_{2}\overline{a}\overline{b}-a_{2}a_{4}\left(\overline{a}\hat{b}+\hat{a}\overline{b}\right)\nonumber\\
		&=a^{2}_{4}a_{2}b_{1}+a^{2}_{2}a_{4}b_{3}-a_{2}a_{4}\left(a_{2}b_{3}+a_{4}b_{1}\right)=0\;\text{ and}\label{eq:cubicontaardbis1}\\
		\left(b_{3}\hat{a}-b_{1}\overline{a}\right)\left(b_{3}\hat{b}-b_{1}\overline{b}\right)&=b^{2}_{3}\hat{a}\hat{b}+b^{2}_{1}\overline{a}\overline{b}-b_{1}b_{3}\left(\overline{a}\hat{b}+\hat{a}\overline{b}\right)\nonumber\\
		&=b^{2}_{3}a_{2}b_{1}+b^{2}_{1}a_{4}b_{3}-b_{1}b_{3}\left(a_{2}b_{3}+a_{4}b_{1}\right)=0\;.\label{eq:cubicontaardbis2}
	\end{align}
	Recall that it not possible that either $a_{2}=a_{4}=0$ or $b_{1}=b_{3}=0$. Then it follows from Equations \eqref{eq:cubicontaard1}, \eqref{eq:cubicontaard2} and \eqref{eq:cubicontaard3} that it is impossible that $\overline{a}=\hat{a}=0$ or $\overline{b}=\hat{b}=0$. Furthermore, if $\overline{a}=\overline{b}=0$, then it follows from Equations \eqref{eq:cubicontaardbis1} and \eqref{eq:cubicontaardbis2} that $a_{4}=b_{3}=0$. Then, we also find that
	\begin{align*}
		\gamma'_{2}&=a_{1}+a_{2}\gamma_{0}-c_{5}\gamma'_{1} &\Leftrightarrow&& \gamma'_{2}+c_{5}\gamma'_{1}-d_{5}&=\hat{a}+a_{2}\gamma_{0}\;\text{ and}\\
		\gamma_{0}\gamma'_{2}&=b_{1}+b_{2}\gamma_{0}-c_{5}\gamma_{0}\gamma'_{1} &\Leftrightarrow&& \gamma_{0}\left(\gamma'_{2}+c_{5}\gamma'_{1}-d_{5}\right)&=b_{1}+\hat{b}\gamma_{0}\;.
	\end{align*}
	So, $a_{2}\gamma^{2}_{0}+(\hat{a}-\hat{b})\gamma_{0}-b_{1}=0$, and since $\{1,\gamma_{0},\gamma^{2}_{0}\}$ is an $\F_{q}$-independent set, we must have $a_{2}=0$, a contradiction as also $a_{4}=0$. So, we cannot have $\overline{a}=\overline{b}=0$. Similarly we can prove that we also cannot have $\hat{a}=\hat{b}=0$.
	\par Considering these remarks, we see that there are four possibilities given Equations \eqref{eq:cubicontaardbis1} and \eqref{eq:cubicontaardbis2}. We discuss them one by one. Recall that we showed in the intermezzo that $\{1,\gamma'_{1},\gamma'_{2}\}$ is an $\F_{q}$-independent set.
	\begin{itemize}
		\item There are $k,k'\in\F^{*}_{q}$ such that $(a_{2},a_{4})=k(\hat{a},\overline{a})$ and $(b_{1},b_{3})=k'(\hat{a},\overline{a})$. From Equations \eqref{eq:cubicontaard1}, \eqref{eq:cubicontaard2} and \eqref{eq:cubicontaard3} it follows that $\overline{a}(kk'\overline{a}-\overline{b})=0=\hat{a}(kk'\hat{a}-\hat{b})$ and $2kk'\overline{a}\hat{a}-\overline{a}\hat{b}-\hat{a}\overline{b}=0$. As $(\overline{a},\hat{a})\neq(0,0)$ and $(\overline{a},\overline{b})\neq(0,0)\neq(\hat{a},\hat{b})$, we must have $\overline{b}=kk'\overline{a}$ and $\hat{b}=kk'\hat{a}$. So, in this case we find that
		\begin{align*}
			&&\gamma'_{2}&=(\hat{a}+d_{5})+k\hat{a}\gamma_{0}+(\overline{a}-c_{5})\gamma'_{1}+k\overline{a}\gamma_{0}\gamma'_{1}\\
			\Leftrightarrow&& \gamma'_{2}+c_{5}\gamma'_{1}-d_{5}&=(1+k\gamma_{0})(\hat{a}+\overline{a}\gamma'_{1})\;\text{ and}\\
			&&\gamma_{0}\gamma'_{2}&=k'\hat{a}+(\hat{b}+d_{5})\gamma_{0}+k'\overline{a}\gamma'_{1}+(\overline{b}-c_{5})\gamma_{0}\gamma'_{1}\\
			\Leftrightarrow&& \gamma_{0}\left(\gamma'_{2}+c_{5}\gamma'_{1}-d_{5}\right)&=k'(1+k\gamma_{0})(\hat{a}+\overline{a}\gamma'_{1})\;.
		\end{align*}
		It follows that $(\gamma_{0}-k')(\gamma'_{2}+c_{5}\gamma'_{1}-d_{5})=0$,
		which is a contradiction since $\gamma_{0}\notin\F_{q}$ and $\{1,\gamma'_{1},\gamma'_{2}\}$ is a linearly independent set over $\F_{q}$.
		\item There are $k,k'\in\F^{*}_{q}$ such that $(a_{2},a_{4})=k(\hat{a},\overline{a})$ and $(b_{1},b_{3})=k'(\hat{b},\overline{b})$. From Equations \eqref{eq:cubicontaard1}, \eqref{eq:cubicontaard2} and \eqref{eq:cubicontaard3} it follows that $\overline{a}\overline{b}(kk'-1)=0=\hat{a}\hat{b}(kk'-1)$ and $(kk'-1)\overline{a}\hat{b}+(kk'-1)\hat{a}\overline{b}=0$. As $(\overline{a},\hat{a})\neq(0,0)$ and $(\overline{a},\overline{b})\neq(0,0)\neq(\hat{a},\hat{b})$, we must have $kk'=1$. So, in this case we find that
		\begin{align*}
			&&\gamma'_{2}&=(\hat{a}+d_{5})+k\hat{a}\gamma_{0}+(\overline{a}-c_{5})\gamma'_{1}+k\overline{a}\gamma_{0}\gamma'_{1}\\ &\Leftrightarrow& \gamma'_{2}+c_{5}\gamma'_{1}-d_{5}&=(1+k\gamma_{0})(\hat{a}+\overline{a}\gamma'_{1})\qquad\text{ and}\\
			&&\gamma_{0}\gamma'_{2}&=k'\hat{b}+(\hat{b}+d_{5})\gamma_{0}+k'\overline{b}\gamma'_{1}+(\overline{b}-c_{5})\gamma_{0}\gamma'_{1}\\ &\Leftrightarrow& \gamma_{0}\left(\gamma'_{2}+c_{5}\gamma'_{1}-d_{5}\right)&=(k'+\gamma_{0})(\hat{b}+\overline{b}\gamma'_{1})\;.
		\end{align*}
		Hence, we have
		\begin{align*}
			\gamma_{0}(1+k\gamma_{0})(\hat{a}+\overline{a}\gamma'_{1})&=(k'+\gamma_{0})(\hat{b}+\overline{b}\gamma'_{1})\\
			\Leftrightarrow\quad 0&=(k'+\gamma_{0})\left(\hat{b}-k\hat{a}\gamma_{0}+\overline{b}\gamma'_{1}-k\overline{a}\gamma_{0}\gamma'_{1}\right)\;,
		\end{align*}
		which is a contradiction since $\gamma_{0}\notin\F_{q}$ and since $\{1,\gamma_{0},\gamma'_{1},\gamma_{0}\gamma'_{1}\}$ is a linearly independent set over $\F_{q}$ and $(\overline{b},\hat{b})\neq(0,0)$.
		\item There are $k,k'\in\F^{*}_{q}$ such that $(a_{2},a_{4})=k(\hat{b},\overline{b})$ and $(b_{1},b_{3})=k'(\hat{a},\overline{a})$. From Equations \eqref{eq:cubicontaard1}, \eqref{eq:cubicontaard2} and \eqref{eq:cubicontaard3} it follows that $\overline{a}\overline{b}(kk'-1)=0=\hat{a}\hat{b}(kk'-1)$ and $(kk'-1)\overline{a}\hat{b}+(kk'-1)\hat{a}\overline{b}=0$. As $(\overline{a},\hat{a})\neq(0,0)$ and $(\overline{a},\overline{b})\neq(0,0)\neq(\hat{a},\hat{b})$, we must have $kk'=1$. So, in this case we find that 
		\begin{align*}
			(\gamma_{0}-k')(\gamma'_{2}+c_{5}\gamma'_{1}-d_{5})&=k'\hat{a}+(\hat{b}+d_{5})\gamma_{0}+k'\overline{a}\gamma'_{1}+(\overline{b}-c_{5})\gamma_{0}\gamma'_{1}+c_{5}\gamma_{0}\gamma'_{1}-d_{5}\gamma_{0}\\
			&\qquad-k'(\hat{a}+d_{5})-k'k\hat{b}\gamma_{0}-k'(\overline{a}-c_{5})\gamma'_{1}-k'k\overline{b}\gamma_{0}\gamma'_{1}\\
			&\qquad-k'c_{5}\gamma'_{1}+k'd_{5}\\
			&=0\;,
		\end{align*}
		which is a contradiction since $\gamma_{0}\notin\F_{q}$ and $\{1,\gamma'_{1},\gamma'_{2}\}$ is a linearly independent set over $\F_{q}$.
		\item There are $k,k'\in\F^{*}_{q}$ such that $(a_{2},a_{4})=k(\hat{b},\overline{b})$ and $(b_{1},b_{3})=k'(\hat{b},\overline{b})$. From Equations \eqref{eq:cubicontaard1}, \eqref{eq:cubicontaard2} and \eqref{eq:cubicontaard3} it follows that $\overline{b}(kk'\overline{b}-\overline{a})=0=\hat{b}(kk'\hat{b}-\hat{a})$ and $2kk'\overline{b}\hat{b}-\overline{a}\hat{b}-\hat{a}\overline{b}=0$. As $(\overline{b},\hat{b})\neq(0,0)$ and $(\overline{a},\overline{b})\neq(0,0)\neq(\hat{a},\hat{b})$, we must have $\overline{a}=kk'\overline{b}$ and $\hat{a}=kk'\hat{b}$. So, in this case we find that
		\begin{align*}
			(k\gamma_{0}-1)(\gamma'_{2}+c_{5}\gamma'_{1}-d_{5})&=kk'\hat{b}+k(\hat{b}+d_{5})\gamma_{0}+kk'\overline{b}\gamma'_{1}+k(\overline{b}-c_{5})\gamma_{0}\gamma'_{1}+kc_{5}\gamma_{0}\gamma'_{1}\\
			&\qquad-kd_{5}\gamma_{0}-(\hat{a}+d_{5})-k\hat{b}\gamma_{0}-(\overline{a}-c_{5})\gamma'_{1}-k\overline{b}\gamma_{0}\gamma'_{1}\\&\qquad-c_{5}\gamma'_{1}+d_{5}\\
			&=0\;,
		\end{align*}
		which is a contradiction since $\gamma_{0}\notin\F_{q}$ and $\{1,\gamma'_{1},\gamma'_{2}\}$ is a linearly independent set over $\F_{q}$.
	\end{itemize}
	So, we find a contradiction in each case, and we conclude this interlude: the cubic $\overline{C}$ cannot decompose in three lines through $R=\langle(1,0,0)\rangle$.}
	\par Now, we denote the number of points on $\overline{C}$ by $N$ and the number of points on $\overline{C}\cap\ell_{\infty}$ by $N_{\infty}$. If $c_{5}=0$ (and consequently $a_{3}b_{4}-a_{4}b_{3}\neq0$), Equations \eqref{eq:1standaardB13} and \eqref{eq:1standaardC} have no solutions, and \eqref{eq:1standaardD13} has one solution.
	Since \eqref{eq:1standaardB13bis} has one solution and, as seen before, $R\in\overline{C}\cap\ell_{\infty}$, it follows that $N_{\infty}=2$ in this case. If $c_{5}\neq0$, Equations \eqref{eq:1standaardC} and \eqref{eq:1standaardD13} have no solutions, and \eqref{eq:1standaardB13bis} has $N_{\infty}-1$ solutions. So, regardless of the behaviour of $c_{5}$, we find that the total number of solutions of Equations \eqref{eq:1standaardB}, \eqref{eq:1standaardC} and \eqref{eq:1standaardD} equals $N_{\infty}-1$ in this case.
	\par  For the discussion of the number of solutions of the Equations \eqref{eq:1standaardA}, \eqref{eq:1standaardB}, \eqref{eq:1standaardC} and \eqref{eq:1standaardD} together, we distinguish between the following cases. Denote the line $\overline{L_{1}}=0$ by $\ell_{1}$ and the line $\overline{L_{2}}=0$ by $\ell_{2}$, in case $\overline{L_{2}}$ does not vanish. 	\begin{itemize}
		\item \textit{The function $\overline{L_{2}}$ does not vanish, and the lines $\ell_{1}$ and $\ell_{2}$ in $\overline{\pi}$ do not coincide.} Then the lines $\ell_{1}$ and $\ell_{2}$ meet in the point $R\in\ell_{\infty}$. In this case Equation \eqref{eq:1cubic13}, and hence also Equation \eqref{eq:1standaardA}, has $N-N_{\infty}$ solutions. We find that the total number of solutions of the four equations equals $N-1$. Including the point $P_{0}$, we find that $\Pi\cap\Omega_{2}$ contains $N$ points.
		\par On the one hand, we showed before that in Case A.3 the affine cubic $C$ contains at most $q$ points. Since the cubic $\overline{C}$ contains $R$ but cannot decompose in three lines through $R$, Lemma \ref{cubic} shows that $N=|\Pi\cap \Omega_2|\in [q-2\sqrt{q}+1,q]$.
		\item \textit{The function $\overline{L_{2}}$ does not vanish, but the lines $\ell_{1}$ and $\ell_{2}:\overline{L_{2}}=0$ in $\overline{\pi}$ coincide} or \textit{the function $\overline{L_{2}}$ vanishes.} Then there is a $k\in\F_{q}$ be such that $\overline{L_{2}}=k\overline{L_{1}}$. Then $\overline{F}=\overline{L_{1}}\left(\overline{C_{2}}-k\overline{C_{1}}\right)$, and hence the cubic $C$ decomposes in the line $\ell_{1}$ and a conic $C':\overline{C_{2}}-k\overline{C_{1}}=0$. Note that $R\in C'$. The solutions of \eqref{eq:1standaardA13bis} correspond to the affine points on $C'\setminus\ell_{1}$. Now,
		\begin{itemize}
			\item if $C'$ is non-degenerate, then the number of affine points on $C'\setminus\ell_{1}$ equals $q-2$ or $q-1$; the latter occurs if $\ell_{1}$ or $\ell_{\infty}$ is a tangent to $C'$, and the former otherwise. If $\ell_{\infty}$ is a tangent then $N_{\infty}=1$, and else $N_{\infty}=2$. We find that the total number of solutions of the four equations equals $q-1$ or $q$. Including the point $P_{0}$, we find that $\Pi\cap\Omega_{2}$ contains $q$ or $q+1$ points.
			\item If $C'$ decomposes in two lines $m$ and $m'$ in $\overline{\pi}$, then at least one of them, say $m$ must contain $R$. However, we know that $\overline{C}$ cannot contain two affine points on the same line through $R$ if this line is different from $\ell_{1}$ (see discussion after Equation \eqref{eq:1cubic13}). Also, we know that $\ell_{\infty}$ contains at most one point of $\overline{C}$ next to $R$ (see the discussion after Equation \eqref{eq:1oneindig13}). So, $m=\ell_{1}$. Now, note that $R\notin m'$ since $\overline{C}$ cannot decompose in three lines through $R$. So, the solutions of Equation \eqref{eq:1cubic13} correspond to the $q-1$ affine points of $m'$ not on $\ell_{1}$. Moreover, as $R\notin m'$, we have $N_{\infty}=2$, so the four equations have $q$ solutions together. Including the point $P_{0}$, we find that $\Pi\cap\Omega_{2}$ contains $q+1$ points.
			\item If $C'$ decomposes in two lines $m$ and $m'$ not in $\overline{\pi}$, but defined over a quadratic extension of $\F_{q}$, then $m$ and $m'$ both contain $R$ -- recall that $C'$ contains $R$. However, then $\overline{C}$ decomposes in three lines through $R$, a contradiction.
		\end{itemize}
	\end{itemize}
	We conclude that in all subcases the statement of theorem follows.
	\par \textit{Case A.4: $\gamma'_{2},\gamma_{0}\gamma'_{2},\delta\in U_{1}$, but $\gamma_{2}\notin U_{1}$}.The arguments in this case are similar to the arguments in Case A.1, albeit easier since we can make a reduction to a conic instead of a cubic. We find that $|\Pi\cap\Omega_{2}|$ is contained in $\{q,q+1,q+2,2q,2q+1\}$. Details can be found in Appendix \ifthenelse{\equal{\versie}{arxiv}}{\ref{ap:th4.5}, see page \pageref{apB:A4}}{B in the arXiv version of this paper}.
	\par \textit{Case A.5: $\gamma'_{2},\gamma_{0}\gamma'_{2},\delta,\gamma_{2}\in U_{1}$, but $\delta\gamma'_{2}+\gamma'_{1}\gamma_{2}\notin U_{1}$}.  The arguments in this case are similar to the arguments in Case A.4, but a bit more involved. We need to handle the case $q=2$ separately, but we can confirm the statement of the Theorem both for $q=2$ and $q\geq3$. Details can be found in Appendix \ifthenelse{\equal{\versie}{arxiv}}{\ref{ap:th4.5}, see page \pageref{apB:A5}}{B in the arXiv version of this paper}.
	\par \textit{Case A.6: $\gamma'_{2},\gamma_{0}\gamma'_{2},\delta,\gamma_{2},\delta\gamma'_{2}+\gamma'_{1}\gamma_{2}\in U_{1}$}.  The arguments in this case are similar to the arguments in Case A.5. Also arguments from Case A.3 are used. We find that $|\Pi\cap\Omega_{2}|=q^{2}+1$. Details can be found in Appendix \ifthenelse{\equal{\versie}{arxiv}}{\ref{ap:th4.5}, see page \pageref{apB:A6}}{B in the arXiv version of this paper}.
	\paragraph*{Case B:} We assume that $\dim U_{1}=\dim U_{2}=3$ and $\dim\left\langle U_{1},U_{2}\right\rangle=4$. Note that it follows that $\gamma'_{1},\gamma'_{2}\notin\F_{q}$. Furthermore, if $\delta\in U_1$, then Equation \eqref{eq:1standaardC} clearly has $q$ solutions which implies that $P_0P_1$ is a $(q+1)$-secant as discussed in the intermezzo; similarly, we showed that if $\gamma_2\in U_2$, then Equation \eqref{eq:1standaardD} has $q$ solutions which implies that $P_0P_2$ is a $(q+1)$-secant. So, by the assumption of the theorem we have that $\delta\notin U_1$ and $\gamma_2\notin U_2$. It follows that Equations \eqref{eq:1standaardC} and \eqref{eq:1standaardD} do not have solutions in any subcase of Case B.
	\par We denote $\left\langle U_{1},U_{2}\right\rangle$ by $U$. Regarding a basis of $U$ we have the following possibilities.
	\begin{enumerate}[(i)]
		\item If $\gamma'_{1},\gamma'_{2}\notin\left\langle 1,\gamma_{0}\right\rangle$, then also $\gamma'_{2}\notin U_{1}$ since $U_{1}\neq U_{2}$. So, $\left\{1,\gamma_{0},\gamma'_{1},\gamma'_{2}\right\}$ is an $\F_{q}$-basis for $U$.
		\item If $\gamma'_{1}\notin\left\langle 1,\gamma_{0}\right\rangle$ but $\gamma'_{2}\in\left\langle 1,\gamma_{0}\right\rangle$, then $\gamma_{0}\gamma'_{2}\notin U_{1}$ since $U_{1}\neq U_{2}$. So, $\left\{1,\gamma_{0},\gamma'_{1},\gamma_{0}\gamma'_{2}\right\}$ is an $\F_{q}$-basis for $U$.
		\item If $\gamma'_{2}\notin\left\langle 1,\gamma_{0}\right\rangle$ but $\gamma'_{1}\in\left\langle 1,\gamma_{0}\right\rangle$, then $\gamma_{0}\gamma'_{1}\notin U_{1}$ since $U_{1}\neq U_{2}$. So, $\left\{1,\gamma_{0},\gamma'_{2},\gamma_{0}\gamma'_{1}\right\}$ is an $\F_{q}$-basis for $U$.
		\item If $\gamma'_{1},\gamma'_{2}\in\left\langle 1,\gamma_{0}\right\rangle$, then $U_1=\langle 1,\gamma_0,\gamma_0\gamma_1'\rangle=\langle 1,\gamma_0,\gamma_0^2\rangle=\langle 1,\gamma_0,\gamma_0\gamma_1'\rangle=U_2$, a contradiction since we assumed that $\dim U=4$.
	\end{enumerate}
	Recall that we argued in the beginning of Case A that the set of Equations \eqref{eq:1standaardA}--\eqref{eq:1standaardD} when interchanging $(\gamma'_{1},\delta)$ and $(\gamma'_{2},\gamma_{2})$ yields the same system of equations after renaming some of the variables $\mu_{i}$ and $\nu_{i}$. For this reason cases (ii) and (iii) are equivalent, and we only need to treat one of them. We now distinguish between the two subcases, corresponding to possibilities (i) and (iii). For Case B.1.1 we present the details. The arguments in the other subcases are similar and can be found in Appendix \ifthenelse{\equal{\versie}{arxiv}}{\ref{ap:th4.5}}{B in the arXiv version of this paper}.
	\par \textit{Case B.1: $\gamma'_{1},\gamma'_{2}\notin\left\langle 1,\gamma_{0}\right\rangle$}. Hence, $\left\{1,\gamma_{0},\gamma'_{1},\gamma'_{2}\right\}$ is an $\F_{q}$-basis for $U$. There are $a_{i},b_{i}\in\F_{q}$, $i=1,\dots,4$, such that
	\begin{align*}
		\gamma_0\gamma_1'&=a_1+a_2\gamma_0+a_3\gamma_1'\quad \text{and}\\
		\gamma_0\gamma_2'&=b_1+b_2\gamma_0+b_4\gamma_2'\;.
	\end{align*}
	It follows immediately that $a_{1}+a_{2}a_{3}\neq0\neq b_{1}+b_{2}b_{4}$ since $\gamma_{0},\gamma'_{1},\gamma'_{2}\notin\F_{q}$. We claim that also $a_3\neq b_4$. Suppose that $a_3=b_4$, then $\gamma_1'(\gamma_0-a_3)=a_1+a_2\gamma_0$ and $\gamma_2'(\gamma_0-a_3)=b_1+b_2\gamma_0$. It follows that $\gamma_1'/\gamma_2'=(a_1+a_2\gamma_0)/(b_1+b_2\gamma_0)$, and hence, that
	\[
		(b_1+b_2\gamma_0)\gamma_1'=(a_1+a_2\gamma_0)\gamma_2'\;,
	\] 
	which in turn yields that
	\[
		(a_1b_2-a_2b_1)+(b_1+a_3b_2)\gamma_1'-(a_1+a_2a_3)\gamma_2'=0\;.
	\]
	But we have seen in the intermezzo that $1,\gamma_1',\gamma_2'$ are linearly independent over $\F_{q}$. This implies that $a_1+a_2a_3=0$, a contradiction. We now make a further distinction based on $\delta$ and $\gamma_{2}$.
	\par \textit{Case B.1.1: $\delta\notin U$.} Hence, $\left\{1,\gamma_{0},\gamma'_{1},\gamma'_{2},\delta\right\}$ is an $\F_{q}$-basis for $\F_{q^{5}}$.	There are $c_{i},e_{i}\in\F_{q}$, $i=1,\dots,5$, such that 
	\begin{align*}
		-\gamma_2&=c_1+c_2\gamma_0+c_3\gamma_1'+c_4\gamma_2'+c_5\delta\quad\text{and}\\
		\delta\gamma'_{2}+\gamma'_{1}\gamma_{2}&=e_1+e_2\gamma_0+e_3\gamma_1'+e_4\gamma_2'+e_5\delta\;.
	\end{align*}
	\par Considering $\F_{q^{5}}$ as a vector space over $\F_{q}$, Equation \eqref{eq:1standaardA} is equivalent to the following system of equations:
	\begin{align}\label{eq:1standaardA211}
		\begin{cases}
			e_1+c_1\nu_3&=-a_1\mu_1-\nu_3\mu_2+b_1\nu_1+\mu_3\nu_2\\
			e_2+c_2\nu_3&=(-a_2+\nu_3)\mu_1+(b_2-\mu_3)\nu_1\\
			e_3+c_3\nu_3&=-a_3\mu_1+\mu_2\\
			e_4+c_4\nu_3&=b_4\nu_1-\nu_2\\
			e_5+c_5\nu_3&=\mu_3
		\end{cases}.
	\end{align}
	It is straightforward to see that there is a one-to-one correspondence between the solutions in $(\mu_{1},\mu_{2},\mu_{3},\nu_{1},\nu_{2},\nu_{3})$ of Equation \eqref{eq:1standaardA211} and the solutions in $(\mu_{1},\mu_{2},\nu_{1},\nu_{2},\nu_{3})$ of
	\begin{align}\label{eq:1standaardA211bis}
		\begin{cases}
			e_1+c_1\nu_3&=-a_1\mu_1-\nu_3\mu_2+b_1\nu_1+(e_5+c_5\nu_3)\nu_2\\
			e_2+c_2\nu_3&=(\nu_3-a_2)\mu_1+(b_2-e_5-c_5\nu_3)\nu_1\\
			e_3+c_3\nu_3&=-a_3\mu_1+\mu_2\\
			e_4+c_4\nu_3&=b_4\nu_1-\nu_2
		\end{cases}.
	\end{align}
	Given $\nu_{3}$, the system of equations in \eqref{eq:1standaardA211bis} has $0$, $1$ or at least $q$ solutions for $(\mu_{1},\mu_{2},\nu_{1},\nu_{2})$. Assume that for $\nu_{3}=\overline{\nu}$ the system of equations in \eqref{eq:1standaardA211bis} would have at least $q$ solutions. Then, looking at \eqref{eq:1algbis} with $(\mu_4,\mu_5,\nu_4,\nu_5)=(0,1,1,0)$, we see that for the corresponding points, we have $\varphi=\frac{e_{5}+c_{5}\overline{\nu}-\gamma'_{2}}{\gamma'_{1}-\overline{\nu}}$, so any two of these at least $q$ points determine a $(q+1)$-secant by Theorem \ref{qplusonesecant}, contradicting the assumption on $\Pi$. So, the system of equations in \eqref{eq:1standaardA211bis} has $0$ solutions or a unique solution in $(\mu_{1},\mu_{2},\nu_{1},\nu_{2})$. The coefficient matrix of the system of equations in \eqref{eq:1standaardA211bis} is
	\[
		A_{11}=
		\begin{pmatrix}
			-a_1&-\nu_3&b_1&e_{5}+c_{5}\nu_3\\
			\nu_3-a_2&0&b_2-e_{5}-c_{5}\nu_3&0\\
			-a_3&1&0&0\\
			0&0&b_4&-1
		\end{pmatrix}\,
	\]
	with
	\begin{align*}
		\det\left(A_{11}\right)&=c_{5}(a_3-b_4)\nu^{2}_{3}+((a_{3}-b_{4})e_{5}+(a_{1}+a_{2}b_{4})c_{5}-b_{1}-a_{3}b_{2})\nu_{3}\nonumber\\&\qquad+(a_{1}+a_{2}b_{4})e_{5}+a_{2}b_{1}-a_{1}b_{2}\nonumber\\ 
		&=D(\nu_3)\;.
	\end{align*}
	For a given $\nu_{3}$ the system of equations in \eqref{eq:1standaardA211bis} has a unique solution if $D(\nu_{3})\neq0$ and no solutions otherwise. We show that $D(\nu_{3})=0$ is a non-vanishing quadratic equation. Assume that it does vanish. Recall that $a_{3}\neq b_{4}$. Hence, we have that $c_{5}=0$ and $(a_{3}-b_{4})e_{5}-b_{1}-a_{3}b_{2}=0=(a_{1}+a_{2}b_{4})e_{5}+a_{2}b_{1}-a_{1}b_{2}$, and consequently
	\[
		(a_{3}-b_{4})(a_{2}b_{1}-a_{1}b_{2})=-(b_{1}+a_{3}b_{2})(a_{1}+a_{2}b_{4})\quad\Leftrightarrow\quad(a_{1}+a_{2}a_{3})(b_{1}+b_{2}b_{4})=0\;,
	\]
	contradicting the statements in the beginning of Case B.1. So, indeed $D(\nu_{3})=0$ is a non-vanishing quadratic equation. Consequently, it has at most two solutions, and thus the system of equations in \eqref{eq:1standaardA211bis} has $q-2$, $q-1$ or $q$ solutions. In particular, if $c_5=0$, this system has $q-1$ or $q$ solutions.
	\par Now, we look at Equation \eqref{eq:1standaardB}; it is equivalent to the following system of equations:
	\begin{align}\label{eq:1standaardB211}
		\begin{cases}
			c_1&=-\mu_2+\mu_4a_1\nu_1+b_1\nu_1\\
			c_2&=\mu_1+\mu_4a_2\nu_1+b_2\nu_1\\
			c_3&=-\mu_4\nu_2+a_3\mu_4\nu_1\\
			c_4&=-\nu_2+b_4\nu_1\\
			c_5&=-\mu_4
		\end{cases}.
	\end{align}
	Recall that $\mu_{4}\in\F^{*}_{q}$. So, the system of equations in \eqref{eq:1standaardB211} has no solutions if $c_{5}=0$. Hence, we assume in the discussion of this system of equations that $c_{5}\neq0$. We can see that there is a one-to-one correspondence between the solutions in $(\mu_{1},\mu_{2},\mu_{4},\nu_{1},\nu_{2})$ of Equation \eqref{eq:1standaardB211} and the solutions in $(\mu_{1},\mu_{2},\nu_{1},\nu_{2})$ of
	\begin{align}\label{eq:1standaardB211bis}
		\begin{cases}
			c_1&=-\mu_2-c_5a_1\nu_1+b_1\nu_1\\
			c_2&=\mu_1-c_5a_2\nu_1+b_2\nu_1\\
			c_3&=c_5\nu_2-a_3c_5\nu_1\\
			c_4&=-\nu_2+b_4\nu_1
		\end{cases}.
	\end{align}
	The coefficient matrix of this system is
	\[
		B_{11}=
		\begin{pmatrix}
			0&-1&b_1-c_5a_1&0\\
			1&0&b_2-c_5a_2&0\\
			0&0&-a_3c_5&c_5\\
			0&0&b_4&-1
		\end{pmatrix}\;.
	\]
	Now $\det\left(B_{11}\right)=0$ if and only if $c_5(a_3-b_4)=0$. Recall that $a_{3}\neq b_{4}$. As $c_5\neq 0$, we find a unique solution to the system of equations in \eqref{eq:1standaardB211bis}. So, Equation \eqref{eq:1standaardB211} has no solutions if $c_{5}=0$ and a unique solution if $c_{5}\neq0$.
	\par We find that the Equations \eqref{eq:1standaardA}, \eqref{eq:1standaardB}, \eqref{eq:1standaardC} and \eqref{eq:1standaardD} in total have between $q-1$ and $q+1$ solutions. Including the point $P_{0}$, we find that $|\Pi\cap\Omega_{2}|$ is contained in $\{q,q+1,q+2\}$.
	\par \textit{Case B.1.2: $\delta\in U$, but $\gamma_{2}\notin U$.}  The arguments in this case are similar to the arguments in Case B.1.1. We find that $|\Pi\cap\Omega_{2}|\in\{q,q+1\}$. Details can be found in Appendix \ifthenelse{\equal{\versie}{arxiv}}{\ref{ap:th4.5}, see page \pageref{apB:B1.2}}{B in the arXiv version of this paper}.
	\par \textit{Case B.1.3: $\delta,\gamma_{2}\in U$.}   The arguments in this case are similar to the arguments in Case B.1.1. We find that $|\Pi\cap\Omega_{2}|\in\{q,q^{2}+1\}$. Details can be found in Appendix \ifthenelse{\equal{\versie}{arxiv}}{\ref{ap:th4.5}, see page \pageref{apB:B1.3}}{B in the arXiv version of this paper}.
	\par \textit{Case B.2: $\gamma'_{2}\notin\left\langle 1,\gamma_{0}\right\rangle$ but $\gamma'_{1}\in\left\langle 1,\gamma_{0}\right\rangle$}. Hence, $\left\{1,\gamma_{0},\gamma'_{2},\gamma_{0}\gamma'_{1}\right\}$ is an $\F_{q}$-basis for $U$. There are $a_{i},b_{i}\in\F_{q}$, $i=1,\dots,4$, such that
	\begin{align*}
		\gamma_1'&=a_1+a_2\gamma_0\quad \text{and}\\
		\gamma_0\gamma_2'&=b_1+b_2\gamma_0+b_4\gamma_2'\;.
	\end{align*}
	It follows immediately that $a_{2}\neq0\neq b_{1}+b_{2}b_{4}$ since $\gamma_{0},\gamma'_{1},\gamma'_{2}\notin\F_{q}$. We now make a further distinction based on $\delta$ and $\gamma_{2}$.
	\par \textit{Case B.2.1: $\delta\notin U$.} The arguments in this case are similar to the arguments in Case B.1.1. We find that $|\Pi\cap\Omega_{2}|\in\{q,q+1,q+2\}$. Details can be found in Appendix \ifthenelse{\equal{\versie}{arxiv}}{\ref{ap:th4.5}, see page \pageref{apB:B2.1}}{B in the arXiv version of this paper}.
	\par \textit{Case B.2.2: $\delta\in U$, but $\gamma_{2}\notin U$.}  The arguments in this case are similar to the arguments in Case B.1.1, albeit easier since Equation \eqref{eq:1standaardB} has no solutions in this case. We find that $|\Pi\cap\Omega_{2}|\in\{q,q+1\}$. Details can be found in Appendix \ifthenelse{\equal{\versie}{arxiv}}{\ref{ap:th4.5}, see page \pageref{apB:B2.2}}{B in the arXiv version of this paper}.
	\par \textit{Case B.2.3: $\delta,\gamma_{2}\in U$.}  The arguments in this case are similar to the arguments in Case B.1.1. We find that $|\Pi\cap\Omega_{2}|\in\{q,q^{2}+1\}$. Details can be found in Appendix \ifthenelse{\equal{\versie}{arxiv}}{\ref{ap:th4.5}, see page \pageref{apB:B2.3}}{B in the arXiv version of this paper}.
\end{proof}

\paragraph{Acknowledgement:} The authors would like to thank John Sheekey for sharing some helpful computational results with us. In particular, he has done an independent verification of the weight distribution of linear sets of rank $5$ in $\PG(1,2^5)$ and $\PG(1,3^5)$ (Theorems \ref{qistwo} and \ref{qisthree}).
\par  This research was partially carried out when the first author was visiting the School of Mathematics and Statistics at the University of Canterbury in the framework of the project MFP-UOC1805 funded by the Marsden council. He wants to thank the School, and in particular the second author, for their hospitality.
\par The authors would like to thank the anonymous referees for their suggestions.
{\vspace{1cm}}

{\bf Address of the authors: }

\vspace{0.5cm}
\noindent
Maarten De Boeck\vspace*{0.2cm}\\
University of Rijeka\\
Faculty of Mathematics\\
Radmile Matej\v{c}i\'c 2\\
51000 Rijeka, Croatia\vspace*{0.2cm}\\
Ghent University\\
Department of Mathematics: Algebra and Geometry\\
Gent, Flanders, Belgium.

\vspace{0.5cm}
\noindent
Geertrui Van de Voorde\\
School of Mathematics and Statistics $\vert$ Te Kura P\=angarau\\
University of Canterbury $\vert$ Te Whare W\=ananga o Waitaha \\
Private bag 4800\\
8140 Christchurch, Canterbury, New Zealand $\vert$  \=Otautahi, Waitaha, Aotearoa

\ifthenelse{\equal{\versie}{arxiv}}{

\appendix
\section{Appendix: The proof of Theorem \texorpdfstring{\ref{eriseensecant}}{4.3}}\label{ap:th4.3}

\subsection*{Case A.3}\label{apA:A3}
\textit{In Case A.3 we assume that $a_{4}=0$ and $b_{4}\neq0$.} Clearly, in this case Equation \eqref{eq:q+1standaardB1bis} has no solutions. We now look at Equation \eqref{eq:q+1standaardA1bis}. From the third equation in \eqref{eq:q+1standaardA1bis} we then have $\nu_{3}=b^{-1}_{4}$, so we can look at the following system of equations:
\begin{align}\label{eq:q+1standaardA1.3}
\begin{cases}
\mu_{1}-(a_{5}\nu_{1}-b_{5}b^{-1}_{4})\mu_{3}=-(b_{1}b^{-1}_{4}-a_{1}\nu_{1})b^{-1}_{4}+a_{2}\nu_{1}-b_{2}b^{-1}_{4}\\
-b^{-1}_{4}\mu_{1}+\nu_{1}\mu_{3}=a_{3}\nu_{1}-b_{3}b^{-1}_{4}\\
\end{cases}.
\end{align}
For a given value of $\nu_{1}$ Equation \eqref{eq:q+1standaardA1.3} is a linear system of equations in $\mu_{1}$ and $\mu_{3}$ and has either 0, 1 or $q$ solutions. It has 0 or $q$ solutions iff
\begin{align}\label{eq:q+1-1.3ontaard}
\nu_{1}-\left(a_{5}\nu_{1}-b_{5}b^{-1}_{4}\right)b^{-1}_{4}=0\quad\Leftrightarrow\quad(b_{4}-a_{5})\nu_{1}+b_{5}b^{-1}_{4}=0\;.
\end{align}
This is a non-vanishing linear equation since we have shown before that	it is not possible that simultaneously $b_{5}=0=a_4$ and $a_{5}=b_{4}\neq0$. More precisely, \eqref{eq:q+1-1.3ontaard} has solution $\nu_{1}=\frac{b_{5}b^{-1}_{4}}{a_{5}-b_{4}}$ which exists if and only if $a_{5}\neq b_{4}$. Hence, for $q-1$ or $q$ values of $\nu_{1}$ Equation \eqref{eq:q+1standaardA1.3} has precisely one solution. If $\nu_{1}=\frac{b_{5}b^{-1}_{4}}{a_{5}-b_{4}}$, then \eqref{eq:q+1standaardA1.3} has $q$ solutions iff
\begin{align}
&&&-b_{4}\left(a_{3}\frac{b_{5}b^{-1}_{4}}{a_{5}-b_{4}}-b_{3}b^{-1}_{4}\right)=-\left(b_{1}b^{-1}_{4}-a_{1}\frac{b_{5}b^{-1}_{4}}{a_{5}-b_{4}}\right)b^{-1}_{4}+a_{2}\frac{b_{5}b^{-1}_{4}}{a_{5}-b_{4}}-b_{2}b^{-1}_{4}\nonumber\\
&\Leftrightarrow&0&=a_{3}b_{5}-(a_{5}-b_{4})b_{3}-(a_{5}-b_{4})b_{1}b^{-2}_{4}+a_{1}b_{5}b^{-2}_{4}+a_{2}b_{5}b^{-1}_{4}-(a_{5}-b_{4})b_{2}b^{-1}_{4}\nonumber\\
&\Leftrightarrow& 0&=b_{5}\left(a_{1}+a_{2}b_{4}+a_{3}b^{2}_{4}\right)-a_{5}\left(b_{1}+b_{2}b_{4}+b_{3}b^{2}_{4}\right)+b_{4}\left(b_{1}+b_{2}b_{4}+b_{3}b^{2}_{4}\right)\;.\label{eq:q+1-1.3qopl}
\end{align}
Now, we also have that
\begin{align*}
&&0&=\gamma\delta_{1}(\gamma-b_{4})-b_{1}-b_{2}\gamma-b_{3}\gamma^{2}-b_{5}\gamma\delta_{2}\\
&\Leftrightarrow&0&=\gamma\delta_{1}(\gamma-b_{4})(\gamma-a_{5})-(b_{1}+b_{2}\gamma+b_{3}\gamma^{2})(\gamma-a_{5})-b_{5}\gamma\delta_{2}(\gamma-a_{5})\\
&&&=\gamma\delta_{1}(\gamma-b_{4})(\gamma-a_{5})-(b_{1}+b_{2}\gamma+b_{3}\gamma^{2})(\gamma-a_{5})-b_{5}(a_{1}+a_{2}\gamma+a_{3}\gamma^{2})\;,
\end{align*}
where we have used \eqref{gammakwadraatdelta_2}. Now adding \eqref{eq:q+1-1.3qopl} to this, we find that
\begin{align}
0&=\gamma\delta_{1}(\gamma-b_{4})(\gamma-a_{5})+a_{5}\left(b_{2}(\gamma-b_{4})+b_{3}(\gamma^{2}-b^{2}_{4})\right)\nonumber\\&\qquad-\left(b_{1}(\gamma-b_{4})+b_{2}\left(\gamma^{2}-b^{2}_{4}\right)+b_{3}\left(\gamma^{3}-b^{3}_{4}\right)\right)-b_{5}\left(a_{2}(\gamma-b_{4})+a_{3}(\gamma^{2}-b^{2}_{4})\right)
\nonumber\\
&=\left(\gamma-b_{4}\right)\left[\gamma\delta_{1}(\gamma-a_{5})+a_{5}\left(b_{2}+b_{3}(\gamma+b_{4})\right)\right.\nonumber\\&\qquad\qquad\qquad\left.-\left(b_{1}+b_{2}\left(\gamma+b_{4}\right)+b_{3}\left(\gamma^{2}+b_{4}\gamma+b^{2}_{4}\right)\right)-b_{5}\left(a_{2}+a_{3}(\gamma+b_{4})\right)\right]\;.\label{eq:q+1-1.3qoplbis}
\end{align}
The first factor in \eqref{eq:q+1-1.3qoplbis} cannot be zero as $\gamma\notin\F_{q}$. Hence, the second factor in \eqref{eq:q+1-1.3qoplbis} must be zero. It follows that \begin{align*}
0&=\gamma^{2}\delta_{1}-a_{5}\gamma\delta_{1}+a_{5}b_{2}+a_{5}b_{3}\gamma+a_{5}b_{3}b_{4}\\&\qquad-b_{1}-b_{2}\gamma-b_{2}b_{4}-b_{3}\gamma^{2}-b_{3}b_{4}\gamma-b_{3}b^{2}_{4}-a_{2}b_{5}-a_{3}b_{5}\gamma-a_{3}b_{5}b_{4}\\
&=(b_{4}-a_{5})\gamma\delta_{1}+b_{5}\gamma\delta_{2}+\left(a_{5}b_{2}-a_{2}b_{5}+(a_{5}b_{3}-a_{3}b_{5}-b_{2})b_{4}-b_{3}b^{2}_{4}\right)\\&\qquad+\left(a_{5}b_{3}-a_{3}b_{5}-b_{3}b_{4}\right)\gamma
\end{align*}
where we have used \eqref{gammakwadraatdelta_1}. However, as $\{1,\gamma,\gamma\delta_{1},\gamma\delta_{2}\}$ is independent over $\F_{q}$, we have that $b_{4}=a_{5}$, contradicting the assumption. So, if $\nu_{1}=\frac{b_{5}b^{-1}_{4}}{a_{5}-b_{4}}$ then \eqref{eq:q+1standaardA1.3} has no solutions.
\par We conclude that in Case A.3, we have $q-1$ or $q$ solutions of Equation \eqref{eq:q+1standaardA1bis} and no solutions of Equation \eqref{eq:q+1standaardB1bis}, so in total there are $q-1$ or $q$ points in $(\Pi\setminus L)\cap\Omega_{2}$.

\subsection*{Case B.1.2}\label{apA:B1.2}
\textit{In Case B.1.2 we assume that $\gamma^{2}\delta_{1}\in\left\langle 1,\gamma,\gamma^{2},\gamma\delta_{2}\right\rangle_q$ and $\dim\left\langle 1,\gamma,\gamma^{2},\gamma\delta_{2},\gamma^{2}\delta_{2}\right\rangle_q=5$.} In other words, $\left\{1,\gamma,\gamma^{2},\gamma\delta_{2},\gamma^{2}\delta_{2}\right\}$ is an $\F_{q}$-basis for $\F_{q^{5}}$, and there are $b_{i}\in\F_{q}$, $i=0,\dots,3$, such that
\begin{align}
\gamma^{2}\delta_{1}&=b_{0}+b_{1}\gamma+b_{2}\gamma^{2}+b_{3}\gamma\delta_{2}\label{gammakwadraatdelta1B12}\;.
\end{align}
Note that $\dim\left\langle \gamma,\gamma^{2},\gamma^{2}\delta_{1},\gamma^{2}\delta_{2}\right\rangle_q=4$ since $\dim\left\langle 1,\gamma,\gamma\delta_{1},\gamma\delta_{2}\right\rangle_q=4$, and hence $(b_{0},b_{3})\neq(0,0)$. Now assume that $(b_{3},c_{3})=(0,0)$. Then $\gamma\delta_1=c_0+c_1\gamma+c_2\gamma^2$ and  $\gamma^2\delta_1=b_0+b_1\gamma+b_2\gamma^2$. Since  $\{1,\gamma,\gamma^{2},\gamma^{3}\}$ is an independent set over $\F_{q}$, we find that $c_2=0$, a contradiction. This implies that $(b_{3},c_{3})\neq (0,0).$
\par We are now ready to discuss the number of solutions to \eqref{eq:q+1standaardA} and \eqref{eq:q+1standaardB} in Case B.1.2. We see that in this case Equation \eqref{eq:q+1standaardB} is equivalent to the following system of equations:
\begin{align}\label{eq:q+1standaardB212}
\begin{cases}
b_{0}=-\mu_{4}\nu_{2}\\
b_{1}=\mu_{4}\nu_{1}-\mu_{2}\\
b_{2}=\mu_{1}\\
b_{3}=-\nu_{2}\\
0=\nu_{1}
\end{cases}
\Leftrightarrow\quad
\begin{cases}
\mu_{1}=b_{2}\\
\mu_{2}=b_{1}\\
\nu_{1}=0\\
\nu_{2}=-b_{3}\\
\mu_{4}b_{3}=b_{0}
\end{cases}.
\end{align}
Clearly, Equation \eqref{eq:q+1standaardB212} has 0, 1 or $q$ solutions and it can only have $q$ solutions if $b_{0}=b_{3}=0$, a contradiction. So, Equation \eqref{eq:q+1standaardB212} has 0 solutions or 1 solution in this case. The former only occurs if $b_{3}=0$ and $b_{0}\neq0$.
\par Equation \eqref{eq:q+1standaardA}, on the other hand, is equivalent to the following system of equations:
\begin{align}\label{eq:q+1standaardA212}
\begin{cases}
0=\mu_{2}-b_{0}\nu_{3}+c_{0}\\
0=\mu_{3}\nu_{2}-\mu_{2}\nu_{3}-\mu_{1}-b_{1}\nu_{3}+c_{1}\\
0=\mu_{1}\nu_{3}-\mu_{3}\nu_{1}-b_{2}\nu_{3}+c_{2}\\
0=-\nu_{2}-b_{3}\nu_{3}+c_{3}\\
0=\nu_{1}
\end{cases}
\Leftrightarrow\quad
\begin{cases}
\mu_{2}=b_{0}\nu_{3}-c_{0}\\
\nu_{1}=0\\
\nu_{2}=-b_{3}\nu_{3}+c_{3}\\
0=\mu_{3}\nu_{2}-\mu_{2}\nu_{3}-\mu_{1}-b_{1}\nu_{3}+c_{1}\\
0=\mu_{1}\nu_{3}-b_{2}\nu_{3}+c_{2}
\end{cases}.
\end{align}
It is straightforward to see that there is a one-to-one correspondence between the solutions in $(\mu_{1},\mu_{2},\mu_{3},\nu_{1},\nu_{2},\nu_{3})$ of Equation \eqref{eq:q+1standaardA212} and the solutions in $(\mu_{1},\mu_{3},\nu_{3})$ of
\begin{align}\label{eq:q+1standaardA212bis}
\begin{cases}
\mu_{1}+(b_{3}\nu_{3}-c_{3})\mu_{3}=(c_{0}-b_{0}\nu_{3})\nu_{3}-b_{1}\nu_{3}+c_{1}\\
\mu_{1}\nu_{3}=b_{2}\nu_{3}-c_{2}
\end{cases}.
\end{align}
For a given value of $\nu_{3}$ Equation \eqref{eq:q+1standaardA212bis} is a linear system of equations in $\mu_{1}$ and $\mu_{3}$ and has either 0, 1 or $q$ solutions. It has 0 or $q$ solutions iff
\begin{align}\label{eq:q+1-2.1.2ontaard}
\left(b_{3}\nu_{3}-c_{3}\right)\nu_{3}=0\;.
\end{align}
This is a non-vanishing quadratic or linear equation since $(b_{3},c_{3})\neq(0,0)$. More precisely, \eqref{eq:q+1-2.1.2ontaard} has solutions $\nu_{3}=0$ and $\nu_{3}=c_{3}b^{-1}_{3}$. Note that the latter solution only exists if $b_{3}\neq0$. Hence for $q-2$ or $q-1$ values of $\nu_{3}$ Equation \eqref{eq:q+1standaardA212bis} has precisely one solution. If $\nu_{3}=0$, then \eqref{eq:q+1standaardA212bis} has no solutions since $c_{2}\neq0$. If $\nu_{3}=c_{3}b^{-1}_{3}\neq 0$, then \eqref{eq:q+1standaardA212bis} has $q$ solutions if and only if
\begin{align}
&b_{2}-c_{2}\frac{b_{3}}{c_{3}}=\left(c_{0}-b_{0}\frac{c_{3}}{b_{3}}\right)\frac{c_{3}}{b_{3}}-b_{1}\frac{c_{3}}{b_{3}}+c_{1}\nonumber\\
\Leftrightarrow\qquad &b_{0}+\left(b_{1}-c_{0}\right)\left(\frac{b_{3}}{c_{3}}\right)+\left(b_{2}-c_{1}\right)\left(\frac{b_{3}}{c_{3}}\right)^{2}-c_{2}\left(\frac{b_{3}}{c_{3}}\right)^{3}=0\;.\label{eq:q+1-2.1.2qopl}
\end{align}
Now, using \eqref{gammadelta1B1} and \eqref{gammakwadraatdelta1B12} we also have that
\begin{align*}
0&=\gamma^{2}\delta_{1}-\gamma(\gamma\delta_{1})=b_{0}+(b_{1}-c_{0})\gamma+(b_{2}-c_{1})\gamma^{2}+b_{3}\gamma\delta_{2}-c_{2}\gamma^{3}-c_{3}\gamma^{2}\delta_{2}\;,
\end{align*}
and subtracting \eqref{eq:q+1-2.1.2qopl} from this, we find that
\begin{align}
0&=(b_{1}-c_{0})\left(\gamma-\frac{b_{3}}{c_{3}}\right)+(b_{2}-c_{1})\left(\gamma^{2}-\left(\frac{b_{3}}{c_{3}}\right)^{2}\right)-c_{2}\left(\gamma^{3}-\left(\frac{b_{3}}{c_{3}}\right)^{3}\right)\nonumber\\&\qquad-c_{3}\gamma\delta_{2}\left(\gamma-\frac{b_{3}}{c_{3}}\right)\nonumber\\
&=\left(\gamma-\frac{b_{3}}{c_{3}}\right)\left[b_{1}-c_{0}+(b_{2}-c_{1})\left(\gamma+\frac{b_{3}}{c_{3}}\right)-c_{2}\left(\gamma^{2}+\frac{b_{3}}{c_{3}}\gamma+\left(\frac{b_{3}}{c_{3}}\right)^{2}\right)-c_{3}\gamma\delta_{2}\right]\;.\label{eq:q+1-2.1.2qoplbis}
\end{align}
The first factor in \eqref{eq:q+1-2.1.2qoplbis} cannot be zero as $\gamma\notin\F_{q}$. Hence, the second factor in \eqref{eq:q+1-2.1.2qoplbis} must be zero. However, as $\{1,\gamma,\gamma^{2},\gamma\delta_{2}\}$ is independent over $\F_{q}$, we have that $c_{2}=0$, a contradiction. So, if $\nu_{3}=c_{3}b^{-1}_{3}$ then \eqref{eq:q+1standaardA212bis} has no solutions.
\par We conclude that in Case B.1.2, we have $q-2$, $q-1$ or $q$ solutions of Equation \eqref{eq:q+1standaardA212bis} and at most 1 solution of Equation \eqref{eq:q+1standaardB212}. Now, recall that \eqref{eq:q+1standaardB212} has no solutions if and only if $b_{3}=0$ and $b_{0}\neq0$, but we showed above that \eqref{eq:q+1standaardA212bis} has $q-1$ solutions if $b_{3}=0$ and $b_{0}\neq0$. So, in total there are at least $q-1$ and at most $q+1$ points in $(\Pi\setminus L)\cap\Omega_{2}$ in this case.

\subsection*{Case B.2.1}\label{apA:B2.1}

\textit{In Case B.2.1 we assume that $\dim\left\langle 1,\gamma,\gamma^{2},\gamma\delta_{1},\gamma^{2}\delta_{1}\right\rangle_q=5$.} In other words, we assume that $\left\{1,\gamma,\gamma^{2},\gamma\delta_{1},\gamma^{2}\delta_{1}\right\}$ is an $\F_{q}$-basis for $\F_{q^{5}}$. Then, there are $a_{i}\in\F_{q}$, $i=1,\dots,5$, such that
\begin{align}
\gamma^{2}\delta_{2}&=a_{0}+a_{1}\gamma+a_{2}\gamma^{2}+a_{3}\gamma\delta_{1}+a_{4}\gamma^{2}\delta_{1}\label{gammakwadraatdelta2B21}\;.
\end{align}
Note that $\dim\left\langle \gamma,\gamma^{2},\gamma^{2}\delta_{1},\gamma^{2}\delta_{2}\right\rangle_q=4$ since $\dim\left\langle 1,\gamma,\gamma\delta_{1},\gamma\delta_{2}\right\rangle_q=4$, and hence $(a_{0},a_{3})\neq(0,0)$.  Note that also $(a_{3},a_{4})\neq(0,0)$ since $(a_{3},a_{4})=(0,0)$ implies that also $d_{2}=0$, a contradiction. In the last implication we use that $\{1,\gamma,\gamma^{2},\gamma^{3}\}$ is an independent set over $\F_{q}$.
\par We are now ready to discuss the number of solutions to \eqref{eq:q+1standaardA} and \eqref{eq:q+1standaardB} in Case B.2.1. We see that in this case Equation \eqref{eq:q+1standaardB} is equivalent to the following system of equations:
\begin{align}\label{eq:q+1standaardB221}
\begin{cases}
0=-\mu_{4}\nu_{2}-d_{0}\nu_{2}+a_{0}\nu_{1}\\
0=\mu_{4}\nu_{1}-\mu_{2}-d_{1}\nu_{2}+a_{1}\nu_{1}\\
0=\mu_{1}-d_{2}\nu_{2}+a_{2}\nu_{1}\\
0=a_{3}\nu_{1}\\
1=a_{4}\nu_{1}
\end{cases}
\Leftrightarrow\quad
\begin{cases}
\mu_{1}=d_{2}\nu_{2}-a_{2}\nu_{1}\\
\mu_{2}=\mu_{4}\nu_{1}-d_{1}\nu_{2}+a_{1}\nu_{1}\\
0=-\mu_{4}\nu_{2}-d_{0}\nu_{2}+a_{0}\nu_{1}\\
0=a_{3}\nu_{1}\\
1=a_{4}\nu_{1}
\end{cases}.
\end{align}
It is clear that \eqref{eq:q+1standaardB221} has no solutions if $a_{3}\neq0$ or if $a_{4}=0$. So, we assume now that $a_{4}\neq0$ and $a_{3}=0$, and hence also $a_{0}\neq0$. Then, it is straightforward that there is a one-to-one correspondence between the solutions in $(\mu_{1},\mu_{2},\mu_{4},\nu_{1},\nu_{2})$ of Equation \eqref{eq:q+1standaardB221} and the solutions in $(\mu_{4},\nu_{2})$ of
\begin{align}\label{eq:q+1standaardB221bis}
0=-\mu_{4}\nu_{2}-d_{0}\nu_{2}+a_{0}a^{-1}_{4}\;.
\end{align}
For every value of $\nu_{2}\in\F^{*}_{q}$ there is a unique solution for $\mu_{4}$, and for $\nu_{2}=0$ Equation \eqref{eq:q+1standaardB221bis} has no solution since $a_{0}\neq0$. So, Equation \eqref{eq:q+1standaardB221bis} has 0 or $q-1$ solutions in this case. The former occurs if $a_{3}\neq0$ or if $a_{4}=0$, and the latter occurs if $a_{3}=0$ and $a_{4}\neq0$.
\par Equation \eqref{eq:q+1standaardA} is equivalent to the following system of equations:
\begin{align}\label{eq:q+1standaardA221}
\begin{cases}
0=\mu_{2}-d_{0}\nu_{2}+a_{0}\nu_{1}\\
0=\mu_{3}\nu_{2}-\mu_{2}\nu_{3}-\mu_{1}-d_{1}\nu_{2}+a_{1}\nu_{1}\\
0=\mu_{1}\nu_{3}-\mu_{3}\nu_{1}-d_{2}\nu_{2}+a_{2}\nu_{1}\\
0=a_{3}\nu_{1}+1\\
0=a_{4}\nu_{1}-\nu_{3}
\end{cases}
\Leftrightarrow\quad
\begin{cases}
\mu_{2}=d_{0}\nu_{2}-a_{0}\nu_{1}\\
\nu_{3}=a_{4}\nu_{1}\\
-1=a_{3}\nu_{1}\\
0=\mu_{3}\nu_{2}-\mu_{2}\nu_{3}-\mu_{1}-d_{1}\nu_{2}+a_{1}\nu_{1}\\
0=\mu_{1}\nu_{3}-\mu_{3}\nu_{1}-d_{2}\nu_{2}+a_{2}\nu_{1}
\end{cases}.
\end{align}
It is clear that \eqref{eq:q+1standaardA221} has no solutions if $a_{3}=0$. So, we assume now that $a_{3}\neq0$. Then, it is straightforward that there is a one-to-one correspondence between the solutions in $(\mu_{1},\mu_{2},\mu_{3},\nu_{1},\nu_{2},\nu_{3})$ of Equation \eqref{eq:q+1standaardA221} and the solutions in $(\mu_{1},\mu_{3},\nu_{2})$ of
\begin{align}\label{eq:q+1standaardA221bis}
\begin{cases}
\mu_{1}-\mu_{3}\nu_{2}=a_{4}a^{-1}_{3}\left(d_{0}\nu_{2}+a_{0}a^{-1}_{3}\right)-d_{1}\nu_{2}-a_{1}a^{-1}_{3}\\
-a_{4}a^{-1}_{3}\mu_{1}+a^{-1}_{3}\mu_{3}=d_{2}\nu_{2}+a_{2}a^{-1}_{3}
\end{cases}.		
\end{align}
For a given value of $\nu_{2}$ Equation \eqref{eq:q+1standaardA221bis} is a linear system of equations in $\mu_{1}$ and $\mu_{3}$ and has either 0, 1 or $q$ solutions. It has 0 or $q$ solutions iff
\begin{align}\label{eq:q+1-2.2.1ontaard}
a^{-1}_{3}\left(1-a_{4}\nu_{2}\right)=0\;.
\end{align}
This is a non-vanishing linear equation since $a_{3}\neq0$. More precisely, \eqref{eq:q+1-2.2.1ontaard} has no solutions if $a_{4}=0$ and one solution $\nu_{2}=a^{-1}_{4}$ if $a_{4}\neq0$. Hence for $q-1$ or $q$ values of $\nu_{2}$ Equation \eqref{eq:q+1standaardA221bis} has precisely one solution. If $\nu_{2}=a^{-1}_{4}$, then \eqref{eq:q+1standaardA221bis} has $q$ solutions iff
\begin{align}
&-\frac{a_{4}}{a_{3}}\left[\frac{a_{4}}{a_{3}}\left(\frac{d_{0}}{a_{4}}+\frac{a_{0}}{a_{3}}\right)-\frac{d_{1}}{a_{4}}-\frac{a_{1}}{a_{3}}\right]=\frac{d_{2}}{a_{4}}+\frac{a_{2}}{a_{3}}\nonumber\\
\Leftrightarrow\qquad &a_{0}-\left(a_{1}-d_{0}\right)\left(\frac{a_{3}}{a_{4}}\right)+\left(a_{2}-d_{1}\right)\left(\frac{a_{3}}{a_{4}}\right)^{2}+d_{2}\left(\frac{a_{3}}{a_{4}}\right)^{3}=0\;.\label{eq:q+1-2.2.1qopl}
\end{align}
Now, we also have from \eqref{gammadelta2B2} and \eqref{gammakwadraatdelta2B21}  that
\begin{align*}
0&=\gamma^{2}\delta_{2}-\gamma(\gamma\delta_{2})=a_{0}+(a_{1}-d_{0})\gamma+(a_{2}-d_{1})\gamma^{2}-d_{2}\gamma^{3}+a_{3}\gamma\delta_{1}+a_{4}\gamma^{2}\delta_{1}\;,
\end{align*}
and subtracting \eqref{eq:q+1-2.2.1qopl} from this, we find that
\begin{align}
0&=(a_{1}-d_{0})\left(\gamma+\frac{a_{3}}{a_{4}}\right)+(a_{2}-d_{1})\left(\gamma^{2}-\left(\frac{a_{3}}{a_{4}}\right)^{2}\right)-d_{2}\left(\gamma^{3}+\left(\frac{a_{3}}{a_{4}}\right)^{3}\right)\nonumber\\&\qquad-a_{4}\gamma\delta_{1}\left(\gamma+\frac{a_{3}}{a_{4}}\right)\nonumber\\
&=\left(\gamma+\frac{a_{3}}{a_{4}}\right)\left[a_{1}-d_{0}+(a_{2}-d_{1})\left(\gamma-\frac{a_{3}}{a_{4}}\right)-d_{2}\left(\gamma^{2}-\left(\frac{a_{3}}{a_{4}}\right)\gamma+\left(\frac{a_{3}}{a_{4}}\right)^{2}\right)-a_{4}\gamma\delta_{1}\right]\;.\label{eq:q+1-2.2.1qoplbis}
\end{align}
The first factor in \eqref{eq:q+1-2.2.1qoplbis} cannot be zero as $\gamma\notin\F_{q}$. Hence, the second factor in \eqref{eq:q+1-2.2.1qoplbis} must be zero. However, as $\{1,\gamma,\gamma^{2},\gamma\delta_{1}\}$ is independent over $\F_{q}$, we have that $d_{2}=0$, a contradiction. So, if $\nu_{2}=a^{-1}_{4}$ then \eqref{eq:q+1standaardA212bis} has no solutions.
\par We conclude that in Case B.2.1, we have $0$ solutions of Equation \eqref{eq:q+1standaardB221} and $q-1$ solutions of Equation \eqref{eq:q+1standaardA221} if $a_{3}\neq0\neq a_{4}$, we have $0$ solutions of Equation \eqref{eq:q+1standaardB221} and $q$ solutions of Equation \eqref{eq:q+1standaardA221} if $a_{3}\neq0=a_{4}$, and we have $q-1$ solutions of Equation \eqref{eq:q+1standaardB221} and $0$ solutions of Equation \eqref{eq:q+1standaardA221} if $a_{3}=0\neq a_{4}$. Recall that $(a_{3},a_{4})\neq(0,0)$, so in Case B.2.1 there are in total $q-1$ or $q$ points in $(\Pi\setminus L)\cap\Omega_{2}$.

\subsection*{Case B.2.2}\label{apA:B2.2}

\textit{In Case B.2.2 we assume that $\dim\left\langle 1,\gamma,\gamma^{2},\gamma\delta_{1},\gamma^2\delta\right\rangle_q\neq 5$.} Recall that it is not possible that both $\gamma^{2}\delta_{1}$ and $\gamma^{2}\delta_{2}$ are contained in $\left\langle 1,\gamma,\gamma^{2},\gamma\delta_{1}\right\rangle_q$, hence, we know that $\dim\left\langle 1,\gamma,\gamma^{2},\gamma\delta_{1},\gamma^{2}\delta_{2}\right\rangle_q=5$. 
In other words, $\left\{1,\gamma,\gamma^{2},\gamma\delta_{1},\gamma^{2}\delta_{2}\right\}$ is an $\F_{q}$-basis for $\F_{q^{5}}$, and there are $b_{i}\in\F_{q}$, $i=0,\dots,3$, such that
\begin{align}
\gamma^{2}\delta_{1}&=b_{0}+b_{1}\gamma+b_{2}\gamma^{2}+b_{3}\gamma\delta_{1}\label{gammakwadraatdelta1B22}\;.
\end{align}
Note that $\dim\left\langle \gamma,\gamma^{2},\gamma^{2}\delta_{1},\gamma^{2}\delta_{2}\right\rangle_q=4$ since $\dim\left\langle 1,\gamma,\gamma\delta_{1},\gamma\delta_{2}\right\rangle_q=4$, and hence $(b_{0},b_{3})\neq(0,0)$.
\par We are now ready to discuss the number of solutions to \eqref{eq:q+1standaardA} and \eqref{eq:q+1standaardB} in Case B.2.2. We see that in this case Equation \eqref{eq:q+1standaardB} is equivalent to the following system of equations:
\begin{align}\label{eq:q+1standaardB222}
\begin{cases}
0=-\mu_{4}\nu_{2}-d_{0}\nu_{2}-b_{0}\\
0=\mu_{4}\nu_{1}-\mu_{2}-d_{1}\nu_{2}-b_{1}\\
0=\mu_{1}-d_{2}\nu_{2}-b_{2}\\
0=-b_{3}\\
0=\nu_{1}
\end{cases}
\Leftrightarrow\quad
\begin{cases}
0=\mu_{4}\nu_{2}+d_{0}\nu_{2}+b_{0}\\
\mu_{1}=d_{2}\nu_{2}+b_{2}\\
\mu_{2}=-d_{1}\nu_{2}-b_{1}\\
\nu_{1}=0\\
0=b_{3}
\end{cases}.
\end{align}
It is clear that \eqref{eq:q+1standaardB222} has no solutions if $b_{3}\neq0$. So, we assume now that $b_{3}=0$, and hence also $b_{0}\neq0$. Then, it is straightforward that there is a one-to-one correspondence between the solutions in $(\mu_{1},\mu_{2},\mu_{4},\nu_{1},\nu_{2})$ of Equation \eqref{eq:q+1standaardB222} and the solutions in $(\mu_{4},\nu_{2})$ of
\begin{align}\label{eq:q+1standaardB222bis}
0=\mu_{4}\nu_{2}+d_{0}\nu_{2}+b_{0}\;.
\end{align}
For every value of $\nu_{2}\in\F^{*}_{q}$ there is a unique solution for $\mu_{4}$, and for $\nu_{2}=0$ Equation \eqref{eq:q+1standaardB222bis} has no solution since $b_{0}\neq0$. So, Equation \eqref{eq:q+1standaardB222bis} has $0$ or $q-1$ solutions in this case. The former occurs if $b_{3}\neq0$ and the latter if $b_{3}=0$.
\par Equation \eqref{eq:q+1standaardA} is equivalent to the following system of equations:
\begin{align}\label{eq:q+1standaardA222}
\begin{cases}
0=\mu_{2}-d_{0}\nu_{2}-b_{0}\nu_{3}\\
0=\mu_{3}\nu_{2}-\mu_{2}\nu_{3}-\mu_{1}-d_{1}\nu_{2}-b_{1}\nu_{3}\\
0=\mu_{1}\nu_{3}-\mu_{3}\nu_{1}-d_{2}\nu_{2}-b_{2}\nu_{3}\\
0=1-b_{3}\nu_{3}\\
0=\nu_{1}
\end{cases}
\Leftrightarrow\quad
\begin{cases}
\mu_{2}=d_{0}\nu_{2}+b_{0}\nu_{3}\\
\nu_{1}=0\\
0=\mu_{3}\nu_{2}-\mu_{2}\nu_{3}-\mu_{1}-d_{1}\nu_{2}-b_{1}\nu_{3}\\
0=\mu_{1}\nu_{3}-d_{2}\nu_{2}-b_{2}\nu_{3}\\
1=b_{3}\nu_{3}
\end{cases}.
\end{align}
It is clear that \eqref{eq:q+1standaardA222} has no solutions if $b_{3}=0$. So, we assume now that $b_{3}\neq0$. Then, it is straightforward that there is a one-to-one correspondence between the solutions in $(\mu_{1},\mu_{2},\mu_{3},\nu_{1},\nu_{2},\nu_{3})$ of Equation \eqref{eq:q+1standaardA222} and the solutions in $(\mu_{1},\mu_{3},\nu_{2})$ of
\begin{align}\label{eq:q+1standaardA222bis}
\begin{cases}
\mu_{1}-\nu_{2}\mu_{3}=-(d_{0}\nu_{2}+b_{0}b^{-1}_{3})b^{-1}_{3}-d_{1}\nu_{2}-b_{1}b^{-1}_{3}\\
\mu_{1}b^{-1}_{3}=d_{2}\nu_{2}+b_{2}b^{-1}_{3}\\
\end{cases}.		
\end{align}
For a given value of $\nu_{2}$ Equation \eqref{eq:q+1standaardA222bis} is a linear system of equations in $\mu_{1}$ and $\mu_{3}$ and has either 0, 1 or $q$ solutions. It has 0 or $q$ solutions if and only if $\nu_{2}=0$. Hence for $q-1$ values of $\nu_{2}$ Equation \eqref{eq:q+1standaardA222bis} has precisely one solution. If $\nu_{2}=0$, then \eqref{eq:q+1standaardA222bis} has $q$ solutions iff
\begin{align}\label{eq:q+1-2.2.2qopl}
b_{3}\left(b_{2}b^{-1}_{3}\right)=-(b_{0}b^{-1}_{3})b^{-1}_{3}-b_{1}b^{-1}_{3}\quad\Leftrightarrow\quad b_{2}b^{2}_{3}+b_{1}b_{3}+b_{0}=0\;.
\end{align}
Subtracting this from the expression for $\gamma^{2}\delta_{1}$ from \eqref{gammakwadraatdelta1B22}, we find that
\begin{align}\label{eq:q+1-2.2.2qoplbis}
\gamma^{2}\delta_{1}&=-b_{1}b_{3}-b_{2}b^{2}_{3}+b_{1}\gamma+b_{2}\gamma^{2}+b_{3}\gamma\delta_{1}\quad\Leftrightarrow\quad 0=\left(\gamma-b_{3}\right)\left(b_{1}+b_{2}(\gamma+b_{3})-\gamma\delta_{1}\right)\;.
\end{align}
The first factor in \eqref{eq:q+1-2.2.2qoplbis} cannot be zero as $\gamma\notin\F_{q}$, and the second factor in \eqref{eq:q+1-2.2.2qoplbis} cannot be zero as $\{1,\gamma,\gamma\delta_{1}\}$ is an independent set over $\F_{q}$. So, if $\nu_{2}=0$ then \eqref{eq:q+1standaardA222bis} has no solutions.
\par We conclude that in Case B.2.2, we have $0$ solutions of Equation \eqref{eq:q+1standaardA222} and $q-1$ solutions of Equation \eqref{eq:q+1standaardB222} if $b_{3}=0$, and we have $q-1$ solutions of Equation \eqref{eq:q+1standaardA222} and $0$ solutions of Equation \eqref{eq:q+1standaardB222} if $b_{3}\neq0$. So, in this case there are in total $q-1$ points in $(\Pi\setminus L)\cap\Omega_{2}$.

\section{Appendix: The proof of Theorem \texorpdfstring{\ref{arcs}}{4.5}}\label{ap:th4.5}

\subsection*{Case A.2}\label{apB:A2}
\textit{In Case A.2 we assume that $\gamma'_{2}\in U_{1}$, but $\gamma_{0}\gamma'_{2}\notin U_{1}$}. Hence, we assume in this case that $\dim\left\langle 1,\gamma_{0},\gamma'_{1},\gamma_{0}\gamma'_{1},\gamma_{0}\gamma'_{2}\right\rangle_q=5$, in other words $\left\{1,\gamma_{0},\gamma'_{1},\gamma_{0}\gamma'_{1},\gamma_{0}\gamma'_{2}\right\}$ is an $\F_{q}$-basis for $\F_{q^{5}}$. Then, there are $a_{i},b_{i},c_{i},d_{i}\in\F_{q}$, $i=1,\dots,5$, such that
\begin{align*}
\gamma'_{2}&=a_{1}+a_{2}\gamma_{0}+a_{3}\gamma'_{1}+a_{4}\gamma_{0}\gamma'_{1}\;,\\
\delta&=b_{1}+b_{2}\gamma_{0}+b_{3}\gamma'_{1}+b_{4}\gamma_{0}\gamma'_{1}+b_{5}\gamma_{0}\gamma'_{2}\;,\\
\gamma_{2}&=c_{1}+c_{2}\gamma_{0}+c_{3}\gamma'_{1}+c_{4}\gamma_{0}\gamma'_{1}+c_{5}\gamma_{0}\gamma'_{2}\quad\text{and}\\
\delta\gamma'_{2}+\gamma'_{1}\gamma_{2}&=d_{1}+d_{2}\gamma_{0}+d_{3}\gamma'_{1}+d_{4}\gamma_{0}\gamma'_{1}+d_{5}\gamma_{0}\gamma'_{2}\;.
\end{align*}
We saw in the intermezzo that $\left\langle P_{0},P_{2}\right\rangle$ is a $(q+1)$-secant if $\gamma_{2}\in U_{2}$ and $\dim U_{2}=3$. In this case, these conditions are fulfilled if and only if $a_{3}=a_{4}=0$ and $c_{3}=c_{4}=0$. So, we may assume that $(a_{3},a_{4},c_{3},c_{4})\neq(0,0,0,0)$.
\par Considering $\F_{q^{5}}$ as a vector space over $\F_{q}$, Equation \eqref{eq:1standaardA} is equivalent to the following system of equations:
\begin{align}\label{eq:1standaardA12}
\begin{cases}
d_{1}=\mu_{3}\nu_{2}-\mu_{2}\nu_{3}-\nu_{2}a_{1}+\mu_{3}b_{1}+\nu_{3}c_{1}\\
d_{2}=\mu_{1}\nu_{3}-\mu_{3}\nu_{1}-\nu_{2}a_{2}+\mu_{3}b_{2}+\nu_{3}c_{2}\\
d_{3}=\mu_{2}-\nu_{2}a_{3}+\mu_{3}b_{3}+\nu_{3}c_{3}\\
d_{4}=-\mu_{1}-\nu_{2}a_{4}+\mu_{3}b_{4}+\nu_{3}c_{4}\\
d_{5}=\nu_{1}+\mu_{3}b_{5}+\nu_{3}c_{5}
\end{cases}
\!\Leftrightarrow\ 
\begin{cases}
\mu_{1}=-d_{4}-\nu_{2}a_{4}+\mu_{3}b_{4}+\nu_{3}c_{4}\\
\mu_{2}=d_{3}+\nu_{2}a_{3}-\mu_{3}b_{3}-\nu_{3}c_{3}\\
\nu_{1}=d_{5}-\mu_{3}b_{5}-\nu_{3}c_{5}\\
d_{1}=\mu_{3}\nu_{2}-\mu_{2}\nu_{3}-\nu_{2}a_{1}+\mu_{3}b_{1}+\nu_{3}c_{1}\\
d_{2}=\mu_{1}\nu_{3}-\mu_{3}\nu_{1}-\nu_{2}a_{2}+\mu_{3}b_{2}+\nu_{3}c_{2}
\end{cases}\!\!\!\!.
\end{align}
It is straightforward to see that there is a one-to-one correspondence between the solutions in $(\mu_{1},\mu_{2},\mu_{3},\nu_{1},\nu_{2},\nu_{3})$ of Equation \eqref{eq:1standaardA12} and the solutions in $(\mu_{3},\nu_{2},\nu_{3})$ of
\begin{align}\label{eq:1standaardA12bis}
&\begin{cases}
d_{1}=\mu_{3}\nu_{2}-(d_{3}+\nu_{2}a_{3}-\mu_{3}b_{3}-\nu_{3}c_{3})\nu_{3}-\nu_{2}a_{1}+\mu_{3}b_{1}+\nu_{3}c_{1}\\
d_{2}=(-d_{4}-\nu_{2}a_{4}+\mu_{3}b_{4}+\nu_{3}c_{4})\nu_{3}-\mu_{3}(d_{5}-\mu_{3}b_{5}-\nu_{3}c_{5})-\nu_{2}a_{2}+\mu_{3}b_{2}+\nu_{3}c_{2}
\end{cases}\nonumber\\
\Leftrightarrow\quad
&\begin{cases}
L_{1}(\mu_{3},\nu_{3})\nu_{2}=C_{1}(\mu_{3},\nu_{3})\\
L_{2}(\mu_{3},\nu_{3})\nu_{2}=C_{2}(\mu_{3},\nu_{3})
\end{cases}
\end{align}
with
\begin{align*}
L_{1}(\mu_{3},\nu_{3})&=-\mu_{3}+a_{3}\nu_{3}+a_{1}\;,\\
L_{2}(\mu_{3},\nu_{3})&=a_{4}\nu_{3}+a_{2}\;,\\
C_{1}(\mu_{3},\nu_{3})&=b_{3}\mu_{3}\nu_{3}+c_{3}\nu^{2}_{3}+b_{1}\mu_{3}+(c_{1}-d_{3})\nu_{3}-d_{1}\;\text{ and}\\
C_{2}(\mu_{3},\nu_{3})&=b_{5}\mu^{2}_{3}+(b_{4}+c_{5})\mu_{3}\nu_{3}+c_{4}\nu^{2}_{3}+(b_{2}-d_{5})\mu_{3}+(c_{2}-d_{4})\nu_{3}-d_{2}\;.
\end{align*}
Given $\mu_{3}$ and $\nu_{3}$, the system of equations in \eqref{eq:1standaardA12bis} has $0$, $1$ or $q$ solutions for $\nu_{1}$. Assume that for $(\mu_{3},\nu_{3})=(\overline{\mu_{3}},\overline{\nu_{3}})$ the system of equations in \eqref{eq:1standaardA12bis} would have $q$ solutions. Then, looking at \eqref{eq:1algbis} with $(\mu_4,\mu_5,\nu_4,\nu_5)=(0,1,1,0)$, we see that for the $q$ corresponding points, we have $\varphi=\frac{\overline{\mu_{3}}-\gamma'_{2}}{\gamma'_{1}-\overline{\nu_{3}}}$. Hence any two of these $q$ points determine a $(q+1)$-secant by Theorem \ref{qplusonesecant}, contradicting the assumption on $\Pi$. So, \eqref{eq:1standaardA12bis} has either 0 solutions or a unique solution in $\nu_{1}$, and the latter occurs iff 
\begin{multline}\label{eq:1cubic12}
F(\mu_{3},\nu_{3})=L_{1}(\mu_{3},\nu_{3})C_{2}(\mu_{3},\nu_{3})-L_{2}(\mu_{3},\nu_{3})C_{1}(\mu_{3},\nu_{3})=0\\ \wedge\quad\left(L_{1}(\mu_{3},\nu_{3}),L_{2}(\mu_{3},\nu_{3})\right)\neq(0,0)\;.
\end{multline}
The equation $F(\mu_{3},\nu_{3})=0$ determines a cubic curve $C$ in the $(\mu_{3},\nu_{3})$-plane $\pi\cong\AG(2,q)$. We embed this affine plane in the projective plane $\PG(2,q)$ by adding the line at infinity $\ell_{\infty}$ and extend $C$ to the cubic curve $\overline{C}$ by going to a homogeneous equation $\overline{F}(\mu_{3},\nu_{3},\rho)=0$. Analogously, we define the homogeneous functions $\overline{L_{1}}(\mu_{3},\nu_{3},\rho)$, $\overline{L_{2}}(\mu_{3},\nu_{3},\rho)$, $\overline{C_{1}}(\mu_{3},\nu_{3},\rho)$, and $\overline{C_{2}}(\mu_{3},\nu_{3},\rho)$. Note that neither $\overline{L_{1}}$ nor $\overline{L_{2}}$ can be identically zero; the former is obvious, and in case $L_{2}\equiv0$ we would have that $\gamma'_{2}=a_{1}+a_{3}\gamma'_{1}$, hence that $\{1,\gamma'_{1},\gamma'_{2}\}$ is a linearly dependent set over $\F_{q}$ which forces the existence of a $(q^2+q+1)$-secant to $\Pi$ as seen in the intermezzo.  The lines defined by $\overline{L_{1}}=0$ and $\overline{L_{2}}=0$ in the projective plane $\PG(2,q)$, clearly do not coincide. So, these lines have precisely one intersection point $R$, which is on $\ell_{\infty}$ if and only if $a_{4}=0$. It is clear that $R$ is on the cubic curve $\overline{F}=0$.
\par We denote the number of points on $\overline{C}$ by $N$ and the number of points on $\overline{C}\cap\ell_{\infty}$ by $N_{\infty}$. Furthermore, we set $\varepsilon=1$ if $R$ is an affine point, $\varepsilon=0$ if $R\in\ell_{\infty}$. We find that Equation \eqref{eq:1cubic12}, and hence also Equation \eqref{eq:1standaardA11}, has $N-N_{\infty}-\varepsilon$ solutions.
\par Now, we look at Equation \eqref{eq:1standaardB}; it is equivalent to the following system of equations:
\begin{align}\label{eq:1standaardB12}
\begin{cases}
-c_{1}=-\mu_{2}-a_{1}\nu_{2}-b_{1}\mu_{4}\\
-c_{2}=\mu_{1}-a_{2}\nu_{2}-b_{2}\mu_{4}\\
-c_{3}=-\nu_{2}\mu_{4}-a_{3}\nu_{2}-b_{3}\mu_{4}\\
-c_{4}=\mu_{4}\nu_{1}-a_{4}\nu_{2}-b_{4}\mu_{4}\\
-c_{5}=\nu_{1}-b_{5}\mu_{4}
\end{cases}
\Leftrightarrow\quad
\begin{cases}
\mu_{1}=a_{2}\nu_{2}+b_{2}\mu_{4}-c_{2}\\
\mu_{2}=c_{1}-a_{1}\nu_{2}-b_{1}\mu_{4}\\
\nu_{1}=b_{5}\mu_{4}-c_{5}\\
-c_{3}=-\nu_{2}\mu_{4}-a_{3}\nu_{2}-b_{3}\mu_{4}\\
-c_{4}=\mu_{4}\nu_{1}-a_{4}\nu_{2}-b_{4}\mu_{4}
\end{cases}\;.
\end{align}
Recall that $\mu_{4}\in\F^{*}_{q}$. It is straightforward that there is a one-to-one correspondence between the solutions in $(\mu_{1},\mu_{2},\mu_{4},\nu_{1},\nu_{2})$ of Equation \eqref{eq:1standaardB12} and the solutions in $(\mu_{4},\nu_{1})$ of
\begin{align}\label{eq:1standaardB12bis}
&\begin{cases}
-c_{3}=-\nu_{2}\mu_{4}-a_{3}\nu_{2}-b_{3}\mu_{4}\\
-c_{4}=\mu_{4}(b_{5}\mu_{4}-c_{5})-a_{4}\nu_{2}-b_{4}\mu_{4}
\end{cases}
&\Leftrightarrow\quad
&\begin{cases}
\left(\mu_{4}+a_{3}\right)\nu_{2}=-b_{3}\mu_{4}+c_{3}\\
a_{4}\nu_{2}=b_{5}\mu^{2}_{4}-(b_{4}+c_{5})\mu_{4}+c_{4}
\end{cases}\nonumber\\
&&\Leftrightarrow\quad
&\begin{cases}
-\overline{L_{1}}\left(-\mu_{4},1,0\right)\nu_{2}=\overline{C_{1}}\left(-\mu_{4},1,0\right)\\
-\overline{L_{2}}\left(-\mu_{4},1,0\right)\nu_{2}=\overline{C_{2}}\left(-\mu_{4},1,0\right)
\end{cases}.
\end{align}
Given $\mu_{4}$, the system of equations in \eqref{eq:1standaardB12bis} has $0$, $1$ or $q$ solutions for $\nu_{2}$. Assume that for $\mu_{4}=\overline{\mu_{4}}$ the system of equations in \eqref{eq:1standaardB12bis} would have $q$ solutions. Then, looking at \eqref{eq:1algbis} with$(\mu_{3},\nu_{3},\nu_{4},\mu_{5},\nu_{5})=(0,1,0,1,0)$, we see that for the $q$ corresponding points, we have $\varphi=\overline{\mu_{4}}\gamma'_{1}+\gamma'_{2}$. Hence, any two of these $q$ points determine a $(q+1)$-secant by Theorem \ref{qplusonesecant}, contradicting the assumption on $\Pi$. So, \eqref{eq:1standaardB12bis} has either 0 solutions or a unique solution in $\nu_{2}$, and the latter occurs if and only if $\mu_{4}\in\F^{*}_{q}$ fulfils
\begin{align}\label{eq:1oneindig12}
0&=\overline{L_{1}}\left(-\mu_{4},1,0\right)\overline{C_{2}}\left(-\mu_{4},1,0\right)-\overline{L_{2}}\left(-\mu_{4},1,0\right)\overline{C_{1}}\left(-\mu_{4},1,0\right)\nonumber\\
&=\overline{F}(-\mu_{4},1,0)\;,
\end{align}
and simultaneously $(\overline{L_{1}}\left(-\mu_{4},1,0\right),\overline{L_{2}}\left(-\mu_{4},1,0\right))\neq(0,0)$. However, if 
\[
(\overline{L_{1}}\left(-\mu_{4},1,0\right),\overline{L_{2}}\left(-\mu_{4},1,0\right))=(0,0)\;,
\] then $R=\langle(-\mu_{4},1,0)\rangle\in\ell_{\infty}$. So, the solutions of \eqref{eq:1oneindig12} correspond to the points of $\overline{C}\cap(\ell_{\infty}\setminus\{R,\langle(1,0,0)\rangle,\langle(0,1,0)\rangle\})$.
\par Now we look at Equations \eqref{eq:1standaardC} and \eqref{eq:1standaardD}. It is immediately clear that Equation \eqref{eq:1standaardC} has 1 solution if $b_{5}=0$ and no solutions otherwise. Equation \eqref{eq:1standaardD} is equivalent to
\begin{multline*}
c_{1}+c_{2}\gamma_{0}+c_{3}\gamma'_{1}+c_{4}\gamma_{0}\gamma'_{1}+c_{5}\gamma_{0}\gamma'_{2}\\=(\nu_{2}+a_{1}\mu_{2})-(\nu_{1}-a_{2}\mu_{2})\gamma_{0}+a_{3}\mu_{2}\gamma'_{1}+a_{4}\mu_{2}\gamma_{0}\gamma'_{1}-\mu_{1}\gamma_{0}\gamma'_{2}\;.
\end{multline*}
This equation has one solution if $a_{3}c_{4}=a_{4}c_{3}$ and $(a_{3},a_{4})\neq(0,0)$ and no solutions otherwise; recall (from the beginning of this case) that it is not possible that $a_{3}=a_{4}=c_{3}=c_{4}=0$. So, it has a solution if and only if $(0,1,0)\in\overline{C}$, but $R\neq(0,1,0)$.
\par We note that $R$ cannot be the point $(1,0,0)$, and we conclude that regardless of the behaviour of $b_{5}$ and $a_{3}c_{4}-a_{4}c_{3}$ and the position of $R$, the total number of solutions of the Equations \eqref{eq:1standaardB}, \eqref{eq:1standaardC} and \eqref{eq:1standaardD} together equals $N_{\infty}-(1-\varepsilon)$. Including the solutions from Equation \eqref{eq:1standaardA} and the point $P_{0}$, we find that $\Pi\cap\Omega_{2}$ contains $N$ points. Note that if $\Pi\cap\Omega_{2}$ contains $N=q^2+q+1$ points, then there are at least $2$ points with the same type (recall that there are $q^{2}+1$ $G$-orbits), and hence, there is a $(q+1)$-secant by Theorem \ref{qplusonesecant}. This implies that $\overline{F}$ does not vanish.
\par By the analysis of Equations \eqref{eq:1standaardA12bis} and \eqref{eq:1standaardB12bis} we know that $R$ cannot be on both conics $\overline{C_{1}}=0$ and $\overline{C_{2}}=0$ if it is an affine point or a point on $\ell_{\infty}\setminus\{\langle(1,0,0)\rangle,\langle(0,1,0)\rangle\}$. Similarly, if $R=\langle(0,1,0)\rangle$ this point cannot be on both conics since it is not possible that $a_{3}=a_{4}=c_{3}=c_{4}=0$. Recall that $R\neq(1,0,0)$. Hence, we can apply Lemma \ref{cubic} to the cubic $\overline{C}$ and the statement of the theorem follows.

\subsection*{Case A.4}\label{apB:A4}
\textit{In Case A.4 we assume that $\gamma'_{2},\gamma_{0}\gamma'_{2},\delta\in U_{1}$, but $\gamma_{2}\notin U_{1}$}. Hence, we assume that $\dim\left\langle 1,\gamma_{0},\gamma'_{1},\gamma_{0}\gamma'_{1},\gamma_{2}\right\rangle_q=5$, in other words $\left\{1,\gamma_{0},\gamma'_{1},\gamma_{0}\gamma'_{1},\gamma_{2}\right\}$ is an $\F_{q}$-basis for $\F_{q^{5}}$. Note that in this case $U_{2}\leq U_{1}$. Then, there are $a_{i},b_{i},c_{i},d_{i}\in\F_{q}$, $i=1,\dots,5$, such that
\begin{align*}
\gamma'_{2}&=a_{1}+a_{2}\gamma_{0}+a_{3}\gamma'_{1}+a_{4}\gamma_{0}\gamma'_{1}\;,\\
\gamma_{0}\gamma'_{2}&=b_{1}+b_{2}\gamma_{0}+b_{3}\gamma'_{1}+b_{4}\gamma_{0}\gamma'_{1}\;,\\
\delta&=c_{1}+c_{2}\gamma_{0}+c_{3}\gamma'_{1}+c_{4}\gamma_{0}\gamma'_{1}\quad\text{and}\\
\delta\gamma'_{2}+\gamma'_{1}\gamma_{2}&=d_{1}+d_{2}\gamma_{0}+d_{3}\gamma'_{1}+d_{4}\gamma_{0}\gamma'_{1}+d_{5}\gamma_{2}\;.
\end{align*}
\par Considering $\F_{q^{5}}$ as a vector space over $\F_{q}$, Equation \eqref{eq:1standaardA} is equivalent to the following system of equations:
\begin{align}\label{eq:1standaardA14}
\begin{cases}
d_{1}=\mu_{3}\nu_{2}-\mu_{2}\nu_{3}-\nu_{2}a_{1}+\nu_{1}b_{1}+\mu_{3}c_{1}\\
d_{2}=\mu_{1}\nu_{3}-\mu_{3}\nu_{1}-\nu_{2}a_{2}+\nu_{1}b_{2}+\mu_{3}c_{2}\\
d_{3}=\mu_{2}-\nu_{2}a_{3}+\nu_{1}b_{3}+\mu_{3}c_{3}\\
d_{4}=-\mu_{1}-\nu_{2}a_{4}+\nu_{1}b_{4}+\mu_{3}c_{4}\\
d_{5}=\nu_{3}
\end{cases}
\!\Leftrightarrow\ 
\begin{cases}
\mu_{1}=-\nu_{2}a_{4}+\nu_{1}b_{4}+\mu_{3}c_{4}-d_{4}\\
\mu_{2}=\nu_{2}a_{3}-\nu_{1}b_{3}-\mu_{3}c_{3}+d_{3}\\
\nu_{3}=d_{5}\\
d_{1}=\mu_{3}\nu_{2}-\mu_{2}\nu_{3}-\nu_{2}a_{1}+\nu_{1}b_{1}+\mu_{3}c_{1}\\
d_{2}=\mu_{1}\nu_{3}-\mu_{3}\nu_{1}-\nu_{2}a_{2}+\nu_{1}b_{2}+\mu_{3}c_{2}
\end{cases}\!\!\!\!.
\end{align}
It is straightforward to see that there is a one-to-one correspondence between the solutions in $(\mu_{1},\mu_{2},\mu_{3},\nu_{1},\nu_{2},\nu_{3})$ of Equation \eqref{eq:1standaardA14} and the solutions in $(\nu_{1},\nu_{2},\mu_{3})$ of
\begin{align}\label{eq:1standaardA14bis}
&\begin{cases}
d_{1}=\mu_{3}\nu_{2}-(\nu_{2}a_{3}-\nu_{1}b_{3}-\mu_{3}c_{3}+d_{3})d_{5}-\nu_{2}a_{1}+\nu_{1}b_{1}+\mu_{3}c_{1}\\
d_{2}=(-\nu_{2}a_{4}+\nu_{1}b_{4}+\mu_{3}c_{4}-d_{4})d_{5}-\mu_{3}\nu_{1}-\nu_{2}a_{2}+\nu_{1}b_{2}+\mu_{3}c_{2}
\end{cases}\nonumber\\
\Leftrightarrow\quad
&\begin{cases}
-L_{1}(\nu_{1},\nu_{2})\mu_{3}=C_{1}(\nu_{1},\nu_{2})\\
-L_{2}(\nu_{1},\nu_{2})\mu_{3}=C_{2}(\nu_{1},\nu_{2})
\end{cases}
\end{align}
with
\begin{align*}
L_{1}(\nu_{1},\nu_{2})&=\nu_{2}+c_{1}+c_{3}d_{5}\;,\\
L_{2}(\nu_{1},\nu_{2})&=-\nu_{1}+c_{2}+c_{4}d_{5}\;,\\
C_{1}(\nu_{1},\nu_{2})&=(b_{1}+b_{3}d_{5})\nu_{1}-(a_{1}+a_{3}d_{5})\nu_{2}-d_{1}-d_{3}d_{5}\;\text{ and}\\
C_{2}(\nu_{1},\nu_{2})&=(b_{2}+b_{4}d_{5})\nu_{1}-(a_{2}+a_{4}d_{5})\nu_{2}-d_{2}-d_{4}d_{5}\;.
\end{align*}
Given $\nu_{1}$ and $\nu_{2}$, the system of equations in \eqref{eq:1standaardA14bis} has $0$, $1$ or $q$ solutions for $\mu_{3}$. Assume that for $(\nu_{1},\nu_{2})=(\overline{\nu_{1}},\overline{\nu_{2}})$ the system of equations in \eqref{eq:1standaardA14bis} would have $q$ solutions. Then, $L_{1}(\overline{\nu_{1}},\overline{\nu_{2}})=L_{2}(\overline{\nu_{1}},\overline{\nu_{2}})=C_{1}(\overline{\nu_{1}},\overline{\nu_{2}})=C_{2}(\overline{\nu_{1}},\overline{\nu_{2}})=0$. It follows that $\overline{\nu_{1}}=c_{2}+c_{4}d_{5}$ and $\overline{\nu_{2}}=-c_{1}-c_{3}d_{5}$, and we find that
\begin{align}
0&=\left(a_{3}c_{3}+b_{3}c_{4}\right)d^{2}_{5}+\left(a_{1}c_{3}+a_{3}c_{1}+b_{1}c_{4}+b_{3}c_{2}-d_{3}\right)d_{5}+a_{1}c_{1}+b_{1}c_{2}-d_{1}\quad\text{and}\label{eq:1standaardA14qopl1}\\
0&=\left(a_{4}c_{3}+b_{4}c_{4}\right)d^{2}_{5}+\left(a_{2}c_{3}+a_{4}c_{1}+b_{2}c_{4}+b_{4}c_{2}-d_{4}\right)d_{5}+a_{2}c_{1}+b_{2}c_{2}-d_{2}\label{eq:1standaardA14qopl2}\;.
\end{align}
Now, we also have that
\begin{align*}
\gamma_{2}\left(\gamma'_{1}-d_{5}\right)&=d_{1}+d_{2}\gamma_{0}+d_{3}\gamma'_{1}+d_{4}\gamma_{0}\gamma'_{1}-\gamma'_{2}\delta\\
&=d_{1}+d_{2}\gamma_{0}+d_{3}\gamma'_{1}+d_{4}\gamma_{0}\gamma'_{1}-\left(a_{1}+a_{2}\gamma_{0}+a_{3}\gamma'_{1}+a_{4}\gamma_{0}\gamma'_{1}\right)\left(c_{1}+c_{3}\gamma'_{1}\right)\\&\qquad-\left(b_{1}+b_{2}\gamma_{0}+b_{3}\gamma'_{1}+b_{4}\gamma_{0}\gamma'_{1}\right)\left(c_{2}+c_{4}\gamma'_{1}\right)\\
&=\left(d_{1}-a_{1}c_{1}-b_{1}c_{2}\right)+\left(d_{3}-a_{1}c_{3}-a_{3}c_{1}-b_{1}c_{4}-b_{3}c_{2}\right)\gamma'_{1}-\left(a_{3}c_{3}+b_{3}c_{4}\right)\gamma'^{2}_{1}\\&\qquad+\left(d_{2}-a_{2}c_{1}-b_{2}c_{2}\right)\gamma_{0}+\left(d_{4}-a_{2}c_{3}-a_{4}c_{1}-b_{2}c_{4}-b_{4}c_{2}\right)\gamma_{0}\gamma'_{1}\\&\qquad-\left(a_{4}c_{3}+b_{4}c_{4}\right)\gamma_{0}\gamma'^{2}_{1}\;.
\end{align*}
Substituting Equations \eqref{eq:1standaardA14qopl1} and \eqref{eq:1standaardA14qopl2} in this expression, we find that
\begin{align*}
0&=(\gamma'_{1}-d_{5})\left[\gamma_{2}+\left(a_{1}c_{3}+a_{3}c_{1}+b_{1}c_{4}+b_{3}c_{2}-d_{3}\right)+\left(a_{3}c_{3}+b_{3}c_{4}\right)(\gamma'_{1}+d_{5})\right.\\&\qquad\qquad\qquad\left.+\left(a_{2}c_{3}+a_{4}c_{1}+b_{2}c_{4}+b_{4}c_{2}-d_{4}\right)\gamma_{0}+\left(a_{4}c_{3}+b_{4}c_{4}\right)\gamma_{0}(\gamma'_{1}+d_{5})\right].
\end{align*}
Since $\gamma'_{1}\notin\F_{q}$ and $\gamma_{2}\notin U_{1}$ by the assumption, we find a contradiction. So, the system of equations in \eqref{eq:1standaardA14bis} has $0$ solutions or a unique solution in $\mu_{3}$, and the latter occurs if and only if 
\begin{multline}\label{eq:1cubic14}
F(\nu_{1},\nu_{2})=L_{1}(\nu_{1},\nu_{2})C_{2}(\nu_{1},\nu_{2})-L_{2}(\nu_{1},\nu_{2})C_{1}(\nu_{1},\nu_{2})=0\\ \wedge\quad\left(L_{1}(\nu_{1},\nu_{2}),L_{2}(\nu_{1},\nu_{2})\right)\neq(0,0)\;.
\end{multline}
The equation $F(\nu_{1},\nu_{2})=0$ determines a conic $C$ in the $(\nu_{1},\nu_{2})$-plane $\pi\cong\AG(2,q)$. We embed this affine plane in the projective plane $\overline{\pi}\cong\PG(2,q)$ by adding the line at infinity $\ell_{\infty}$ and extend $C$ to the conic $\overline{C}$ by going to a homogeneous equation $\overline{F}(\nu_{1},\nu_{2},\rho)=0$. Analogously, we define the homogeneous functions $\overline{L_{1}}(\nu_{1},\nu_{2},\rho)$ and $\overline{L_{2}}(\nu_{1},\nu_{2},\rho)$.
\par Note that both $\overline{L_{1}}$ and $\overline{L_{2}}$ cannot be identically zero. Moreover, $\overline{L_{1}}=0$ and $\overline{L_{2}}=0$ determine different lines in $\pi$ and their intersection point $R=(c_{2}+c_{4}d_{5},-c_{1}-c_{3}d_{5},1)$ is not on $\ell_{\infty}$. It is clear that $R$ is on the conic $\overline{C}$. Moreover, this conic cannot decompose in two lines (either over $\F_{q}$ or an algebraic extension) through $R$, since for this to happen we should have that $R$ is also on $C_{1}=0$ and $C_{2}=0$, but we have showed before that an affine point cannot be on all four lines $L_{1}=0$, $L_{2}=0$, $C_{1}=0$ and $C_{2}=0$.	
\par We know that the number of points on $\overline{C}$ equals $q+1$ or $2q+1$. Subtracting $R$ and the number of points on $\overline{C}\cap\ell_{\infty}$ we find that the number of solutions of Equation \eqref{eq:1cubic14}, and hence also of Equation \eqref{eq:1standaardA14}, is contained in $\{q-2,q-1,q,2q-2,2q-1\}$.
\par Now we look at Equations \eqref{eq:1standaardB}, \eqref{eq:1standaardC} and \eqref{eq:1standaardD}. It is immediately clear that by the assumption of Case A.4 Equation \eqref{eq:1standaardB} has no solutions, Equation \eqref{eq:1standaardC} has a unique solution and Equation \eqref{eq:1standaardD} has no solutions. Including the point $P_{0}$, we find that $|\Pi\cap\Omega_{2}|$ is contained in $\{q,q+1,q+2,2q,2q+1\}$.

\subsection*{Case A.5}\label{apB:A5}
\textit{In Case A.5 we assume that $\gamma'_{2},\gamma_{0}\gamma'_{2},\delta,\gamma_{2}\in U_{1}$, but $\delta\gamma'_{2}+\gamma'_{1}\gamma_{2}\notin U_{1}$}. By this assumption we have $\dim\left\langle 1,\gamma_{0},\gamma'_{1},\gamma_{0}\gamma'_{1},\delta\gamma'_{2}+\gamma'_{1}\gamma_{2}\right\rangle_q=5$, in other words $\left\{1,\gamma_{0},\gamma'_{1},\gamma_{0}\gamma'_{1},\delta\gamma'_{2}+\gamma'_{1}\gamma_{2}\right\}$ is an $\F_{q}$-basis for $\F_{q^{5}}$. Note that in this case $U_{2}\leq U_{1}$ Then, there are $a_{i},b_{i},c_{i},d_{i}\in\F_{q}$, $i=1,\dots,5$, such that
\begin{align*}
\gamma'_{2}&=a_{1}+a_{2}\gamma_{0}+a_{3}\gamma'_{1}+a_{4}\gamma_{0}\gamma'_{1}\;,\\
\gamma_{0}\gamma'_{2}&=b_{1}+b_{2}\gamma_{0}+b_{3}\gamma'_{1}+b_{4}\gamma_{0}\gamma'_{1}\;,\\
\delta&=c_{1}+c_{2}\gamma_{0}+c_{3}\gamma'_{1}+c_{4}\gamma_{0}\gamma'_{1}\quad\text{and}\\
\gamma_{2}&=d_{1}+d_{2}\gamma_{0}+d_{3}\gamma'_{1}+d_{4}\gamma_{0}\gamma'_{1}\;.
\end{align*}
We mentioned before that $\left\langle P_{0},P_{2}\right\rangle$ is a $(q+1)$-secant if $\gamma_{2}\in U_{2}$ and $\dim U_{2}=3$. In this case, these conditions are fulfilled if and only if $\rk\left(\begin{smallmatrix} a_{3} & b_{3} & d_{3}\\a_{4} & b_{4} & d_{4}\end{smallmatrix}\right)=1$. Suppose that $\rk\left(\begin{smallmatrix} a_{3} & b_{3} & d_{3}\\a_{4} & b_{4} & d_{4}\end{smallmatrix}\right)=0$. This implies that $\gamma_2' \in \F_q$, which in turn implies that $\{1,\gamma_1',\gamma_2'\}$ is not an $\F_q$-independent set. As seen in the intermezzo, this shows that there is a $(q^2+q+1)$-secant. We conclude that 	$\rk\left(\begin{smallmatrix} a_{3} & b_{3} & d_{3}\\a_{4} & b_{4} & d_{4}\end{smallmatrix}\right)=2$. We also note that
\begin{align*}
\delta\gamma'_{2}+\gamma'_{1}\gamma_{2}&=\left(c_{1}+c_{3}\gamma'_{1}\right)\left(a_{1}+a_{2}\gamma_{0}+a_{3}\gamma'_{1}+a_{4}\gamma_{0}\gamma'_{1}\right)\\
&\qquad+\left(c_{2}+c_{4}\gamma'_{1}\right)\left(b_{1}+b_{2}\gamma_{0}+b_{3}\gamma'_{1}+b_{4}\gamma_{0}\gamma'_{1}\right)+\gamma'_{1}\left(d_{1}+d_{2}\gamma_{0}+d_{3}\gamma'_{1}+d_{4}\gamma_{0}\gamma'_{1}\right)\\
&=\left(a_{1}c_{1}+b_{1}c_{2}\right)+\left(a_{2}c_{1}+b_{2}c_{2}\right)\gamma_{0}+\left(a_{1}c_{3}+a_{3}c_{1}+b_{1}c_{4}+b_{3}c_{2}+d_{1}\right)\gamma'_{1}\\&\qquad+\left(a_{2}c_{3}+a_{4}c_{1}+b_{2}c_{4}+b_{4}c_{2}+d_{2}\right)\gamma_{0}\gamma'_{1}+\left(a_{3}c_{3}+b_{3}c_{4}+d_{3}\right)\gamma'^{2}_{1}\\&\qquad+\left(a_{4}c_{3}+b_{4}c_{4}+d_{4}\right)\gamma_{0}\gamma'^{2}_{1}\;.
\end{align*}
so we cannot have that
\begin{align}\label{niet00}
\left(a_{3}c_{3}+b_{3}c_{4}+d_{3},a_{4}c_{3}+b_{4}c_{4}+d_{4}\right)\neq (0,0)
\end{align}
by the assumption that $\delta\gamma'_{2}+\gamma'_{1}\gamma_{2}\notin U_{1}$.
\par It is obvious that Equation \eqref{eq:1standaardA} has no solutions in this case. We look at Equation \eqref{eq:1standaardB}; it is equivalent to the following system of equations:
\begin{align}\label{eq:1standaardB15}
\begin{cases}
-d_{1}=-\mu_{2}-a_{1}\nu_{2}+b_{1}\nu_{1}-c_{1}\mu_{4}\\
-d_{2}=\mu_{1}-a_{2}\nu_{2}+b_{2}\nu_{1}-c_{2}\mu_{4}\\
-d_{3}=-\mu_{4}\nu_{2}-a_{3}\nu_{2}+b_{3}\nu_{1}-c_{3}\mu_{4}\\
-d_{4}=\mu_{4}\nu_{1}-a_{4}\nu_{2}+b_{4}\nu_{1}-c_{4}\mu_{4}\\
\end{cases}
\Leftrightarrow\quad
\begin{cases}
\mu_{1}=a_{2}\nu_{2}-b_{2}\nu_{1}+c_{2}\mu_{4}-d_{2}\\
\mu_{2}=-a_{1}\nu_{2}+b_{1}\nu_{1}-c_{1}\mu_{4}+d_{1}\\
-d_{3}=-\mu_{4}\nu_{2}-a_{3}\nu_{2}+b_{3}\nu_{1}-c_{3}\mu_{4}\\
-d_{4}=\mu_{4}\nu_{1}-a_{4}\nu_{2}+b_{4}\nu_{1}-c_{4}\mu_{4}\\
\end{cases}\;.
\end{align}
It is straightforward to see that there is a one-to-one correspondence between the solutions in $(\mu_{1},\mu_{2},\mu_{4},\nu_{1},\nu_{2})$ of Equation \eqref{eq:1standaardB13} and the solutions in $(\mu_{4},\nu_{1},\nu_{2})$ of
\begin{align}\label{eq:1standaardB15bis}
&\begin{cases}
-d_{3}=-\mu_{4}\nu_{2}-a_{3}\nu_{2}+b_{3}\nu_{1}-c_{3}\mu_{4}\\
-d_{4}=\mu_{4}\nu_{1}-a_{4}\nu_{2}+b_{4}\nu_{1}-c_{4}\mu_{4}\\
\end{cases}
&\Leftrightarrow\quad
&\begin{cases}
(\nu_{2}+c_{3})\mu_{4}=b_{3}\nu_{1}-a_{3}\nu_{2}+d_{3}\\
(-\nu_{1}+c_{4})\mu_{4}=b_{4}\nu_{1}-a_{4}\nu_{2}+d_{4}
\end{cases}\nonumber\\
&&\Leftrightarrow\quad
&\begin{cases}
L_{1}\left(\nu_{1},\nu_{2}\right)\mu_{4}=C_{1}\left(\nu_{1},\nu_{2}\right)\\
L_{2}\left(\nu_{1},\nu_{2}\right)\mu_{4}=C_{2}\left(\nu_{1},\nu_{2}\right)
\end{cases}
\end{align}
with
\begin{align*}
L_{1}(\nu_{1},\nu_{2})&=\nu_{2}+c_{3}\;,\\
L_{2}(\nu_{1},\nu_{2})&=-\nu_{1}+c_{4}\;,\\
C_{1}(\nu_{1},\nu_{2})&=b_{3}\nu_{1}-a_{3}\nu_{2}+d_{3}\;\text{ and}\\
C_{2}(\nu_{1},\nu_{2})&=b_{4}\nu_{1}-a_{4}\nu_{2}+d_{4}\;.
\end{align*}
The system of equations in \eqref{eq:1standaardB15bis} has $0$, $1$ or $q$ solutions for $\mu_{4}$.  Assume that for $(\nu_{1},\nu_{2})=(\overline{\nu_{1}},\overline{\nu_{2}})$ the system of equations in \eqref{eq:1standaardB15bis} would have $q$ solutions. Then, $L_{1}(\overline{\nu_{1}},\overline{\nu_{2}})=L_{2}(\overline{\nu_{1}},\overline{\nu_{2}})=C_{1}(\overline{\nu_{1}},\overline{\nu_{2}})=C_{2}(\overline{\nu_{1}},\overline{\nu_{2}})=0$. It follows that $b_{3}c_{4}+a_{3}c_{3}+d_{3}=0=b_{4}c_{4}+a_{4}c_{3}+d_{4}$, contradicting the observation we made above. So, \eqref{eq:1standaardB15bis} has either 0 solutions or a unique solution in $\mu_{4}$, and the latter occurs iff
\begin{multline}\label{eq:1cubic15}
F(\nu_{1},\nu_{2})=L_{1}(\nu_{1},\nu_{2})C_{2}(\nu_{1},\nu_{2})-L_{2}(\nu_{1},\nu_{2})C_{1}(\nu_{1},\nu_{2})=0\\ \wedge\quad\left(L_{1}(\nu_{1},\nu_{2}),L_{2}(\nu_{1},\nu_{2})\right)\neq(0,0)\quad\wedge\quad\left(C_{1}(\nu_{1},\nu_{2}),C_{2}(\nu_{1},\nu_{2})\right)\neq(0,0)\;.
\end{multline}
Recall for this last condition that $\mu_{4}\in\F^{*}_{q}$. The equation $F(\nu_{1},\nu_{2})=0$ determines a conic $C$ in the $(\nu_{1},\nu_{2})$-plane $\pi\cong\AG(2,q)$. We embed this affine plane in the projective plane $\overline{\pi}\cong\PG(2,q)$ by adding the line at infinity $\ell_{\infty}$ and extend $C$ to the conic $\overline{C}$ by going to a homogeneous equation $\overline{F}(\nu_{1},\nu_{2},\rho)=0$. Analogously, we define the homogeneous functions $\overline{L_{1}}(\nu_{1},\nu_{2},\rho)$, $\overline{L_{2}}(\nu_{1},\nu_{2},\rho)$, $\overline{C_{1}}(\nu_{1},\nu_{2},\rho)$ and $\overline{C_{2}}(\nu_{1},\nu_{2},\rho)$.
\par Note that both $\overline{L_{1}}$ and $\overline{L_{2}}$ cannot be identically zero. Moreover, $\overline{L_{1}}=0$ and $\overline{L_{2}}=0$ determine different lines in $\pi$ and their intersection point $R=\langle(c_{4},-c_{3},1)\rangle$ is not on $\ell_{\infty}$. It is clear that $R$ is on the conic $\overline{C}$. Furthermore, there is precisely one point $R'$ in $\overline{\pi}$ that is on $\overline{C_{1}}=0$ and $\overline{C_{2}}=0$ since $\rk\left(\begin{smallmatrix} b_{3} & -a_{3} & d_{3}\\b_{4} & -a_{4} & d_{4}\end{smallmatrix}\right)=\rk\left(\begin{smallmatrix} a_{3} & b_{3} & d_{3}\\a_{4} & b_{4} & d_{4}\end{smallmatrix}\right)=2$. In other words, $\overline{C_{1}}=0$ and $\overline{C_{2}}=0$ determine non-coinciding lines. Note that $R'\in\ell_{\infty}$ if and only if $a_{4}b_{3}-a_{3}b_{4}=0$. We set $\varepsilon=1$ if $R'$ is affine, and $\varepsilon=0$ if $R'\in\ell_{\infty}$. Furthermore $R\neq R'$ since we showed above that $L_{1}$, $L_{2}$, $C_{1}$ and $C_{2}$ cannot be simultaneously zero.
\par Also, it is impossible that simultaneously the lines $\overline{L_{1}}=0$ and $\overline{C_{1}}=0$ coincide and the lines $\overline{L_{2}}=0$ and $\overline{C_{2}}=0$ coincide in $\overline{\pi}$, since then we would have that $b_{3}=a_{4}=0$ and $a_{3}c_{3}+d_{3}=b_{4}c_{4}+d_{4}=0$, which contradicts \eqref{niet00}. We conclude that $\overline{L_{1}}=0$ and $\overline{C_{1}}=0$ intersect in a point $R_{1}$, or $\overline{L_{2}}=0$ and $\overline{C_{2}}=0$ intersect in a point $R_{2}$. Without loss of generality, we can assume that $R_{1}$ exists; we see that $R_{1}$ is on $\overline{C}$. Since the lines $\overline{L_{1}}=0$ and $\overline{C_{1}}=0$ do not coincide, it follows that the lines $\overline{L_{1}}=0$ and $\overline{C_{1}}=0$ contain at most one point different from $R_{1}$ on $\overline{C}$, the points $R$ and $R'$, respectively.
\par Now note that, if there are two points of $C=\overline{C}\cap\pi$ different from $R_{1}$ on the same line through $R_{1}$ (different from $\overline{L_{1}}=0$ and $\overline{C_{1}}=0$), then these points correspond to the same solution $\overline{\mu}$ for $\mu_{4}$ in \eqref{eq:1standaardB15bis}; hence looking at \eqref{eq:1algbis} with$(\mu_{3},\nu_{3},\nu_{4},\mu_{5},\nu_{5})=(0,1,0,1,0)$, we see that the corresponding rank 2 points of $\Pi$ both have $\varphi=\overline{\mu}\gamma'_{1}+\gamma'_{2}$, so these two points determine a $(q+1)$-secant by Theorem \ref{qplusonesecant}, contradicting the assumption on $\Pi$.
\par Assume that $q>2$ and that $\overline{C}$ decomposes in two lines over $\F_{q}$ (so in $\overline{\pi}$). One of these two lines, say $m$, contains $R_{1}$. Since $\overline{L_{1}}=0$ and $\overline{C_{1}}=0$ contain at most two points of $\overline{C}$, the line $m$ is different from $\overline{L_{1}}=0$ and $\overline{C_{1}}=0$. However, then the line $m$ (and hence $\overline{C}$) contains $q-1\geq2$ affine points, which are obviously on the same line through $R_{1}$, a contradiction. So, if $q>2$, the conic $\overline{C}$ cannot decompose in two lines over $\F_{q}$. The conic $\overline{C}$ also cannot decompose in two lines over a quadratic extension, since $\overline{C}$ contains at least two different points $R$ and $R'$. Hence $\overline{C}$ is a non-degenerate conic if $q>2$ and it contains $q+1$ points, of which $q-1$, $q$ or $q+1$ are affine (on $C$). So, Equation \eqref{eq:1cubic15} and hence also Equation \eqref{eq:1standaardB15} has $q-2-\varepsilon$, $q-1-\varepsilon$ or $q-\varepsilon$ solutions since we must disregard the solutions corresponding to $R$ and $R'$. If $q=2$ the conic $\overline{C}$ contains at least two points, $R$ and $R'$, and hence it contains $q+1=3$ or $2q+1=5$ points, of which $q-1=1$, $q=2$, $q+1=2q-1=3$ or $2q=4$ are affine. So, Equation \eqref{eq:1standaardB15} has $0$, $1-\varepsilon$, $2-\varepsilon$ or $3-\varepsilon$ solutions.
\par Now we look at Equations \eqref{eq:1standaardC} and \eqref{eq:1standaardD}. It is immediately clear that Equation \eqref{eq:1standaardC} has precisely one solution by the assumption of Case A.5. Equation \eqref{eq:1standaardD} has no solutions if and only if $\dim U_{2}=3$ and $\gamma_{2}\notin U_{2}$, so if and only if $a_{3}b_{4}-a_{4}b_{3}=0$; recall that $\rk\left(\begin{smallmatrix} a_{3} & b_{3} & d_{3}\\a_{4} & b_{4} & d_{4}\end{smallmatrix}\right)=2$. It has one solution otherwise. In other words, Equation \eqref{eq:1standaardD} has $\varepsilon$ solutions.
\par We find that Equations \eqref{eq:1standaardA}, \eqref{eq:1standaardB}, \eqref{eq:1standaardC} and \eqref{eq:1standaardD} in total have between $q-1$ and $q+1$ solutions, if $q>2$. Including the point $P_{0}$, we find that $|\Pi\cap\Omega_{2}|$ is contained in $\{q,q+1,q+2\}$. If $q=2$, we find in the same way that $|\Pi\cap\Omega_{2}|$ is contained in $\{1,\dots,5\}$. So, both for $q=2$ and $q>2$ we find that the theorem is true in Case A.5.

\subsection*{Case A.6}\label{apB:A6}
\textit{In Case A.6 we assume that $\gamma'_{2},\gamma_{0}\gamma'_{2},\delta,\gamma_{2},\delta\gamma'_{2}+\gamma'_{1}\gamma_{2}\in U_{1}$}. Note that in this case $U_{2}\leq U_{1}$. There are $a_{i},b_{i},c_{i},d_{i},e_{i}\in\F_{q}$, $i=1,\dots,4$, such that
\begin{align*}
	\gamma'_{2}&=a_{1}+a_{2}\gamma_{0}+a_{3}\gamma'_{1}+a_{4}\gamma_{0}\gamma'_{1}\;,\\
	\gamma_{0}\gamma'_{2}&=b_{1}+b_{2}\gamma_{0}+b_{3}\gamma'_{1}+b_{4}\gamma_{0}\gamma'_{1}\;,\\
	\delta&=c_{1}+c_{2}\gamma_{0}+c_{3}\gamma'_{1}+c_{4}\gamma_{0}\gamma'_{1}\;,\\
	\gamma_{2}&=d_{1}+d_{2}\gamma_{0}+d_{3}\gamma'_{1}+d_{4}\gamma_{0}\gamma'_{1}\quad\text{and}\\
	\delta\gamma'_{2}+\gamma'_{1}\gamma_{2}&=e_{1}+e_{2}\gamma_{0}+e_{3}\gamma'_{1}+e_{4}\gamma_{0}\gamma'_{1}\;.
\end{align*}
	Analogous to the deduction in the beginning of Case A.5, we find that $\rk\left(\begin{smallmatrix} a_{3} & b_{3} & d_{3}\\a_{4} & b_{4} & d_{4}\end{smallmatrix}\right)=2$. Note that we cannot have $a_{2}=a_{4}=0$ or $b_{1}=b_{3}=0$: in both cases we would have that $\{1,\gamma'_{1},\gamma'_{2}\}$ is not a linearly independent set over $\F_{q}$, contradicting a statement from the intermezzo.
	\par Considering now $\F_{q^{5}}$ as a vector space over $\F_{q}$, Equation \eqref{eq:1standaardA} is equivalent to the following system of equations:
	\begin{align}\label{eq:1standaardA16}
	&\begin{cases}
	e_{1}=\mu_{3}\nu_{2}-\mu_{2}\nu_{3}-\nu_{2}a_{1}+\nu_{1}b_{1}+\mu_{3}c_{1}+\nu_{3}d_{1}\\
	e_{2}=\mu_{1}\nu_{3}-\mu_{3}\nu_{1}-\nu_{2}a_{2}+\nu_{1}b_{2}+\mu_{3}c_{2}+\nu_{3}d_{2}\\
	e_{3}=\mu_{2}-\nu_{2}a_{3}+\nu_{1}b_{3}+\mu_{3}c_{3}+\nu_{3}d_{3}\\
	e_{4}=-\mu_{1}-\nu_{2}a_{4}+\nu_{1}b_{4}+\mu_{3}c_{4}+\nu_{3}d_{4}
	\end{cases}\nonumber\\
	\Leftrightarrow\quad
	&\begin{cases}
	\mu_{1}=-\nu_{2}a_{4}+\nu_{1}b_{4}+\mu_{3}c_{4}+\nu_{3}d_{4}-e_{4}\\
	\mu_{2}=\nu_{2}a_{3}-\nu_{1}b_{3}-\mu_{3}c_{3}-\nu_{3}d_{3}+e_{3}\\
	e_{1}=\mu_{3}\nu_{2}-\mu_{2}\nu_{3}-\nu_{2}a_{1}+\nu_{1}b_{1}+\mu_{3}c_{1}+\nu_{3}d_{1}\\
	e_{2}=\mu_{1}\nu_{3}-\mu_{3}\nu_{1}-\nu_{2}a_{2}+\nu_{1}b_{2}+\mu_{3}c_{2}+\nu_{3}d_{2}
	\end{cases}\;.
	\end{align}
	It is straightforward to see that there is a one-to-one correspondence between the solutions in $(\mu_{1},\mu_{2},\mu_{3},\nu_{1},\nu_{2},\nu_{3})$ of Equation \eqref{eq:1standaardA16} and the solutions in $(\nu_{1},\nu_{2},\mu_{3},\nu_{3})$ of
	\begin{align}\label{eq:1standaardA16bis}
	&\begin{cases}
	e_{1}=\mu_{3}\nu_{2}-(\nu_{2}a_{3}-\nu_{1}b_{3}-\mu_{3}c_{3}-\nu_{3}d_{3}+e_{3})\nu_{3}-\nu_{2}a_{1}+\nu_{1}b_{1}+\mu_{3}c_{1}+\nu_{3}d_{1}\\
	e_{2}=(-\nu_{2}a_{4}+\nu_{1}b_{4}+\mu_{3}c_{4}+\nu_{3}d_{4}-e_{4})\nu_{3}-\mu_{3}\nu_{1}-\nu_{2}a_{2}+\nu_{1}b_{2}+\mu_{3}c_{2}+\nu_{3}d_{2}
	\end{cases}\nonumber\\
	\Leftrightarrow\quad
	&\begin{cases}
	-L_{11}(\mu_{3},\nu_{3})\nu_{1}+L_{12}(\mu_{3},\nu_{3})\nu_{2}=C_{1}(\mu_{3},\nu_{3})\\
	-L_{21}(\mu_{3},\nu_{3})\nu_{1}+L_{22}(\mu_{3},\nu_{3})\nu_{2}=C_{2}(\mu_{3},\nu_{3})
	\end{cases}
	\end{align}
	with
	\begin{align*}
	L_{11}(\mu_{3},\nu_{3})&=b_{3}\nu_{3}+b_{1}\;,\\
	L_{12}(\mu_{3},\nu_{3})&=-\mu_{3}+a_{3}\nu_{3}+a_{1}\;,\\
	L_{21}(\mu_{3},\nu_{3})&=-\mu_{3}+b_{4}\nu_{3}+b_{2}\;,\\
	L_{22}(\mu_{3},\nu_{3})&=a_{4}\nu_{3}+a_{2}\;,\\
	C_{1}(\mu_{3},\nu_{3})&=c_{3}\mu_{3}\nu_{3}+d_{3}\nu^{2}_{3}+c_{1}\mu_{3}+(d_{1}-e_{3})\nu_{3}-e_{1}\;\text{ and}\\
	C_{2}(\mu_{3},\nu_{3})&=c_{4}\mu_{3}\nu_{3}+d_{4}\nu^{2}_{3}+c_{2}\mu_{3}+(d_{2}-e_{4})\nu_{3}-e_{2}\;.
	\end{align*}
	Given $\mu_{3}$ and $\nu_{3}$, the system of equations in \eqref{eq:1standaardA16bis} has $0$, $1$, $q$ or $q^{2}$ solutions for $(\nu_{1},\nu_{2})$. Assume that for $(\mu_{3},\nu_{3})=(\overline{\mu},\overline{\nu})$ the system of equations in \eqref{eq:1standaardA16bis} would have $q$ or $q^{2}$ solutions. Then, looking at \eqref{eq:1algbis} with $(\mu_4,\mu_5,\nu_4,\nu_5)=(0,1,1,0)$,  for the $q$ or $q^{2}$ corresponding points, we have $\varphi=\frac{\overline{\mu}-\gamma'_{2}}{\gamma'_{1}-\overline{\nu}}$, so any two of these $q$ points determine a $(q+1)$-secant by Theorem \ref{qplusonesecant}, contradicting the assumption on $\Pi$. So, the system of equations in \eqref{eq:1standaardA16bis} has $0$ solutions or a unique solution in $(\nu_{1},\nu_{2})$, and the former occurs iff
	\begin{align}\label{eq:1cubic16}
	F(\mu_{3},\nu_{3})=L_{11}(\mu_{3},\nu_{3})L_{22}(\mu_{3},\nu_{3})-L_{12}(\mu_{3},\nu_{3})L_{21}(\mu_{3},\nu_{3})=0\;.
	\end{align}
	The equation $F(\nu_{1},\nu_{2})=0$ determines a conic $C$ in the $(\nu_{1},\nu_{2})$-plane $\pi\cong\AG(2,q)$. We embed this affine plane in the projective plane $\overline{\pi}\cong\PG(2,q)$ by adding the line at infinity $\ell_{\infty}$ and extend $C$ to the conic $\overline{C}$ by going to a homogeneous equation $\overline{F}(\nu_{1},\nu_{2},\rho)=0$. Analogously, we define the homogeneous functions $\overline{L_{11}}(\nu_{1},\nu_{2},\rho)$, $\overline{L_{12}}(\nu_{1},\nu_{2},\rho)$, $\overline{L_{21}}(\nu_{1},\nu_{2},\rho)$, $\overline{L_{22}}(\nu_{1},\nu_{2},\rho)$, $\overline{C_{1}}(\nu_{1},\nu_{2},\rho)$ and $\overline{C_{2}}(\nu_{1},\nu_{2},\rho)$.
	\par We denote the number of points on $\overline{C}$ by $N$ and the number of points on $\overline{C}\cap\ell_{\infty}$ by $N_{\infty}$. We find that Equation \eqref{eq:1cubic16} has $N-N_{\infty}$ solutions, and hence Equation \eqref{eq:1standaardA16bis}, has $q^{2}-N+N_{\infty}$ solutions.
	\par Now, we look at Equation \eqref{eq:1standaardB}; it is equivalent to the following system of equations:
	\begin{align}\label{eq:1standaardB16}
	\begin{cases}
	-d_{1}=-\mu_{2}-a_{1}\nu_{2}+b_{1}\nu_{1}-c_{1}\mu_{4}\\
	-d_{2}=\mu_{1}-a_{2}\nu_{2}+b_{2}\nu_{1}-c_{2}\mu_{4}\\
	-d_{3}=-\mu_{4}\nu_{2}-a_{3}\nu_{2}+b_{3}\nu_{1}-c_{3}\mu_{4}\\
	-d_{4}=\mu_{4}\nu_{1}-a_{4}\nu_{2}+b_{4}\nu_{1}-c_{4}\mu_{4}\\
	\end{cases}
	\Leftrightarrow\quad
	\begin{cases}
	\mu_{2}=-a_{1}\nu_{2}+b_{1}\nu_{1}-c_{1}\mu_{4}+d_{1}\\
	\mu_{1}=a_{2}\nu_{2}-b_{2}\nu_{1}+c_{2}\mu_{4}+d_{2}\\
	-d_{3}=-\mu_{4}\nu_{2}-a_{3}\nu_{2}+b_{3}\nu_{1}-c_{3}\mu_{4}\\
	-d_{4}=\mu_{4}\nu_{1}-a_{4}\nu_{2}+b_{4}\nu_{1}-c_{4}\mu_{4}\\
	\end{cases}\;.
	\end{align}
	It is straightforward to see that there is a one-to-one correspondence between the solutions in $(\mu_{1},\mu_{2},\mu_{4},\nu_{1},\nu_{2})$ of Equation \eqref{eq:1standaardB13} and the solutions in $(\mu_{4},\nu_{1},\nu_{2})$ of
	\begin{align}\label{eq:1standaardB16bis}
	&\begin{cases}
	-d_{3}=-\mu_{4}\nu_{2}-a_{3}\nu_{2}+b_{3}\nu_{1}-c_{3}\mu_{4}\\
	-d_{4}=\mu_{4}\nu_{1}-a_{4}\nu_{2}+b_{4}\nu_{1}-c_{4}\mu_{4}
	\end{cases}\nonumber\\
	\Leftrightarrow\quad
	&\begin{cases}
	-b_{3}\nu_{1}+(\mu_{4}+a_{3})\nu_{2}=-c_{3}\mu_{4}+d_{3}\\
	-(\mu_{4}+b_{4})\nu_{1}+a_{4}\nu_{2}=-c_{4}\mu_{4}+d_{4}
	\end{cases}\nonumber\\
	\Leftrightarrow\quad
	&\begin{cases}
	-\overline{L_{11}}\left(-\mu_{4},1,0\right)\nu_{1}+\overline{L_{12}}\left(-\mu_{4},1,0\right)\nu_{2}=\overline{C_{1}}\left(-\mu_{4},1,0\right)\\
	-\overline{L_{12}}\left(-\mu_{4},1,0\right)\nu_{1}+\overline{L_{22}}\left(-\mu_{4},1,0\right)\nu_{2}=\overline{C_{2}}\left(-\mu_{4},1,0\right)
	\end{cases}.
	\end{align}
	The system of equations in \eqref{eq:1standaardB16bis} has $0$, $1$, $q$ or $q^{2}$ solutions for $(\nu_{1},\nu_{2})$. Assume that for $\mu_{4}=\overline{\mu}$ the system of equations in \eqref{eq:1standaardB16bis} would have $q$ or $q^{2}$ solutions. Then, looking at \eqref{eq:1algbis} with $(\mu_{3},\nu_{3},\nu_{4},\mu_{5},\nu_{5})=(0,1,0,1,0)$, we see that for the $q$ or $q^{2}$ corresponding points, we have $\varphi=\overline{\mu}\gamma'_{1}+\gamma'_{2}$, so any two of these $q$ points determine a $(q+1)$-secant by Theorem \ref{qplusonesecant}, contradicting the assumption on $\Pi$. So, \eqref{eq:1standaardB16bis} has either 0 solutions or a unique solution in $(\nu_{1},\nu_{2})$, and the former occurs iff
	\begin{align}\label{eq:1oneindig16}
	0&=\overline{L_{11}}\left(-\mu_{4},1,0\right)\overline{L_{22}}\left(-\mu_{4},1,0\right)-\overline{L_{12}}\left(-\mu_{4},1,0\right)\overline{L_{21}}\left(-\mu_{4},1,0\right)\nonumber\\
	&=\overline{F}(-\mu_{4},1,0)\nonumber\\
	&=-\mu^{2}_{4}+(a_{3}+b_{4})\mu_{4}-(a_{3}b_{4}-a_{4}b_{3})\;.
	\end{align}
	Recall that $\mu_{4}\in\F^{*}_{q}$. Note that the point $(1,0,0)\notin\overline{C}$ and that $(0,1,0)\in\overline{C}\Leftrightarrow a_{3}b_{4}-a_{4}b_{3}=0$. We set $\varepsilon=0$ if $a_{3}b_{4}-a_{4}b_{3}=0$ and $\varepsilon=1$ otherwise. So, Equation \eqref{eq:1oneindig16} has $(q-1)-(N_{\infty}-1+\varepsilon)$ solutions.
	\par Now we look at Equations \eqref{eq:1standaardC} and \eqref{eq:1standaardD}. It is immediately clear that Equation \eqref{eq:1standaardC} has precisely one solution by the assumption of Case A.6. Equation \eqref{eq:1standaardD} has no solutions if and only if $\dim U_{2}=3$ and $\gamma_{2}\notin U_{2}$, so if and only if $a_{3}b_{4}-a_{4}b_{3}=0$; recall that $\rk\left(\begin{smallmatrix} a_{3} & b_{3} & d_{3}\\a_{4} & b_{4} & d_{4}\end{smallmatrix}\right)=2$. It has one solution otherwise. In other words, Equation \eqref{eq:1standaardD} has $\varepsilon$ solutions.
	\par We find that the Equations \eqref{eq:1standaardA}, \eqref{eq:1standaardB}, \eqref{eq:1standaardC} and \eqref{eq:1standaardD} in total have $(q^{2}-N+N_{\infty})+(q-N_{\infty}-\varepsilon)+1+\varepsilon=q^{2}+q+1-N$ solutions. Including the point $P_{0}$, we find that $|\Pi\cap\Omega_{2}|$ equals $q^{2}+q+2-N$. We conclude this case by showing that $\overline{C}$ is a non-degenerate conic, hence that $N=q+1$ and consequently $|\Pi\cap\Omega_{2}|=q^{2}+1$.
	\par The conic $\overline{C}$ is given by the equation
	\begin{align*}
	0&=\overline{F}(\mu_{3},\nu_{3},\rho)\\&=\mu^{2}_{3}-(a_{3}+b_{4})\mu_{3}\nu_{3}+(a_{3}b_{4}-a_{4}b_{3})\nu^{2}_{3}-(a_{1}+b_{2})\mu_{3}\\&\qquad+(a_{1}b_{4}-a_{2}b_{3}+a_{3}b_{2}-a_{4}b_{1})\nu_{3}+(a_{1}b_{2}-a_{2}b_{1})\;.
	\end{align*}
	We distinguish between two cases of degeneracy.
	\begin{itemize}
		\item If $\overline{C}$ decomposes in two lines in $\overline{\pi}$ (over $\F_{q}$), then there exist $k,k'\in\F_{q}$ such that
		\[
		\overline{F}(\mu_{3},\nu_{3},\rho)=\left(\mu_{3}+k\nu_{3}+k'\right)\left(\mu_{3}-(a_{3}+b_{4}+k)\nu_{3}-(a_{1}+b_{2}+k')\right)
		\]
		with $k$ and $k'$ fulfilling
		\begin{align*}
		a_{3}b_{4}-a_{4}b_{3}&=-k(a_{3}+b_{4}+k)\;,\\
		a_{1}b_{2}-a_{2}b_{1}&=-k'(a_{1}+b_{2}+k')\;\text{ and}\\
		a_{1}b_{4}-a_{2}b_{3}+a_{3}b_{2}-a_{4}b_{1}&=-2kk'-k(a_{1}+b_{2})-k'(a_{3}+b_{4})\;,
		\end{align*}
		equivalently
		\begin{align*}
		(k+a_{3})(k+b_{4})&=a_{4}b_{3}\;,\\
		(k'+a_{1})(k'+b_{2})&=a_{2}b_{1}\;\text{ and}\\
		(a_{1}+k')(b_{4}+k)+(a_{3}+k)(b_{2}+k')&=a_{2}b_{3}+a_{4}b_{1}\;.
		\end{align*}
		For brevity of notation, we introduce $\overline{a}=a_{3}+k$, $\overline{b}=b_{4}+k$, $\hat{a}=a_{1}+k'$ and $\hat{b}=b_{2}+k'$. The previous equations can then be rewritten as
		\begin{align*}
		0&=a_{4}b_{3}-\overline{a}\overline{b}\;,\\
		0&=a_{2}b_{3}+a_{4}b_{1}-\overline{a}\hat{b}-\hat{a}\overline{b}\;\text{ and}\\
		0&=a_{2}b_{1}-\hat{a}\hat{b}\;.
		\end{align*}
		These equations in $a_{2},a_{4},b_{1},b_{3},\overline{a},\hat{a},\overline{b},\hat{b}$ are similar to the ones in Equations \eqref{eq:cubicontaard1}, \eqref{eq:cubicontaard2} and \eqref{eq:cubicontaard3}, with $(k,k')$ replacing $(c_{5},-d_{5})$. So, similarly we can derive a contradiction. Hence, $\overline{C}$ does not decompose over $\F_{q}$.
		\item Now we assume that $\overline{C}$ decomposes in two lines not in $\overline{\pi}$, so over a quadratic extension $\F_{q^{2}}$ of $\F_{q}$. Then $\overline{C}$ contains only one point in $\overline{\pi}$. Let $\ell_{i,j}$ be the line with equation $L_{ij}=0$. It is clear that the points $R_{1}=\ell_{1,1}\cap\ell_{1,2}$, $R_{2}=\ell_{1,1}\cap\ell_{2,1}$, $R_{3}=\ell_{2,2}\cap\ell_{1,2}$ and $R_{4}=\ell_{2,2}\cap\ell_{2,1}$ are all on $\overline{C}$ -- note that these points are always well-defined since $\ell_{i,i}$ and $\ell_{j,3-j}$ cannot coincide for any choice of  $i,j\in\{1,2\}$. As $\overline{C}$ contains only one point in $\overline{\pi}$, we must have that $R_{1}=R_{2}=R_{3}=R_{4}$, but $\langle(1,0,0)\rangle$ is a point that is contained in $\ell_{1,1}$ and $\ell_{2,2}$ but surely not contained in $\ell_{1,2}$ and $\ell_{2,1}$. Hence, we must have that $\ell_{1,1}$ and $\ell_{2,2}$ coincide. So $\overline{C}$ is of the form $tL^{2}_{11}-L_{12}L_{21}=0$. Note that $L_{12}$ and $L_{21}$ are not indentically zero. So $\overline{C}$ can only be degenerate if $L_{12}=sL_{21}$ for some $s\in \F_q^*$. This implies that the lines $\ell_{1,2}$ and $\ell_{2,1}$ coincide.
		\par Since $\ell_{1,1}$ and $\ell_{2,2}$ coincide, and also $\ell_{1,2}$ and $\ell_{2,1}$ coincide, we have $a_{1}=b_{2}$ and $a_{3}=b_{4}$, and $k(a_{2},a_{4})=(b_{1},b_{3})$ for some $k\in\F^{*}_{q}$. Recall that $(a_{2},a_{4})\neq(0,0)\neq(b_{1},b_{3})$. So, the conic $\overline{C}$ is given by
		\[
		0=\overline{F}(\mu_{3},\nu_{3},\rho)=k(a_{4}\nu_{3}+a_{2})^{2}-\left(-\mu_{3}+a_{3}\nu_{3}+a_{1}\right)^{2}\;.
		\]
		As $\overline{C}$ contains only one point in $\overline{\pi}$ we have that $k$ is a non-square (and necessarily that $q$ is odd), and since $\overline{C}$ decomposes over $\F_{q^{2}}$, there is a $k'\in\F_{q^{2}}\setminus\F_{q}$ such that $k'^{2}=k$. We find that
		\begin{align}\label{eq:conicontaard}
		(\gamma_{0}-k')(\gamma'_{2}+(k'a_{4}-a_{3})\gamma'_{1}+k'a_{2}-a_{1})&=ka_{2}+a_{1}\gamma_{0}+ka_{4}\gamma'_{1}+a_{3}\gamma_{0}\gamma'_{1}-k'a_{1}\nonumber\\&\qquad-k'a_{2}\gamma_{0}-k'a_{3}\gamma'_{1}-k'a_{4}\gamma_{0}\gamma'_{1}\nonumber\\&\qquad+(k'a_{4}-a_{3})\gamma_{0}\gamma'_{1}+(k'a_{2}-a_{1})\gamma_{0}\nonumber\\&\qquad-k'(k'a_{4}-a_{3})\gamma'_{1}-k'(k'a_{2}-a_{1})\nonumber\\
		&=0\;.
		\end{align}
		It is important to note that $\F_{q^{5}}\cap\F_{q^{2}}=\F_{q}$. Hence, the first factor in \eqref{eq:conicontaard} cannot be zero since $k'\notin\F_{q}$. So, the second factor has to be zero, but then $\gamma'_{2}-a_{3}\gamma'_{1}-a_{1}=0$, a contradiction since $\{1,\gamma'_{1},\gamma'_{2}\}$ is a linearly independent set over $\F_{q}$.
	\end{itemize}
	So, in both cases we have found a contradiction, leading to the conclusion that indeed $N=q+1$ and $|\Pi\cap\Omega_{2}|=q^{2}+1$.
	
\subsection*{Case B.1.2}\label{apB:B1.2}
\textit{In Case B.1.2 we assume that $\gamma'_{1},\gamma'_{2}\notin\left\langle 1,\gamma_{0}\right\rangle$ and that $\delta\in U$, but $\gamma_{2}\notin U$.} Hence, $\left\{1,\gamma_{0},\gamma'_{1},\gamma'_{2},\gamma_{2}\right\}$ is an $\F_{q}$-basis for $\F_{q^{5}}$. There are $c_{i},e_{i}\in\F_{q}$, $i=1,\dots,5$, such that 
\begin{align*}
\delta&=c_1+c_2\gamma_0+c_3\gamma_1'+c_4\gamma_2'\quad\text{and}\\
\delta\gamma'_{2}+\gamma'_{1}\gamma_{2}&=e_1+e_2\gamma_0+e_3\gamma_1'+e_4\gamma_2'+e_5\gamma_{2}\;.
\end{align*}
\par Considering  $\F_{q^{5}}$ as a vector space over $\F_{q}$, Equation \eqref{eq:1standaardA} is equivalent to the following system of equations:
\begin{align}\label{eq:1standaardA212}
\begin{cases}
e_1-c_{1}\mu_{3}=-a_1\mu_1-\nu_3\mu_2+b_1\nu_1+\mu_3\nu_2\\
e_2-c_{2}\mu_{3}=(\nu_3-a_2)\mu_1+(b_2-\mu_3)\nu_1\\
e_3-c_{3}\mu_{3}=-a_3\mu_1+\mu_2\\
e_4-c_{4}\mu_{3}=b_4\nu_1-\nu_2\\
e_5=\nu_3
\end{cases}.
\end{align}
It is straightforward to see that there is a one-to-one correspondence between the solutions in $(\mu_{1},\mu_{2},\mu_{3},\nu_{1},\nu_{2},\nu_{3})$ of Equation \eqref{eq:1standaardA212} and the solutions in $(\mu_{1},\mu_{2},\mu_{3},\nu_{1},\nu_{2})$ of
\begin{align}\label{eq:1standaardA212bis}
\begin{cases}
e_1-c_{1}\mu_{3}=-a_1\mu_1-e_5\mu_2+b_1\nu_1+\mu_3\nu_2\\
e_2-c_{2}\mu_{3}=(e_5-a_2)\mu_1+(b_2-\mu_3)\nu_1\\
e_3-c_{3}\mu_{3}=-a_3\mu_1+\mu_2\\
e_4-c_{4}\mu_{3}=b_4\nu_1-\nu_2
\end{cases}.
\end{align}
Given $\mu_{3}$, the system of equations in \eqref{eq:1standaardA212bis} has $0$, $1$ or at least $q$ solutions for $(\mu_{1},\mu_{2},\nu_{1},\nu_{2})$. Assume that for $\mu_{3}=\overline{\mu}$ the system of equations in \eqref{eq:1standaardA212bis} would have at least $q$ solutions. Then, looking at \eqref{eq:1algbis} with $(\mu_4,\mu_5,\nu_4,\nu_5)=(0,1,1,0)$, we see that for the corresponding points, we have $\varphi=\frac{\overline{\mu}-\gamma'_{2}}{\gamma'_{1}-e_{5}}$, so any two of these at least $q$ points determine a $(q+1)$-secant by Theorem \ref{qplusonesecant}, contradicting the assumption on $\Pi$. So, the system of equations in \eqref{eq:1standaardA212bis} has $0$ solutions or a unique solution in $(\mu_{1},\mu_{2},\nu_{1},\nu_{2})$. The coefficient matrix of the system of equations in \eqref{eq:1standaardA212bis} is
\[
A_{12}=
\begin{pmatrix}
-a_1&-e_{5}&b_1&\mu_3\\
e_{5}-a_2&0&b_2-\mu_3&0\\
-a_3&1&0&0\\
0&0&b_4&-1
\end{pmatrix}\,
\]
with
\begin{align*}
\det\left(A_{12}\right)&=(a_{1}+a_{2}b_{4}+(a_{3}-b_{4})e_{5})\mu_{3}+ (a_{2}b_{1}-b_{2}a_{1}-(b_{1}+a_{3}b_{2})e_{5})\nonumber\\
&=D(\mu_3)\;.
\end{align*}
For a given $\mu_{3}$ the system of equations in \eqref{eq:1standaardA212bis} has a unique solution if $D(\mu_{3})\neq0$ and no solutions otherwise. We show that $D(\mu_{3})=0$ is a non-vanishing linear equation. Assume that it does vanish. Then, we have that $a_{1}+a_{2}b_{4}+(a_{3}-b_{4})e_{5}=0=a_{2}b_{1}-b_{2}a_{1}-(b_{1}+a_{3}b_{2})e_{5}$, and consequently 	\[
(a_{3}-b_{4})(a_{2}b_{1}-a_{1}b_{2})=-(a_{1}+a_{2}b_{4})(b_{1}+a_{3}b_{2})\quad\Leftrightarrow\quad(a_{1}+a_{2}a_{3})(b_{1}+b_{2}b_{4})=0\;,
\]
contradicting the statements in the beginning of Case B.1. So, indeed $D(\nu_{3})=0$ is a non-vanishing linear equation. Consequently, it has at most one solution, and thus the system of equations in \eqref{eq:1standaardA212bis} has $q-1$ or $q$ solutions.
\par Now, we look at Equation \eqref{eq:1standaardB}. However, since $\gamma_{2}\notin U$ by the assumption of this case, it is clear that Equation \eqref{eq:1standaardB} has no solutions. \par We find that the Equations \eqref{eq:1standaardA}, \eqref{eq:1standaardB}, \eqref{eq:1standaardC} and \eqref{eq:1standaardD} in total have $q-1$ or $q$ solutions. Including the point $P_{0}$, we find that $|\Pi\cap\Omega_{2}|$ equals $q$ or $q+1$.

\subsection*{Case B.1.3}\label{apB:B1.3}
\textit{In Case B.1.3 we assume that $\gamma'_{1},\gamma'_{2}\notin\left\langle 1,\gamma_{0}\right\rangle$ and that $\delta,\gamma_{2}\in U$.} There are $c_{i},d_{i}\in\F_{q}$, $i=1,\dots,4$, such that 
	\begin{align*}
	-\gamma_2&=c_1+c_2\gamma_0+c_3\gamma_1'+c_4\gamma_2'\quad\text{and}\\
	\delta&=d_1+d_2\gamma_0+d_3\gamma_1'+d_4\gamma_2'\;.
	\end{align*}
	\par First we look at Equation \eqref{eq:1standaardA}. If $\delta\gamma_2'+\gamma_1'\gamma_2\notin U$, then there are no solutions to this equation. So, we assume that $\delta\gamma_2'+\gamma_1'\gamma_2\in U$. Then, there are $e_{i}\in\F_{q}$, $i=1,\dots,4$, such that
	\[
	\delta\gamma_2'+\gamma_1'\gamma_2=e_1+e_2\gamma_0+e_3\gamma_1'+e_4\gamma_2'\;.
	\]
	Considering $\F_{q^{5}}$ as a vector space over $\F_{q}$, Equation \eqref{eq:1standaardA} is equivalent to the following system of equations:
	\begin{align}\label{eq:1standaardA213}
	\begin{cases}
	e_1-d_1\mu_3+c_1\nu_3=-a_1\mu_1-\nu_3\mu_2+b_1\nu_1+\mu_3\nu_2\\
	e_2-d_2\mu_3+c_2\nu_3=(\nu_3-a_2)\mu_1+(b_2-\mu_3)\nu_1\\
	e_3-d_3\mu_3+c_3\nu_3=-a_3\mu_1+\mu_2\\
	e_4-d_4\mu_3+c_4\nu_3=b_4\nu_1-\nu_2
	\end{cases}.
	\end{align}
	Given $\mu_{3}$ and $\nu_{3}$, the system of equations in \eqref{eq:1standaardA213} has $0$, $1$ or at least $q$ solutions for $(\mu_{1},\mu_{2},\nu_{1},\nu_{2})$. Assume that for $(\mu_{3},\nu_{3})=(\overline{\mu},\overline{\nu})$ the system of equations in \eqref{eq:1standaardA213} would have at least $q$ solutions. Then, looking at \eqref{eq:1algbis} with $(\mu_4,\mu_5,\nu_4,\nu_5)=(0,1,1,0)$, we see that for the corresponding points, we have $\varphi=\frac{\overline{\mu}-\gamma'_{2}}{\gamma'_{1}-\overline{\nu}}$, so any two of these at least $q$ points determine a $(q+1)$-secant by Theorem \ref{qplusonesecant}, contradicting the assumption on $\Pi$. So, the system of equations in \eqref{eq:1standaardA213} has $0$ solutions or a unique solution in $(\mu_{1},\mu_{2},\nu_{1},\nu_{2})$. The coefficient matrix of the system of equations in \eqref{eq:1standaardA213} is
	\[
	A_{13}=
	\begin{pmatrix}
	-a_1&-\nu_3&b_1&\mu_3\\
	\nu_3-a_2&0&b_2-\mu_3&0\\
	-a_3&1&0&0\\
	0&0&b_4&-1
	\end{pmatrix}\,
	\]
	with
	\begin{align*}
	\det\left(A_{13}\right)&=(a_3-b_4)\mu_3\nu_3+(a_1+a_2b_4)\mu_3-(b_1+a_3b_2)\nu_3+a_2b_1-a_1b_2\nonumber\\
	&=D(\mu_{3},\nu_3)\;.
	\end{align*} 
	Given $\mu_{3}$ and $\nu_{3}$ the system of equations in \eqref{eq:1standaardA213} has a unique solution if $D(\mu_{3},\nu_{3})\neq0$ and no solutions otherwise. The  equation $D(\mu_3,\nu_3)=0$ represents a conic $C$ in the $(\mu_{3},\nu_{3})$-plane $\pi\cong\AG(2,q)$. Clearly, $C$ has two points on the line at infinity. One can check that the conic $C$ is singular if and only if $(a_3-b_4)(a_1+a_2a_3)(b_1+b_2b_4)=0$. Now we have seen before that $a_3\neq b_4$, that $a_1+a_2a_3\neq 0$ and that $b_1+b_2b_4\neq 0$. This implies that $C$ is non-singular, and so it has $q-1$ points in $\pi$.
	Consequently, the system of equations in \eqref{eq:1standaardA213} has $q^{2}-q+1$ solutions.
	\par Now, we look at Equation \eqref{eq:1standaardB}; it is equivalent to the following system of equations:
	\begin{align}\label{eq:1standaardB213}
	\begin{cases}
	c_1+d_1\mu_4=-\mu_2+\mu_4a_1\nu_1+\nu_1b_1\\
	c_2+d_2\mu_4=\mu_1+\mu_4a_2\nu_1+\nu_1b_2\\
	c_3+d_3\mu_4=\mu_4a_3\nu_1-\mu_4\nu_2\\
	c_4+d_4\mu_4=b_4\nu_1-\nu_2
	\end{cases}.
	\end{align}
	Recall that $\mu_{4}\in\F^{*}_{q}$. Given $\mu_{4}$, the coefficient matrix of this system of equations in $(\mu_{1},\mu_{2},\nu_{1},\nu_{2})$ is
	\[
	B_{13}=
	\begin{pmatrix}
	0&-1&\mu_4a_1+b_1&0\\
	1&0&\mu_4a_2+b_2&0\\
	0&0&\mu_4a_3&-\mu_4\\
	0&0&b_4&-1
	\end{pmatrix}\;.
	\]
	We find that $\det\left(B_{13}\right)=\mu_4(b_4-a_3)$. Since $a_3\neq b_4$ and $\mu_4\neq 0$, there is a unique solution in $\mu_1,\mu_2,\nu_1,\nu_2$ to the linear system in \eqref{eq:1standaardB213}. Hence, we find exactly $q-1$ solutions to \eqref{eq:1standaardB}. 
	\par We find that the Equations \eqref{eq:1standaardA}, \eqref{eq:1standaardB}, \eqref{eq:1standaardC} and \eqref{eq:1standaardD} in total have either $q-1$ or $q^{2}$ solutions, depending on whether $\delta\gamma_2'+\gamma_1'\gamma_2$ is contained in $U$ or not. Including the point $P_{0}$, we find that $|\Pi\cap\Omega_{2}|$ is either $q$ or $q^{2}+1$.
	
\subsection*{Case B.2.1}\label{apB:B2.1}

\textit{In Case B.2.1 we assume that $\gamma'_{2}\notin\left\langle 1,\gamma_{0}\right\rangle$ but $\gamma'_{1}\in\left\langle 1,\gamma_{0}\right\rangle$ and that $\delta\notin U$.} Hence, $\left\{1,\gamma_{0},\gamma'_{2},\gamma_{0}\gamma'_{1},\delta\right\}$ is an $\F_{q}$-basis for $\F_{q^{5}}$. There are $c_{i},e_{i}\in\F_{q}$, $i=1,\dots,5$, such that 
\begin{align*}
-\gamma_2&=c_1+c_2\gamma_0+c_3\gamma_{0}\gamma_1'+c_4\gamma_2'+c_5\delta\text{ and}\\
\delta\gamma'_{2}+\gamma'_{1}\gamma_{2}&=e_1+e_2\gamma_0+e_3\gamma_{0}\gamma_1'+e_4\gamma_2'+e_5\delta\;.
\end{align*}
\par Considering now $\F_{q^{5}}$ as a vector space over $\F_{q}$, Equation \eqref{eq:1standaardA} is equivalent to the following system of equations:
\begin{align}\label{eq:1standaardA221}
\begin{cases}
e_1+c_1\nu_3=(a_1-\nu_3)\mu_2+b_1\nu_1+\mu_3\nu_2\\
e_2+c_2\nu_3=\nu_3\mu_1+a_2\mu_2+(b_2-\mu_3)\nu_1\\
e_3+c_3\nu_3=-\mu_1\\
e_4+c_4\nu_3=b_4\nu_1-\nu_2\\
e_5+c_5\nu_3=\mu_3
\end{cases}.
\end{align}
It is straightforward to see that there is a one-to-one correspondence between the solutions in $(\mu_{1},\mu_{2},\mu_{3},\nu_{1},\nu_{2},\nu_{3})$ of Equation \eqref{eq:1standaardA221} and the solutions in $(\mu_{1},\mu_{2},\nu_{1},\nu_{2},\nu_{3})$ of
\begin{align}\label{eq:1standaardA221bis}
\begin{cases}
e_1+c_1\nu_3=(a_1-\nu_3)\mu_2+b_1\nu_1+(e_5+c_5\nu_3)\nu_2\\
e_2+c_2\nu_3=\nu_3\mu_1+a_2\mu_2+(b_2-e_5-c_5\nu_3)\nu_1\\
e_3+c_3\nu_3=-\mu_1\\
e_4+c_4\nu_3=b_4\nu_1-\nu_2
\end{cases}.
\end{align}
Given $\nu_{3}$, the system of equations in \eqref{eq:1standaardA221bis} has $0$, $1$ or at least $q$ solutions for $(\mu_{1},\mu_{2},\nu_{1},\nu_{2})$. Assume that for $\nu_{3}=\overline{\nu}$ the system of equations in \eqref{eq:1standaardA221bis} would have at least $q$ solutions. Then, looking at \eqref{eq:1algbis} with $(\mu_4,\mu_5,\nu_4,\nu_5)=(0,1,1,0)$, we see that for the corresponding points, we have $\varphi=\frac{e_{5}+c_{5}\overline{\nu}-\gamma'_{2}}{\gamma'_{1}-\overline{\nu}}$, so any two of these at least $q$ points determine a $(q+1)$-secant by Theorem \ref{qplusonesecant}, contradicting the assumption on $\Pi$. So, the system of equations in \eqref{eq:1standaardA221bis} has $0$ solutions or a unique solution in $(\mu_{1},\mu_{2},\nu_{1},\nu_{2})$. The coefficient matrix of the system of equations in \eqref{eq:1standaardA221bis} is
\[
A_{21}=
\begin{pmatrix}
0&a_1-\nu_3&b_1&e_{5}+c_{5}\nu_3\\
\nu_3&a_2&b_2-e_{5}-c_{5}\nu_3&0\\
-1&0&0&0\\
0&0&b_4&-1
\end{pmatrix}
\]
with
\begin{align*}
\det\left(A_{21}\right)&=c_{5}\nu^{2}_{3}-((a_{1}+a_{2}b_{4})c_{5}+b_{2}-e_{5})\nu_{3}+a_{1}b_{2}-a_{2}b_{1}-(a_{1}+a_{2}b_{4})e_{5}\\
&=D(\nu_3)\;.
\end{align*}
For a given $\nu_{3}$ the system of equations in \eqref{eq:1standaardA221bis} has a unique solution if $D(\nu_{3})\neq0$ and no solutions otherwise. We show that $D(\nu_{3})=0$ is a non-vanishing quadratic equation. Assume that it does vanish. Then, we have that $c_{5}=0$, that $e_{5}=b_{2}$ and that $a_{2}(b_{1}+b_{2}b_{4})=0$, contradicting the statements in the beginning of Case B.2. So, indeed $D(\nu_{3})=0$ is a non-vanishing quadratic equation. Consequently, it has at most two solutions, and thus the system of equations in \eqref{eq:1standaardA221bis} has $q-2$, $q-1$ or $q$ solutions. In particular, if $c_5=0$, this system has $q-1$ or $q$ solutions.
\par Now, we look at Equation \eqref{eq:1standaardB}; it is equivalent to the following system of equations:	
\begin{align}\label{eq:1standaardB221}
\begin{cases}
c_1=-\mu_2-a_1\mu_4\nu_2+b_1\nu_1\\
c_2=\mu_1-a_2\mu_4\nu_2+b_2\nu_1\\
c_3=\mu_4\nu_1\\
c_4=-\nu_2+b_4\nu_1\\
c_5=-\mu_4
\end{cases}.
\end{align}
Recall that $\mu_{4}\in\F^{*}_{q}$. So, the system of equations in \eqref{eq:1standaardB221} has no solutions if $c_{5}=0$. Hence, we assume in the discussion of this system of equations that $c_{5}\neq0$. We can see that there is a one-to-one correspondence between the solutions in $(\mu_{1},\mu_{2},\mu_{4},\nu_{1},\nu_{2})$ of Equation \eqref{eq:1standaardB211} and the solutions in $(\mu_{1},\mu_{2},\nu_{1},\nu_{2})$ of
\begin{align}\label{eq:1standaardB221bis}
\begin{cases}
c_1=-\mu_2+a_1c_5\nu_2+b_1\nu_1\\
c_2=\mu_1+a_2c_5\nu_2+b_2\nu_1\\
c_3=-c_5\nu_1\\
c_4=-\nu_2+b_4\nu_1
\end{cases}.
\end{align}
The coefficient matrix of this system is
\[
B_{21}=
\begin{pmatrix}
0&-1&b_1&c_5a_1\\
1&0&b_2&c_5a_2\\
0&0&-c_5&0\\
0&0&b_4&-1
\end{pmatrix}\;.
\]
Now $\det\left(B_{21}\right)=c_{5}$. As $c_5\neq 0$, we find a unique solution to the system of equations in \eqref{eq:1standaardB221bis}. So, Equation \eqref{eq:1standaardB221} has no solutions if $c_{5}=0$ and a unique solution if $c_{5}\neq0$.
\par We find that the Equations \eqref{eq:1standaardA}, \eqref{eq:1standaardB}, \eqref{eq:1standaardC} and \eqref{eq:1standaardD} in total have between $q-1$ and $q+1$ solutions. Including the point $P_{0}$, we find that $|\Pi\cap\Omega_{2}|$ is contained in $\{q,q+1,q+2\}$.
	
\subsection*{Case B.2.2}\label{apB:B2.2}

\textit{In Case B.2.2 we assume that $\gamma'_{2}\notin\left\langle 1,\gamma_{0}\right\rangle$ but $\gamma'_{1}\in\left\langle 1,\gamma_{0}\right\rangle$ and that $\delta\in U$, but $\gamma_{2}\notin U$.} Hence, $\left\{1,\gamma_{0},\gamma'_{2},\gamma_{0}\gamma'_{1},\gamma_{2}\right\}$ is an $\F_{q}$-basis for $\F_{q^{5}}$. There are $c_{i},e_{i}\in\F_{q}$, $i=1,\dots,5$, such that 
\begin{align*}
\delta&=c_1+c_2\gamma_0+c_3\gamma_{0}\gamma_1'+c_4\gamma_2'\quad\text{and}\\
\delta\gamma'_{2}+\gamma'_{1}\gamma_{2}&=e_1+e_2\gamma_0+e_3\gamma_{0}\gamma_1'+e_4\gamma_2'+e_5\gamma_{2}\;.
\end{align*}
\par Considering $\F_{q^{5}}$ as a vector space over $\F_{q}$, Equation \eqref{eq:1standaardA} is equivalent to the following system of equations:
\begin{align}\label{eq:1standaardA222}
\begin{cases}
e_1-c_{1}\mu_{3}=(a_{1}-\nu_3)\mu_2+b_1\nu_1+\mu_3\nu_2\\
e_2-c_{2}\mu_{3}=\nu_3\mu_1+a_2\mu_{2}+(b_2-\mu_3)\nu_1\\
e_3-c_{3}\mu_{3}=-\mu_1\\
e_4-c_{4}\mu_{3}=b_4\nu_1-\nu_2\\
e_5=\nu_3
\end{cases}.
\end{align}
It is straightforward to see that there is a one-to-one correspondence between the solutions in $(\mu_{1},\mu_{2},\mu_{3},\nu_{1},\nu_{2},\nu_{3})$ of Equation \eqref{eq:1standaardA222} and the solutions in $(\mu_{1},\mu_{2},\mu_{3},\nu_{1},\nu_{2})$ of
\begin{align}\label{eq:1standaardA222bis}
\begin{cases}
e_1-c_{1}\mu_{3}=(a_{1}-e_5)\mu_2+b_1\nu_1+\mu_3\nu_2\\
e_2-c_{2}\mu_{3}=e_5\mu_1+a_2\mu_{2}+(b_2-\mu_3)\nu_1\\
e_3-c_{3}\mu_{3}=-\mu_1\\
e_4-c_{4}\mu_{3}=b_4\nu_1-\nu_2
\end{cases}.
\end{align}
Given $\mu_{3}$, the system of equations in \eqref{eq:1standaardA222bis} has $0$, $1$ or at least $q$ solutions for $(\mu_{1},\mu_{2},\nu_{1},\nu_{2})$. Assume that for $\mu_{3}=\overline{\mu}$ the system of equations in \eqref{eq:1standaardA222bis} would have at least $q$ solutions. Then, looking at \eqref{eq:1algbis} with $(\mu_4,\mu_5,\nu_4,\nu_5)=(0,1,1,0)$, we see that for the corresponding points, we have $\varphi=\frac{\overline{\mu}-\gamma'_{2}}{\gamma'_{1}-e_{5}}$, so any two of these at least $q$ points determine a $(q+1)$-secant by Theorem \ref{qplusonesecant}, contradicting the assumption on $\Pi$. So, the system of equations in \eqref{eq:1standaardA222bis} has $0$ solutions or a unique solution in $(\mu_{1},\mu_{2},\nu_{1},\nu_{2})$. The coefficient matrix of the system of equations in \eqref{eq:1standaardA222bis} is
\[
A_{22}=
\begin{pmatrix}
0&a_1-e_{5}&b_1&\mu_3\\
e_{5}&a_2&b_2-\mu_3&0\\
-1&0&0&0\\
0&0&b_4&-1
\end{pmatrix}\,
\]
with
\begin{align*}
\det\left(A_{22}\right)&=-(a_{1}+a_{2}b_{4}-e_{5})\mu_{3}+ (a_{1}b_{2}-a_{2}b_{1}-b_{2}e_{5})\nonumber\\
&=D(\mu_3)\;.
\end{align*}
For a given $\mu_{3}$ the system of equations in \eqref{eq:1standaardA222bis} has a unique solution if $D(\mu_{3})\neq0$ and no solutions otherwise. We show that $D(\mu_{3})=0$ is a non-vanishing linear equation. Assume that it does vanish. Then, we have that $a_{1}+a_{2}b_{4}-e_{5}=0=a_{1}b_{2}-a_{2}b_{1}-b_{2}e_{5}$, and consequently
\[
0=a_{1}b_{2}-a_{2}b_{1}-b_{2}(a_{1}+a_{2}b_{4})\quad\Leftrightarrow\quad a_{2}(b_{1}+b_{2}b_{4})=0\;,
\]
contradicting the statements in the beginning of Case B.2. So, indeed $D(\nu_{3})=0$ is a non-vanishing linear equation. Consequently, it has at most one solution, and thus the system of equations in \eqref{eq:1standaardA222bis} has $q-1$ or $q$ solutions.
\par Now, we look at Equation \eqref{eq:1standaardB}. However, since $\gamma_{2}\notin U$ by the assumption of this case, it is clear that Equation \eqref{eq:1standaardB} has no solutions. \par We find that the Equations \eqref{eq:1standaardA}, \eqref{eq:1standaardB}, \eqref{eq:1standaardC} and \eqref{eq:1standaardD} in total have $q-1$ or $q$ solutions. Including the point $P_{0}$, we find that $|\Pi\cap\Omega_{2}|$ equals $q$ or $q+1$.

\subsection*{Case B.2.3}\label{apB:B2.3}

\textit{In Case B.2.3 we assume that $\gamma'_{2}\notin\left\langle 1,\gamma_{0}\right\rangle$ but $\gamma'_{1}\in\left\langle 1,\gamma_{0}\right\rangle$ and that $\delta,\gamma_{2}\in U$.} There are $c_{i},d_{i}\in\F_{q}$, $i=1,\dots,4$, such that 
\begin{align*}
-\gamma_2&=c_1+c_2\gamma_0+c_3\gamma_{0}\gamma_1'+c_4\gamma_2'\quad\text{and}\\
\delta&=d_1+d_2\gamma_0+d_3\gamma_{0}\gamma_1'+d_4\gamma_2'\;.
\end{align*}
\par First we look at Equation \eqref{eq:1standaardA}. If $\delta\gamma_2'+\gamma_1'\gamma_2\notin U$, then there are no solutions to this equation. So, we assume that $\delta\gamma_2'+\gamma_1'\gamma_2\in U$. Then, there are $e_{i}\in\F_{q}$, $i=1,\dots,4$, such that
\[
\delta\gamma_2'+\gamma_1'\gamma_2=e_1+e_2\gamma_0+e_3\gamma_{0}\gamma_1'+e_4\gamma_2'\;.
\]
Considering now $\F_{q^{5}}$ as a vector space over $\F_{q}$, Equation \eqref{eq:1standaardA} is equivalent to the following system of equations:
\begin{align}\label{eq:1standaardA223}
\begin{cases}
e_1-d_1\mu_3+c_1\nu_3=(a_1-\nu_3)\mu_2+b_1\nu_1+\mu_3\nu_2\\
e_2-d_2\mu_3+c_2\nu_3=\nu_3\mu_1+a_2\mu_{2}+(b_2-\mu_3)\nu_1\\
e_3-d_3\mu_3+c_3\nu_3=-\mu_1\\
e_4-d_4\mu_3+c_4\nu_3=b_4\nu_1-\nu_2
\end{cases}.
\end{align}
Given $\mu_{3}$ and $\nu_{3}$, the system of equations in \eqref{eq:1standaardA223} has $0$, $1$ or at least $q$ solutions for $(\mu_{1},\mu_{2},\nu_{1},\nu_{2})$. Assume that for $(\mu_{3},\nu_{3})=(\overline{\mu},\overline{\nu})$ the system of equations in \eqref{eq:1standaardA223} would have at least $q$ solutions. Then, looking at \eqref{eq:1algbis} with $(\mu_4,\mu_5,\nu_4,\nu_5)=(0,1,1,0)$, we see that for the corresponding points, we have $\varphi=\frac{\overline{\mu}-\gamma'_{2}}{\gamma'_{1}-\overline{\nu}}$, so any two of these at least $q$ points determine a $(q+1)$-secant by Theorem \ref{qplusonesecant}, contradicting the assumption on $\Pi$. So, the system of equations in \eqref{eq:1standaardA223} has $0$ solutions or a unique solution in $(\mu_{1},\mu_{2},\nu_{1},\nu_{2})$. The coefficient matrix of the system of equations in \eqref{eq:1standaardA213} is
\[
A_{23}=
\begin{pmatrix}
0&a_1-\nu_3&b_1&\mu_3\\
\nu_3&a_2&b_2-\mu_3&0\\
-1&0&0&0\\
0&0&b_4&-1
\end{pmatrix}\,
\]
with 
\begin{align*}
\det\left(A_{23}\right)&=\mu_3\nu_3-(a_1+a_2b_4)\mu_3-b_2\nu_3+a_1b_2-a_2b_1\nonumber\\
&=D(\mu_{3},\nu_3)\;.
\end{align*}
Given $\mu_{3}$ and $\nu_{3}$ the system of equations in \eqref{eq:1standaardA223} has a unique solution if $D(\mu_{3},\nu_{3})\neq0$ and no solutions otherwise. The  equation $D(\mu_3,\nu_3)=0$ represents a conic $C$ in the $(\mu_{3},\nu_{3})$-plane $\pi\cong\AG(2,q)$. Clearly, $C$ has two points on the line at infinity. One can check that the conic $C$ is singular if and only if $a_{2}(b_1+b_2b_4)=0$. Now we have seen before that $a_2\neq 0$ and that $b_1+b_2b_4\neq 0$. This implies that $C$ is non-singular, and so it has $q-1$ points in $\pi$. Consequently, the system of equations in \eqref{eq:1standaardA223} has $q^{2}-q+1$ solutions.
\par Now, we look at Equation \eqref{eq:1standaardB}; it is equivalent to the following system of equations:
\begin{align}\label{eq:1standaardB223}
\begin{cases}
c_1+d_1\mu_4&=-\mu_2+b_1\nu_1-a_1\mu_4\nu_2\\
c_2+d_2\mu_4&=\mu_1+b_2\nu_1-a_2\mu_4\nu_2\\
c_3+d_3\mu_4&=\mu_4\nu_1\\
c_4+d_4\mu_4&=b_4\nu_1-\nu_2
\end{cases}.
\end{align}
Recall that $\mu_{4}\in\F^{*}_{q}$. Given $\mu_{4}$, the coefficient matrix of this system of equations in $(\mu_{1},\mu_{2},\nu_{1},\nu_{2})$ is
\[
B_{23}=
\begin{pmatrix}
0&-1&b_1&-\mu_4a_1\\
1&0&b_2&-\mu_4a_2\\
0&0&\mu_4&0\\
0&0&b_4&-1
\end{pmatrix}\;.
\]
We find that $\det\left(B_{23}\right)=-\mu_4$. Since $\mu_4\neq 0$, there is a unique solution in $\mu_1,\mu_2,\nu_1,\nu_2$ to the linear system in \eqref{eq:1standaardB223}. Hence, we find exactly $q-1$ solutions to \eqref{eq:1standaardB}. 
\par We find that the Equations \eqref{eq:1standaardA}, \eqref{eq:1standaardB}, \eqref{eq:1standaardC} and \eqref{eq:1standaardD} in total have either $q-1$ or $q^{2}$ solutions, depending on whether $\delta\gamma_2'+\gamma_1'\gamma_2$ is contained in $U$ or not. Including the point $P_{0}$, we find that $|\Pi\cap\Omega_{2}|$ is either $q$ or $q^{2}+1$.

}{}

\end{document}